\documentclass[a4paper,11pt]{amsart}
\usepackage[left=2.54cm,right=2.54cm,top=3.6cm,bottom=4cm]{geometry}
\usepackage{url}
\usepackage{mathrsfs}
\usepackage{yfonts}
\usepackage{amssymb}
\usepackage{latexsym}
\usepackage{amsmath}
\usepackage{graphicx}%,psfrag,epsfig,psfig,pstricks,pdfpages
\usepackage[dvipsnames]{xcolor}
\usepackage{pst-plot}
\usepackage{mathrsfs}
\usepackage{hyperref}
\usepackage{mathtools}
\usepackage[all]{xy}
\usepackage{comment}

\usepackage{caption}
\usepackage{subcaption}

 \theoremstyle{plain}
 \newtheorem{thm}{Theorem}[section]
 \newtheorem{cor}[thm]{Corollary}
 \newtheorem{lem}[thm]{Lemma}
 
 \newtheorem{propo}[thm]{Proposition}
 \theoremstyle {definition}
 \newtheorem{defn}[thm]{Definition}
 \theoremstyle{remark}

\newcommand{\op}[1]{\operatorname{#1}}

\newcommand{\tif}{\text{if }}

\newcommand{\tand}{\text{ and }}

\newcommand{\tfor}{\text{for }}

\newcommand{\ga}{\alpha}
\newcommand{\gb}{\beta}

\newcommand{\gd}{\delta}
\newcommand{\gep}{\varepsilon}
\newcommand{\gf}{\varphi}
\newcommand{\gga}{\gamma}
\newcommand{\gh}{\eta}

\newcommand{\gl}{\lambda}
\newcommand{\gm}{\mu}

\newcommand{\go}{\omega}

\newcommand{\gs}{\sigma}

\newcommand{\gt}{\tau}

\newcommand{\gy}{\psi}

\newcommand{\gF}{\Phi}

\newcommand{\gO}{\Omega}

\newcommand{\gY}{\Psi}
 
\newcommand{\D}[1]{{\mathbb{#1}}} % math Black board
\newcommand{\refS}[1]{Section ~\ref{#1}} 

\newcommand{\refT}[1]{Theorem ~\ref{#1}}
\newcommand{\refL}[1]{Lemma ~\ref{#1}}
\newcommand{\refD}[1]{Definition ~\ref{#1}}
\newcommand{\refP}[1]{Proposition ~\ref{#1}}

\newcommand{\refE}[1]{Equation ~\eqref{#1}}

{\par \samepage}%
{\par}

\newcommand{\ol}{\overline}

\newcommand{\diam}{\operatorname{diam}\,}
\newcommand{\interior}{\operatorname{int}\,}

\newcommand{\eps}{\varepsilon}

\newcommand{\diams}{\op{diam}_{\hat{\D{C}}}}
\newcommand{\ds}{\op{d}_{\hat{\D{C}}}}

\begin{document}
\title{Hairy Cantor sets}
\author{Davoud Cheraghi}
\address{Department of Mathematics, Imperial College London, London, SW7 2AZ, UK}
\email{d.cheraghi@imperial.ac.uk}
\author{Mohammad Pedramfar}
\address{Department of Mathematics, Imperial College London, London, SW7 2AZ, UK}
\email{m.pedramfar15@imperial.ac.uk}
\date{\today}

\begin{abstract}
We introduce a topological object, called \textit{hairy Cantor set}, which in many ways enjoys the universal 
features of objects like Jordan curve, Cantor set, Cantor bouquet, hairy Jordan curve, etc. 
We give an axiomatic characterisation of hairy Cantor sets, and prove that any two such objects in the plane 
are ambiently homeomorphic.  

Hairy Cantor sets appear in the study of the dynamics of holomorphic maps with infinitely many renormalisation structures. 
They are employed to link the fundamental concepts of polynomial-like renormalisation by Douady-Hubbard 
with the arithmetic conditions obtained by Herman-Yoccoz in the study of the dynamics of analytic circle diffeomorphisms.
\end{abstract}

\maketitle

%%%%%%%%%%%%%%%%%%%%%%%%%%%%%%%%%%%%%%%%%%%%%%
\section{Introduction}
We introduce a topological object which enjoys a similar level of universal features as objects like 
Jordan curve, Cantor set, Cantor bouquet, hairy Jordan curve, Lelek fan, etc  
\cite{Jo1893,Moore1919,Lelek60,BuOv1990,Nadler92,HoOv2016}. \nocite{Veb1905}
These are topological objects uniquely determined, up to ambient homeomorphisms, by some simple axioms. 
These objects frequently appear in dynamical systems, in particular, as the Julia sets or the attractors of holomorphic 
maps on complex spaces.
The object presented here has a delicate fine-scale structure, for instance, it is not locally connected. 
However, due to the way these emerge in iterations of holomorphic mappings, they turn out ubiquitous.

We start by presenting simple examples of our favourite object. 

\begin{defn}\label{D:strainght-hairy-Cantor}
A set $X$ in the plane $\D{R}^2$ is called a \textbf{straight hairy Cantor set}, if there are a Cantor set of points 
$C\subset \D{R}$ and a function $l: C \to [0, +\infty)$ such that
\[ X = \{(x, y) \in \D{R}^2 \mid x \in C, 0 \leq y \leq l(x) \},\] 
and the function $l$ satisfies the following properties: 
\begin{itemize}
\item[(i)] the set of $x \in C$ with $l(x) > 0$ is dense in $C$; 
\item[(ii)] if $x \in C$ is an end point \footnote{$x \in C$ is called an end point, if there is $\gd>0$ such that 
either $(x, x+\gd)\cap C=\emptyset$ or $(x-\gd, x)\cap C=\emptyset$.}, then 
\[l(x)=0= \limsup_{t\to x, \, t\in C} l(t);\]
\item[(iii)] if $x\in C$ is not an end point, then 
\[\limsup_{t\to x^-,\, t\in C} l(t) = \limsup_{t\to x^+,\, t\in C} l(t) = l(x).\]
\end{itemize}
\end{defn}

By property (ii) in the above definition, $l=0$ on a dense subset of $C$. Also, by properties (ii) and (iii), any straight 
hairy Cantor set is compact. In particular, $\sup_{x\in C} l(x)$ is finite. 
On the other hand, as every arc in $X$ is accumulated from both sides by arcs, 
any straight hairy Cantor set is not locally connected. 

It is not immediately clear, but true, that for any open interval $I$ with $I \cap C \neq \emptyset$, 
the closure of the set $\{l(x)\mid x\in I \cap C\}$ forms a connected interval in $\D{R}$.
Due to such features, straight hairy Cantor sets in the plane are topologically unique. 

\begin{thm}\label{T:straight-hairy-homeomorphic}
All straight hairy Cantor sets in the plane are ambiently homeomorphic.
\end{thm}

We are pursuing topological objects in the plane which are ambiently homeomorphic to a straight hairy Cantor set. 
However, for practical reasons, one requires an axiomatic description of such objects.  
Let $X \subset \D{R}^2$, and consider the following axioms: 
\begin{itemize}
\item[($A_1$)] any connected component of $X$ is either a single point or a Jordan arc;
\item[($A_2$)] the closure of the set of point components of $X$ is a Cantor set of points, say $B$;  
\item[($A_3$)] any arc component of $X$ meets $B$ at one of its end points;
\item[($A_4$)] whenever $x_i \to x$ within $X$, the unique arc in $X$ connecting $x_i$ to $B$ 
converges to the unique arc in $X$ connecting $x$ to $B$, with the convergence in the Hausdorff topology; 
\item[($A_5$)] for every arc component $\gamma$ of $X$, and every $x$ in $\gamma$ minus the end points of $\gamma$, 
$x$ is not accessible from $\D{R}^2\setminus X$; 
\item[($A_6$)] $X \setminus B$ is dense in $X$; 
\item[($A_6'$)] the set of the end points of the arc components of $X$ is dense in $X$;
\end{itemize}

All of the above axioms, except $A_5$, are well-defined for any metric space. 

\begin{thm}\label{T:uniformization-abstract}
Let $X$ be a compact metric space. If $X$ satisfies axioms $A_1$ to $A_4$, and $A_6'$, then $X$ is 
homeomorphic to a straight hairy Cantor set. 
\end{thm}

\begin{thm}\label{T:uniformization}
Let $X \subset \D{R}^2$ be a compact set. If $X$ satisfies axioms $A_1$ to $A_6$, then $X$ is ambiently homeomorphic to 
a straight hairy Cantor set. 
\end{thm}

\begin{defn}\label{D:hairy-Cantor-characterization}
A compact set $X \subset \D{R}^2$ is called a \textbf{hairy Cantor set}, if it satisfies the axioms $A_1$ to $A_6$.  
\end{defn}

Combining Theorems \ref{T:straight-hairy-homeomorphic} and \ref{T:uniformization} we conclude that all hairy Cantor sets 
in the plain are ambiently homeomorphic. 

The proof of \refT{T:straight-hairy-homeomorphic} is fairly elementary and is based on a careful analysis of bump-functions 
associated to straight hairy sets. 
In contrast, the proof of \refT{T:uniformization} is more involved. 
The main idea is to identify a one dimensional topological foliation of the plane, equivalent to the foliation of $\D{R}^2$ by 
vertical lines, and each component of $X$ lies in a single leaf of the foliation. 
The (extra portions of) leaves come from hyperbolic geodesics in $\D{C} \setminus X$.
The topology of the leaves/geodesics are studied in the framework of Caratheodory's prime ends \cite{Caratheodory1913}, 
and using classical complex analysis, in paticular, Gehring-Hayman theorem \cite{GeHa62}.
The main technical difficulty is to build a global cross section for the foliation. 
We form a uniformisation of the hairy Cantor set to an straight hairy Cantor set using hyper-spaces and 
Whitney maps \cite{Wh33}. 

Hairy Cantor sets have emerged in the study of the dynamics of holomorphic maps. 
In the counterpart paper \cite{ChePe17} we explain the appearance of hairy Cantor sets as the attractors 
of a wide class of holomorphic maps with infinitely many renormalisation structures.
Roughly speaking, successive perturbations of holomorphic maps with parabolic cycles leads to infinite 
renormalisation structures discovered by Douady and Hubbard in early 80s \cite{DH85}. 
Although such systems were the subject of profound studies in the 80s and 90s, recent significant advances on the topic 
fuels interest in these objects. 
In \cite{ChePe17}, hairy Cantor sets are employed to make a striking connection between the satellite renormalisation structures, and the 
arithmetic conditions which appeared in the study of the linearisation of analytic circle diffeomorphisms by 
Herman-Yoccoz \cite{Her79,Yocc02}.
Also, our results shed light on the remarkable examples of positive area Julia sets by Buff-Ch\'eritat \cite{BC12}, 
the non-locally connected Julia sets by Sorensen \cite{Sor00} and Levin \cite{Lev11}, as well as the infinitely renormalisable 
maps without \textit{a priori} bounds in \cite{ChSh14}. 

Conjecturally, for a dense set of maps on the bifurcation locus of any non-trivial family of rational functions, the attractor of 
the map contains a hairy Cantor set. 
One may compare this to the appearance of Cantor bouquets in the dynamics of exponential 
maps \cite{Dev84,AaOv93,Sixsmith2018}, and hairy Jordan curves in the attractors of holomorphic maps with 
an irrationally indifferent fixed/periodic point \cite{Che17}. 
  
%%%%%%%%%%%%%%%%%%%%%%%%%%%%%%%%%%%%%%%%%%%%%%%%%%%%%%

\section{Straight hairy Cantor sets}
The main purpose in this section is to prove \refT{T:straight-hairy-homeomorphic}. 
Before we embark on the proof of that theorem, we present a straight hairy Cantor set, 
which essentially represents the general form of a straight hairy Cantor set. 

\begin{lem}\label{L:example}
There exists a straight hairy Cantor set in the plane.
\end{lem}

\begin{proof}
For $n\geq 1$, consider the set $\gt_n$ of multi-indexes $(i_1, i_2, \dots, i_n)$ such 
that for each $1 \leq j \leq n$, $1 \leq i_j \leq 2j-1$. 
We aim to define a Cantor set
\[C = \cap _{n\geq 1} \cup_{t\in \gt_n} I_t,\]
where each $I_t$ is a closed interval in $\D{R}$, and $I_t \subseteq I_{t'}$, for $t\in \gt_n$  and $t'\in \gt_m$,  
whenever $n\geq m$ and the first $m$ entries of $t$ are the same as the first $m$ entries of $t'$. 
For $t\in \gt_1$, we let $I_t=[0,1]$. 
Assume that for some $n \geq 2$, $I_t$ is defined for all $t\in \gt_{n-1}$. 
In order to define $I_s$, for $s \in \gt_n$, we divide each $I_t$, $t\in \gt_{n-1}$, into $4n-3$ intervals with equal lengths. 
Then, starting from the left most interval, we alternate discarding one interval, and keeping the next interval, 
until the last interval. 
In other words, with the order on the real line, we keep the odd numbered intervals.  
Thus, we are left with $2n-1$ closed intervals in $I_t$. We label these intervals as $I_s$, such that the first $n-1$ 
entries of $s$ are the same as $t$, and use the last entry to label these $2n-1$ intervals in an order preserving fashion. 
This completes the definition of $I_s$, for $s\in \gt_n$.
With this construction, every $x\in C$ has a unique address $\gt(x) =(t_1(x), t_2(x), t_3(x), \dots)$, 
where 
\[x= \cap_{n\geq 1} I_{t_1(x), t_2(x), \dots , t_n(x)}.\]

By an inductive process, we define a sequence of continuous functions $l_n: \D{R} \to [0, +\infty)$ as follows. 
For $n=0$, let $l_0(x)=1$ for all $x\in \D{R}$. 
Now assume that $l_{n-1}: \D{R} \to [0, +\infty)$ is defined for some $n\geq 1$. 
For $x\in \cup_{t\in \gt_n} I_t$, we let 
\[l_n(x) = (1 - |t_n(x)/n - 1|)l_{n-1}(x).\]
For $x\in (-\infty, 0)$ we let $l_n(x)=l_n(0)$, for $x\in (1, +\infty)$ we let $l_n(x)=l_n(1)$, and on any interval 
$(a,b)$ in the complement of $\cup_{t\in \gt_n} I_t$ with $a$ and $b$ in $\cup_{t\in \gt_n} I_t$, we use a linear 
interpolation of $l(a)$ and $l(b)$. This gives a continuous function $l_n$ on $\D{R}$. 

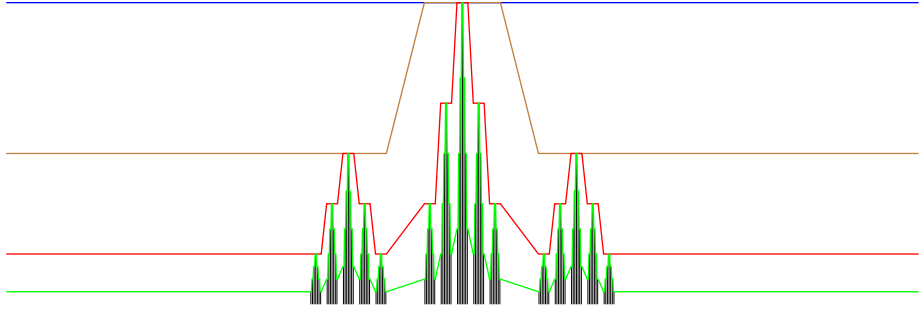
\begin{figure}[ht]
\begin{pspicture}(-4,0)(8,4)% \psgrid
\psset{xunit=4, yunit=4}
\psset{linewidth=0.1pt}

\psline(0.001786, 0.000000)(0.001786, 0.041667)
\psline(0.007143, 0.000000)(0.007143, 0.083333)
\psline(0.012500, 0.000000)(0.012500, 0.125000)
\psline(0.017857, 0.000000)(0.017857, 0.166667)
\psline(0.023214, 0.000000)(0.023214, 0.125000)
\psline(0.028571, 0.000000)(0.028571, 0.083333)
\psline(0.033929, 0.000000)(0.033929, 0.041667)
\psline(0.055357, 0.000000)(0.055357, 0.083333)
\psline(0.060714, 0.000000)(0.060714, 0.166667)
\psline(0.066071, 0.000000)(0.066071, 0.250000)
\psline(0.071429, 0.000000)(0.071429, 0.333333)
\psline(0.076786, 0.000000)(0.076786, 0.250000)
\psline(0.082143, 0.000000)(0.082143, 0.166667)
\psline(0.087500, 0.000000)(0.087500, 0.083333)
\psline(0.108929, 0.000000)(0.108929, 0.125000)
\psline(0.114286, 0.000000)(0.114286, 0.250000)
\psline(0.119643, 0.000000)(0.119643, 0.375000)
\psline(0.125000, 0.000000)(0.125000, 0.500000)
\psline(0.130357, 0.000000)(0.130357, 0.375000)
\psline(0.135714, 0.000000)(0.135714, 0.250000)
\psline(0.141071, 0.000000)(0.141071, 0.125000)
\psline(0.162500, 0.000000)(0.162500, 0.083333)
\psline(0.167857, 0.000000)(0.167857, 0.166667)
\psline(0.173214, 0.000000)(0.173214, 0.250000)
\psline(0.178571, 0.000000)(0.178571, 0.333333)
\psline(0.183929, 0.000000)(0.183929, 0.250000)
\psline(0.189286, 0.000000)(0.189286, 0.166667)
\psline(0.194643, 0.000000)(0.194643, 0.083333)
\psline(0.216071, 0.000000)(0.216071, 0.041667)
\psline(0.221429, 0.000000)(0.221429, 0.083333)
\psline(0.226786, 0.000000)(0.226786, 0.125000)
\psline(0.232143, 0.000000)(0.232143, 0.166667)
\psline(0.237500, 0.000000)(0.237500, 0.125000)
\psline(0.242857, 0.000000)(0.242857, 0.083333)
\psline(0.248214, 0.000000)(0.248214, 0.041667)
\psline(0.376786, 0.000000)(0.376786, 0.083333)
\psline(0.382143, 0.000000)(0.382143, 0.166667)
\psline(0.387500, 0.000000)(0.387500, 0.250000)
\psline(0.392857, 0.000000)(0.392857, 0.333333)
\psline(0.398214, 0.000000)(0.398214, 0.250000)
\psline(0.403571, 0.000000)(0.403571, 0.166667)
\psline(0.408929, 0.000000)(0.408929, 0.083333)
\psline(0.430357, 0.000000)(0.430357, 0.166667)
\psline(0.435714, 0.000000)(0.435714, 0.333333)
\psline(0.441071, 0.000000)(0.441071, 0.500000)
\psline(0.446429, 0.000000)(0.446429, 0.666667)
\psline(0.451786, 0.000000)(0.451786, 0.500000)
\psline(0.457143, 0.000000)(0.457143, 0.333333)
\psline(0.462500, 0.000000)(0.462500, 0.166667)
\psline(0.483929, 0.000000)(0.483929, 0.250000)
\psline(0.489286, 0.000000)(0.489286, 0.500000)
\psline(0.494643, 0.000000)(0.494643, 0.750000)
\psline(0.500000, 0.000000)(0.500000, 1.000000)
\psline(0.505357, 0.000000)(0.505357, 0.750000)
\psline(0.510714, 0.000000)(0.510714, 0.500000)
\psline(0.516071, 0.000000)(0.516071, 0.250000)
\psline(0.537500, 0.000000)(0.537500, 0.166667)
\psline(0.542857, 0.000000)(0.542857, 0.333333)
\psline(0.548214, 0.000000)(0.548214, 0.500000)
\psline(0.553571, 0.000000)(0.553571, 0.666667)
\psline(0.558929, 0.000000)(0.558929, 0.500000)
\psline(0.564286, 0.000000)(0.564286, 0.333333)
\psline(0.569643, 0.000000)(0.569643, 0.166667)
\psline(0.591071, 0.000000)(0.591071, 0.083333)
\psline(0.596429, 0.000000)(0.596429, 0.166667)
\psline(0.601786, 0.000000)(0.601786, 0.250000)
\psline(0.607143, 0.000000)(0.607143, 0.333333)
\psline(0.612500, 0.000000)(0.612500, 0.250000)
\psline(0.617857, 0.000000)(0.617857, 0.166667)
\psline(0.623214, 0.000000)(0.623214, 0.083333)
\psline(0.751786, 0.000000)(0.751786, 0.041667)
\psline(0.757143, 0.000000)(0.757143, 0.083333)
\psline(0.762500, 0.000000)(0.762500, 0.125000)
\psline(0.767857, 0.000000)(0.767857, 0.166667)
\psline(0.773214, 0.000000)(0.773214, 0.125000)
\psline(0.778571, 0.000000)(0.778571, 0.083333)
\psline(0.783929, 0.000000)(0.783929, 0.041667)
\psline(0.805357, 0.000000)(0.805357, 0.083333)
\psline(0.810714, 0.000000)(0.810714, 0.166667)
\psline(0.816071, 0.000000)(0.816071, 0.250000)
\psline(0.821429, 0.000000)(0.821429, 0.333333)
\psline(0.826786, 0.000000)(0.826786, 0.250000)
\psline(0.832143, 0.000000)(0.832143, 0.166667)
\psline(0.837500, 0.000000)(0.837500, 0.083333)
\psline(0.858929, 0.000000)(0.858929, 0.125000)
\psline(0.864286, 0.000000)(0.864286, 0.250000)
\psline(0.869643, 0.000000)(0.869643, 0.375000)
\psline(0.875000, 0.000000)(0.875000, 0.500000)
\psline(0.880357, 0.000000)(0.880357, 0.375000)
\psline(0.885714, 0.000000)(0.885714, 0.250000)
\psline(0.891071, 0.000000)(0.891071, 0.125000)
\psline(0.912500, 0.000000)(0.912500, 0.083333)
\psline(0.917857, 0.000000)(0.917857, 0.166667)
\psline(0.923214, 0.000000)(0.923214, 0.250000)
\psline(0.928571, 0.000000)(0.928571, 0.333333)
\psline(0.933929, 0.000000)(0.933929, 0.250000)
\psline(0.939286, 0.000000)(0.939286, 0.166667)
\psline(0.944643, 0.000000)(0.944643, 0.083333)
\psline(0.966071, 0.000000)(0.966071, 0.041667)
\psline(0.971429, 0.000000)(0.971429, 0.083333)
\psline(0.976786, 0.000000)(0.976786, 0.125000)
\psline(0.982143, 0.000000)(0.982143, 0.166667)
\psline(0.987500, 0.000000)(0.987500, 0.125000)
\psline(0.992857, 0.000000)(0.992857, 0.083333)
\psline(0.998214, 0.000000)(0.998214, 0.041667)

\psset{linewidth=0.5pt}

\psline[linecolor=blue](-1.000000, 1.000000)(0.000000, 1.000000)(1.000000, 1.000000)(2.000000, 1.000000)
\psline[linecolor=brown](-1.000000, 0.500000)(0.000000, 0.500000)(0.250000, 0.500000)(0.375000, 1.000000)(0.625000, 1.000000)(0.750000, 0.500000)(1.000000, 0.500000)(2.000000, 0.500000)
\psline[linecolor=red](-1.000000, 0.166667)(0.000000, 0.166667)(0.035714, 0.166667)(0.053571, 0.333333)(0.089286, 0.333333)(0.107143, 0.500000)(0.142857, 0.500000)(0.160714, 0.333333)(0.196429, 0.333333)(0.214286, 0.166667)(0.250000, 0.166667)(0.375000, 0.333333)(0.410714, 0.333333)(0.428571, 0.666667)(0.464286, 0.666667)(0.482143, 1.000000)(0.517857, 1.000000)(0.535714, 0.666667)(0.571429, 0.666667)(0.589286, 0.333333)(0.625000, 0.333333)(0.750000, 0.166667)(0.785714, 0.166667)(0.803571, 0.333333)(0.839286, 0.333333)(0.857143, 0.500000)(0.892857, 0.500000)(0.910714, 0.333333)(0.946429, 0.333333)(0.964286, 0.166667)(1.000000, 0.166667)(2.000000, 0.166667)
\psline[linecolor=green](-1.000000, 0.041667)(0.000000, 0.041667)(0.003571, 0.041667)(0.005357, 0.083333)(0.008929, 0.083333)(0.010714, 0.125000)(0.014286, 0.125000)(0.016071, 0.166667)(0.019643, 0.166667)(0.021429, 0.125000)(0.025000, 0.125000)(0.026786, 0.083333)(0.030357, 0.083333)(0.032143, 0.041667)(0.035714, 0.041667)(0.053571, 0.083333)(0.057143, 0.083333)(0.058929, 0.166667)(0.062500, 0.166667)(0.064286, 0.250000)(0.067857, 0.250000)(0.069643, 0.333333)(0.073214, 0.333333)(0.075000, 0.250000)(0.078571, 0.250000)(0.080357, 0.166667)(0.083929, 0.166667)(0.085714, 0.083333)(0.089286, 0.083333)(0.107143, 0.125000)(0.110714, 0.125000)(0.112500, 0.250000)(0.116071, 0.250000)(0.117857, 0.375000)(0.121429, 0.375000)(0.123214, 0.500000)(0.126786, 0.500000)(0.128571, 0.375000)(0.132143, 0.375000)(0.133929, 0.250000)(0.137500, 0.250000)(0.139286, 0.125000)(0.142857, 0.125000)(0.160714, 0.083333)(0.164286, 0.083333)(0.166071, 0.166667)(0.169643, 0.166667)(0.171429, 0.250000)(0.175000, 0.250000)(0.176786, 0.333333)(0.180357, 0.333333)(0.182143, 0.250000)(0.185714, 0.250000)(0.187500, 0.166667)(0.191071, 0.166667)(0.192857, 0.083333)(0.196429, 0.083333)(0.214286, 0.041667)(0.217857, 0.041667)(0.219643, 0.083333)(0.223214, 0.083333)(0.225000, 0.125000)(0.228571, 0.125000)(0.230357, 0.166667)(0.233929, 0.166667)(0.235714, 0.125000)(0.239286, 0.125000)(0.241071, 0.083333)(0.244643, 0.083333)(0.246429, 0.041667)(0.250000, 0.041667)(0.375000, 0.083333)(0.378571, 0.083333)(0.380357, 0.166667)(0.383929, 0.166667)(0.385714, 0.250000)(0.389286, 0.250000)(0.391071, 0.333333)(0.394643, 0.333333)(0.396429, 0.250000)(0.400000, 0.250000)(0.401786, 0.166667)(0.405357, 0.166667)(0.407143, 0.083333)(0.410714, 0.083333)(0.428571, 0.166667)(0.432143, 0.166667)(0.433929, 0.333333)(0.437500, 0.333333)(0.439286, 0.500000)(0.442857, 0.500000)(0.444643, 0.666667)(0.448214, 0.666667)(0.450000, 0.500000)(0.453571, 0.500000)(0.455357, 0.333333)(0.458929, 0.333333)(0.460714, 0.166667)(0.464286, 0.166667)(0.482143, 0.250000)(0.485714, 0.250000)(0.487500, 0.500000)(0.491071, 0.500000)(0.492857, 0.750000)(0.496429, 0.750000)(0.498214, 1.000000)(0.501786, 1.000000)(0.503571, 0.750000)(0.507143, 0.750000)(0.508929, 0.500000)(0.512500, 0.500000)(0.514286, 0.250000)(0.517857, 0.250000)(0.535714, 0.166667)(0.539286, 0.166667)(0.541071, 0.333333)(0.544643, 0.333333)(0.546429, 0.500000)(0.550000, 0.500000)(0.551786, 0.666667)(0.555357, 0.666667)(0.557143, 0.500000)(0.560714, 0.500000)(0.562500, 0.333333)(0.566071, 0.333333)(0.567857, 0.166667)(0.571429, 0.166667)(0.589286, 0.083333)(0.592857, 0.083333)(0.594643, 0.166667)(0.598214, 0.166667)(0.600000, 0.250000)(0.603571, 0.250000)(0.605357, 0.333333)(0.608929, 0.333333)(0.610714, 0.250000)(0.614286, 0.250000)(0.616071, 0.166667)(0.619643, 0.166667)(0.621429, 0.083333)(0.625000, 0.083333)(0.750000, 0.041667)(0.753571, 0.041667)(0.755357, 0.083333)(0.758929, 0.083333)(0.760714, 0.125000)(0.764286, 0.125000)(0.766071, 0.166667)(0.769643, 0.166667)(0.771429, 0.125000)(0.775000, 0.125000)(0.776786, 0.083333)(0.780357, 0.083333)(0.782143, 0.041667)(0.785714, 0.041667)(0.803571, 0.083333)(0.807143, 0.083333)(0.808929, 0.166667)(0.812500, 0.166667)(0.814286, 0.250000)(0.817857, 0.250000)(0.819643, 0.333333)(0.823214, 0.333333)(0.825000, 0.250000)(0.828571, 0.250000)(0.830357, 0.166667)(0.833929, 0.166667)(0.835714, 0.083333)(0.839286, 0.083333)(0.857143, 0.125000)(0.860714, 0.125000)(0.862500, 0.250000)(0.866071, 0.250000)(0.867857, 0.375000)(0.871429, 0.375000)(0.873214, 0.500000)(0.876786, 0.500000)(0.878571, 0.375000)(0.882143, 0.375000)(0.883929, 0.250000)(0.887500, 0.250000)(0.889286, 0.125000)(0.892857, 0.125000)(0.910714, 0.083333)(0.914286, 0.083333)(0.916071, 0.166667)(0.919643, 0.166667)(0.921429, 0.250000)(0.925000, 0.250000)(0.926786, 0.333333)(0.930357, 0.333333)(0.932143, 0.250000)(0.935714, 0.250000)(0.937500, 0.166667)(0.941071, 0.166667)(0.942857, 0.083333)(0.946429, 0.083333)(0.964286, 0.041667)(0.967857, 0.041667)(0.969643, 0.083333)(0.973214, 0.083333)(0.975000, 0.125000)(0.978571, 0.125000)(0.980357, 0.166667)(0.983929, 0.166667)(0.985714, 0.125000)(0.989286, 0.125000)(0.991071, 0.083333)(0.994643, 0.083333)(0.996429, 0.041667)(1.000000, 0.041667)(2.000000, 0.041667)

\end{pspicture}
\caption{The graphs of the functions $l_1$, $l_2$, $l_3$, and $l_4$ are drown in colors blue, brown, red, and green, respectively. 
The vertical lines in black form the straight hairy Cantor set determined by the functions $(l_n)_{n\geq 1}$. It is possible to zoom 
in on the image for further details.}
\label{F:SHCS}
\end{figure}

Note that for all $n \geq 1$, $1 \leq t_n(x) \leq 2n - 1$, which implies $|t_n(x) - n| \leq n-1$. 
Therefore, $l_{n-1}(x)/n \leq l_n(x) \leq l_{n-1}(x)$, for all $n\geq 1$ and $x\in \D{R}$.
In particular, we may define 
\[l(x)=\lim_{n \to \infty} l_n(x).\]
For $x \in C$, we have
\begin{equation}\label{E:L:example-2}
l(x)=\prod_{n \geq 1} (1 - |t_n(x)/n - 1|).
\end{equation}

Since each $l_n$ is continuous, the set 
\[X_n = \{ (x, y) \in \D{R}^2 \mid x \in \D{R}, 0 \leq y \leq l_n(x) \},\] 
is closed in $\D{R}^2$.
Therefore, 
\[\cap_{n \geq 0} X_n =\{ (x,y) \in \D{R}^2 \mid x\in \D{R}, 0 \leq y \leq l(x)\}\] 
is also closed in $\D{R}^2$.  
This implies that $l: \D{R} \to [0, +\infty)$ is upper semi-continuous, that is, for all $x\in \D{R}$, 
\begin{equation}\label{E:L:example-1}
\limsup_{t \to x} l(t) \leq l(x).
\end{equation}

We claim that the set 
\[X = \{ (x, y) \in \D{R}^2 \mid x \in C, 0 \leq y \leq l(x) \}\]
is a straight hairy Cantor set. 

It follows from the product formula in \refE{E:L:example-2} that for any $x \in C$ with 
$t_n(x) = n$, for large enough $n$, we have $l(x) > 0$.
Evidently, the set of such points is dense in $C$. 
This gives us property (i) in \refD{D:strainght-hairy-Cantor}. 

If $x$ is an end point of $C$, we must have either $t_n(x)=1$ or $t_n(x)=2n-1$, for large enough $n$. 
Using the product formula for $l$, those imply that $l(x)=0$. 
By \refE{E:L:example-1}, $0 \leq \limsup_{t \to x, t\in C} l(t)\leq l(x)=0$. 
Thus, property (ii) in \refD{D:strainght-hairy-Cantor} holds. 

Now assume that $x \in C$ is not an end point. If $l(x)=0$, by \refE{E:L:example-1}, we must have 
$0 \leq \limsup_{t \to x^+, t\in C} l(t) \leq l(x)=0$ and $0 \leq \limsup_{t \to x^-, t\in C} l(t)\leq l(x)=0$.  
If $l(x) > 0$, by \refE{E:L:example-2}, there must be $n_0$ such that for all $n\geq n_0$ we have $n/2 < t_n(x) < 3n/2$. 
For $n \geq \max\{n_0, 3\}$ we may consider the unique point $p_n \in C$ whose address is 
\[(t_1(x), \cdots, t_{n-1}(x), t_n(x) - 1, t_{n+1}(x), t_{n+2}(x), \cdots).\]
Evidently, $\{p_n\}$ is a strictly increasing sequence that tends to $x$. 
Moreover, 
\begin{align*}
\lim_{n \to \infty} \left| \frac{l(p_n)}{l(x)} - 1 \right| 
&= \lim_{n \to \infty} \left| \frac{1 - |(t_n(x) - 1)/n - 1|}{1 - |t_n(x)/n - 1|} - 1\right| \\
&= \lim_{n \to \infty} \left|\frac{|t_n(x)/n - 1| - |t_n(x)/n - 1/n - 1|}{1 - |t_n(x)/n - 1|}\right| \\
&\leq \lim_{n \to \infty} \frac{1/n}{1 - |t_n(x)/n - 1|} = 0.
\end{align*}
Therefore, $l(p_n) \to l(x)$. 
Similarly, we may define a strictly decreasing sequence $\{q_n\}$ in $C$ converging to $x$ such that 
$l(q_n) \to l(x)$. 
This implies part (iii) of \refD{D:strainght-hairy-Cantor}. 
\end{proof}

Below we introduce some basic definitions and notations which will be used for the proof 
of \refT{T:straight-hairy-homeomorphic}. 

Given a straight hairy Cantor set $X$, the corresponding function $l: C \to [0,+\infty)$ determining $X$ in 
\refD{D:strainght-hairy-Cantor} is unique. 
We refer to $l$ as the \textbf{length function} of $X$, and often denote it by $l^X$. 
Also, we refer to the Cantor set $C \subset X$ as the \textbf{base} Cantor set of $X$, or equivalently, say that $X$ 
is \textbf{based} on the Cantor set $C$. 
Note that $C$ is the unique Cantor set of points in $X$ which contains all point components of $X$.

%For our convenience, we extend the length function $l^X$ onto $\D{R}$ by setting $l^X(x)=0$, for all 
%$x \in \D{R} \setminus C$. 

Let $C \subset \D{R}$ be a Cantor set of points. A set $P \subset \D{R}$ is called a \textbf{Cantor partition} 
for $C$, if $P$ is the union of a finite number of closed intervals in $\D{R}$ with $C \subset P$ and 
$\partial P \subset C$. 
A \textbf{nest of Cantor partitions shrinking} to $C$, by definition, is a nest of Cantor partitions 
$P_0 \supset P_1 \supset P_2 \supset \dots$ for $C$ such that $C = \cap_{n \geq 0} P_n$. 
When we write a Cantor partition $P$ for $C$ as $P = \cup_{i = 1}^k J_i$, we mean that each $J_i$ is a non-empty 
closed interval in $\D{R}$ and $J_i\cap J_j=\emptyset$, for $1\leq i < j \leq k$.  
Also, we assume that the intervals $J_i$ are labelled in a way that $\inf J_i < \inf J_j$ if and only if $i< j$. 

Let $X$ be a straight hairy Cantor set based on the Cantor set of points $C$.  
For any Cantor partition $P$ for $C$, by property (ii) in \refD{D:strainght-hairy-Cantor}, 
$l^X(t)=0$, for all $t \in \partial P$. 

Let $C$ be a Cantor set, and let $l: C \to [0, \infty)$ be an upper semi-continuous function. 
For any closed interval $[u,v] \subset \D{R}$ with $[u,v] \cap C \neq \emptyset$, we use the notation 
\[\op{Max}(l, [u,v]) = \max \{ l(x) \mid x \in [u,v] \cap C\}.\]
By the upper semi-continuity of $l$, there is $m_{[u,v]} \in [u,v] \cap C$ with 
\[\op{Max}(l, [u,v])=l(m_{[u,v]}).\] 
Clearly, $m_{[u,v]}$ with this property might not be unique. 
For any such choice of $m_{[u,v]}$, we define
\[B^l_{[u,v]} : [u,v] \to [0, \op{Max}(l, [u,v])]\] 
as 
\[B^l_{[u,v]} (x) = \begin{cases}
\op{Max}(l, [u, x])  	&   x \leq m_{[u,v]},\\
\op{Max}(l, [x, v])     & x \geq m_{[u,v]}.
\end{cases}\]
We refer to $B^l_{[u,v]}$ as the \textbf{bump function} of $l$ on the interval $[u,v]$ associated to $m_{[u,v]}$. 

Assume that $P=\cup_{i=1}^k J_i$ is a Cantor partition for $C$. 
We may consider a bump function on each interval $J_i$, say $B^l_{J_i}$, and combine them to 
define a bump function for $l$ associated to the partition $P$,  
\[B_P^l:\D{R} \to [0, +\infty),\]
according to 
\[B_P^l(x)= 
\begin{cases}
B_{J_i}^l(x) & \tif x\in J_i, \\
0  &  \tif x\notin P.
\end{cases}\]

The following lemma is fairy easy, but will be used several times in the upcoming arguments.   

\begin{lem}\label{L:upper_sc_limit}
Let $C$ be a Cantor set, and $l : C \to [0, \infty)$ be an upper semi-continuous function. 
Also assume that $(I_n)_{n \geq 0}$ be a sequence of closed intervals in $\D{R}$ shrinking to a point $z \in C$.
Then
\[ l(z) = \lim_{n \to \infty} \op{Max}(l, I_n). \]
\end{lem}

\begin{proof}
For all $n \geq 0$, we have $l(z) \leq \op{Max}(l, I_n)$, which gives 
\[ l(z) \leq \liminf_{n \to \infty} \op{Max}(l, I_n). \]
On the other hand, since $l$ is upper semi-continuous, we have
\[ l(z) \geq \limsup_{n \to \infty} \op{Max}(l, I_n). \qedhere \]
\end{proof}

The key property of bump functions is stated in the following lemma. 

\begin{lem}\label{L:bump-functions}
Let $X$ be a straight hairy Cantor set, based on a Cantor set $C$, and with length function $l^X: C \to [0, \infty)$. 
Assume that $P=\cup_{i=1}^k J_i$ is a Cantor partition for $C$. 
Any bump function $B_P^{l^X}:\D{R} \to [0, +\infty)$ is piecewise monotone and continuous.
\end{lem}

%For any $u\leq v$ in $C$ with $l^X(u)=l^X(v)=0$, and any choice of $m_{[u,v]} \in [u,v]$ satisfying 
%$\op{Max}(l^X, [u,v])=l^X(m_{[u,v]})$, the bump function $B^X_{[u,v]}: [u,v] \to [0, +\infty)$ 
%is piecewise monotone and continuous.

\begin{proof}
Recall that for any $u\in \partial P$, $l^X(u)=0$. 
Therefore, it is enough to show that each $B^{l^X}_{J_i}$ is piecewise monotone and continuous. 
Fix an arbitrary $i$, and let $J_i=J=[u,v]$ and $B_{J}^X=B^{l^X}_{J}$.
Assume that $B^X_J$ is associated to some $m_J \in J$ satisfying $\op{Max}(l^X, J)=l^X(m_J)$. 

By definition, $B^X_J$ is increasing on the interval $[u, m_J]$ and is decreasing on $[m_J, v]$. 
 
If $u=v$, then $B^X_J(u)=0$ and there is nothing to prove. 
When $u \neq v$, by property (i) in \refD{D:strainght-hairy-Cantor}, $l^X$ is positive at some point in $J$, and 
therefore, $m_J \in (u,v)$. 

By the increasing property of $B^X_J$ on $[u, m_J]$, and properties (ii) and (iii) in \refD{D:strainght-hairy-Cantor}, 
we note that 
\[l^X(u)= B^X_J(u) \leq \lim_{t \to u^+} B^X_J(t) \leq \limsup_{t\to u^+} l^X(t)=l^X(u).\]
This implies the continuity of $B^X_J$ at $u$. 
Similarly, one obtains the continuity of $B^X_J$ at $v$. 

Since $l(m_J) \neq 0$, by property (ii), $m_J$ is not an end point of $C$, and by property (iii), we have 
\[\limsup_{t \to m_J^-,\, t \in C} l^X(t) = \limsup_{t \to m_J^+,\, t\in C} l^X(t) = l^X(m_J).\]
Hence, 
\begin{align*}
l^X(m_J)= \limsup_{t \to m_J^-,\, t \in C} l^X(t)  & \leq \lim_{t\to m_J^-} B^X_J(t) \\ 
&\leq B^X_J(m_J) \leq \lim_{t\to m_J^+} B^X_J(t)  % \\
\leq \limsup_{t \to m_J^+,\, t \in C} l^X(t)= l^X(m_J).
\end{align*}
This implies the continuity of $B^X_J$ at $m_J$. 

For an arbitrary $x \in (u, m_J)$, if $l(x) \neq B^X_J(x)$, then $B^X_J$ is constant near $x$, and in particular, it is 
continuous at $x$. If $l^X(x)=B^X_J(x)$, then  
\begin{align*}
l^X(x)=\limsup_{t\to x^-, \, t\in C} l^X(t) & \leq \lim_{t \to x^-} B^X_J(t) \\
& \leq B^X_J(x) \leq \lim_{t \to x^+} B^X_J(t)   % \\ 
\leq \limsup_{t\to x^+, \, t\in C} l^X(t)=l^X(x). 
\end{align*}
This implies the continuity of $B^X_J$ at $x$. 
Similarly, one proves the continuity of $B^X_J$ on $(m_J, v)$. 
\end{proof}

The proof of \refT{T:straight-hairy-homeomorphic} will be based on the following two propositions. 

\begin{propo}\label{P:straight-hairy-partitions-exist}
Let $X$ and $Y$ be straight hairy Cantor sets based on the same Cantor set $C$ and with length functions 
$l^X$ and $l^Y$.
There are nests of Cantor partitions $P^X_n= \cup_{i=1}^{k_n} J^X_{n,i}$ and $P^Y_n= \cup_{i=1}^{k_n} J^Y_{n,i}$, 
for $n\geq 0$, shrinking to $C$, such that for all $n\geq 0$ we have 
\begin{itemize}
\item[(i)] for all $1\leq i \leq k_n$ we have $\big | J^X_{n,i} \big| < 2^{-n} |C|$ and $\big| J^Y_{n,i} \big| < 2^{-n}|C|$; 
\footnote{We use the notation $|J|$ to denote the Euclidean diameter of a given $J \subset \D{R}$.}
\item[(ii)] for all $1\leq i \leq k_{n+1}$ and $1 \leq j \leq k_n$, $J^X_{n+1, i} \subset J^X_{n, j}$ 
if and only if $J^Y_{n+1, i} \subset J^Y_{n, j}$;
\item[(iii)] whenever $J^X_{n+1, i} \subset J^X_{n, j}$ for some $i$ and $j$, then 
\[1-2^{-(n+2)} < 
\frac{\op{Max}(l^Y,J^Y_{n+1,i}) /\op{Max}(l^Y, J^Y_{n, j})}
{\op{Max}(l^X, J^X_{n+1, i})/\op{Max} (l^X, J^X_{n, j})} 
< 1 + 2^{-(n+2)}.\]
\end{itemize}
\end{propo}

\begin{proof}
We shall use the notations 
\[J^X_{n,i}=[u^X_{n,i}, v^X_{n,i}], J^Y_{n,i}= [u^Y_{n,i}, v^Y_{n,i}], \quad \tfor n\geq 0 \tand 1\leq i \leq k_n,\]
with 
\[u^X_{n,1} < v^X_{n, 1} < u^X_{n,2} < \cdots < u^X_{n, k_n} < v^X_{n, k_n},\]
and 
\[u^Y_{n, 1} < v^Y_{n, 1} < u^Y_{n,2} < \cdots < u^Y_{n, k_n} < v^Y_{n, k_n}.\]

By an inductive argument, we shall simultaneously define the partitions $P_n^X$ and $P_n^Y$. 
For $n=0$, we set $k_0=1$ and $J_{0,1}^X = J_{0,1}^Y= [\inf C, \sup C]$. That is, each of  $P_0^X$ and $P_0^Y$ 
consists of one closed interval. Property (i) in the proposition holds, and there is nothing to verify for items (ii) and (iii).  

Now assume that the partitions $(P_m^X)_{m=0}^n$ and $(P_m^Y)_{m=0}^n$ are defined for some $n\geq 0$ 
and satisfy the properties in items (i)-(iii) in the proposition. Below we defined $P_{n+1}^X$ and $P_{n+1}^Y$. 

Let us assume that $n$ is even (the odd case is mentioned at the end of the proof). 
Let $P_{n+1}^Y=\cup_{i=1}^{k_{n+1}}J^Y_{n+1, i}$ be an arbitrary Cantor partition of $C$ such that $P_{n+1}^Y$ 
is a subset of $P_{n}^Y$ and 
\begin{equation}\label{E:base-step}
\big| J^Y_{n+1,i}\big|  < 2^{-(n+2)} |C|, \quad \tfor 1\leq i \leq k_{n+1}.
\end{equation}

We shall use $P^Y_{n+1}$ to construct $P_{n+1}^X$.
To that end, it is enough to identify all the intervals $J^X_{n+1,i}$ which are contained in $J^X_{n,j}$, for each $j$ 
with $1\leq j \leq k_n$. 
Before identifying those intervals, we note that the number of such intervals for each $j$ must be equal to the 
number of the intervals $J^Y_{n+1,i}$ which are contained in $J^Y_{n,j}$. 
This will clearly guarantee property (ii). To ensure property (iii) we need more detailed analysis. 

Fix an arbitrary $j$ with $1\leq j \leq k_n$. 
Recall that 
\[l^X(u^X_{n,j})= l^X(v^X_{n,j})= l^Y(u^Y_{n,j})= l^Y(v^Y_{n,j})=0.\] 
There are $m_{n, j}^X \in (u_{n, j}^X, v_{n, j}^X)$ and $m_{n, j}^Y \in (u_{n, j}^Y, v_{n, j}^Y)$
such that 
\[l^X(m_{n, j}^X) = \op{Max}(l^X, J^X_{n,j}), \quad l^Y(m_{n, j}^Y) = \op{Max}(l^Y, J^Y_{n,j}).\] 

Since $P_{n+1}^Y \subseteq P_n^Y$ there are $0 \leq p \leq q \leq k_{n+1}$ such that 
\[u_{n+1, p}^Y = u_{n, j}^Y ,\quad v_{n+1, q}^Y = v_{n, j}^Y. \]
We set
\[u_{n+1, p}^X = u_{n, j}^X , \quad v_{n+1, q}^X = v_{n, j}^X.\]
If $p=q$, then we are done. Below we assume that $0 \leq p < q \leq k_{n+1}$. 
We need to identify $u_{n+1, i}^X$ and $v_{n+1,i-1}^X$, for $p < i \leq q$. 

There is a unique integer $l$ with $p \leq  l \leq q $ such that
\[u_{n+1,l}^Y<m_{n,j}^Y<v_{n+1,l}^Y.\]
We shall determine $v_{n+1,i}^X$ and $u_{n+1,i+1}^X$ recursively by going from $i=p$ to $l-1$, if there are any such 
$i$, and independently going from $i= q-1$ to $l$, if there are any such $i$, so that 
\[u_{n+1,p}^X<v_{n+1,p}^X<\cdots<u_{n+1,l}^X<m_{n,j}^X<v_{n+1,l}^X<\cdots<u_{n+1,q}^X<v_{n+1,q}^X.\]
We only explain the process for $p \leq i \leq l-1$, the other case being similar. 

Let us assume that for some $p \leq i < l - 1$, we have already defined
\[u_{n+1,p}^X < v_{n+1,p}^X < \cdots < u_{n+1,i}^X < m_{n,j}^X.\]
Since 
\[\op{Max}(l^X, [u^X_{n+1, p}, v^X_{n+1,q}])= \op{Max}(l^X, J^X_{n,j})= l^X(m^X_{n,j}),\] 
and $m^X_{n,j} > u^X_{n+1,i}$, we have 
\[\op{Max}\big (l^X, [u_{n+1, i}^X, v_{n+1, q}^X]\big)= l^X(m^X_{n,j}).\] 
Therefore, we may use $m_{n, j}^X$ to define a bump function $B^{l^X}_{[u_{n+1, i}^X, v_{n+1, q}^X]}$ on the set 
$[u_{n+1, i}^X, v_{n+1, q}^X]$.
By \refL{L:bump-functions}, as we move $v$ from $u_{n+1, i}^X$ to $m_{n, j}^X$, the value of 
$B_{[u_{n+1, i}^{X}, v_{n+1, q}^X]}^{l^X}(v)$ continuously increases from $0$ to $\op{Max}(l^X, J^X_{n, j})$. 
Therefore, since 
\[0 < \op{Max}(l^Y, J^Y_{n+1, i}) / \op{Max}(l^Y, J^Y_{n, j}) \leq 1,\] 
there must be $v' \in (u_{n+1, i}^X, m_{n,j}^X)$ such that 
\[1-2^{-(n+2)} < \frac{\op{Max}(l^Y, J^Y_{n+1, i}) / \op{Max}(l^Y, J^Y_{n, j})}
{l^X(v') / \op{Max}(l^X, J^X_{n,j})}  < 1 + 2^{-(n+2)}.\] 
Now, by slightly moving $v'$ to the right, so that it becomes an end point of $C$, we may find an interval 
$(v_{n+1, i}^X, u_{n+1,i+1}^X)$ such that
\begin{gather*} 
u_{n+1, i}^X < v' < v_{n+1, i}^X < u_{n+1, i+1}^X < m_{n, j}^X , \\
v_{n+1, i}^X, u_{n+1, i+1}^X \in C, \quad (v_{n+1, i}^X, u_{n+1,i+1}^X) \cap C= \emptyset,
\end{gather*}
and 
\begin{gather*} 
1-2^{-(n+2)} < 
\frac{\op{Max}(l^Y, J^Y_{n+1,i}) / \op{Max}(l^Y, J^Y_{n, j})}
{\op{Max}(l^X, [u^X_{n+1,i}, v^X_{n+1, i}]) / \op{Max}(l^X, J^X_{n,j})} < 1 + 2^{-(n+2)}. 
\end{gather*}

Here, $v_{n+1,i}^X=v'$ and $u_{n+1, i}^X= \inf \{t\in C \mid t> v'\}$. 
Also, we have used that the map $v' \mapsto \op{Max}(l^X, [u_{n+1,i}^X, v'])$ depends continuously on $v'$. 
This completes the induction step to identify $v_{n+1,i}^X$ and $u_{n+1,i+1}^X$. 
Note that at the end of this construction, $m^X_{n,j} \in [u_{n+1,l}^X, v_{n+1,l}^X]$, thus, 
\[\frac{\op{Max}(l^Y, J^Y_{n+1,l}) / \op{Max}(l^Y, J^Y_{n,j})}
{\op{Max}(l^X, J^X_{n+1,l}) / \op{Max}(l^X, J^X_{n,j})}=1.\]

We assumed at the induction step that $n$ is even. For odd $n$, we interchange the role of $X$ and $Y$ in the 
above process, that is, we start with an arbitrary partition $P^X_{n+1}$ for $C$ which is contained in $P^X_n$
and satisfies \refE{E:base-step}, and build $P^Y_{n+1}$ in the same fashion. 
This guarantees that for all $n\geq 0$ and $1\leq i \leq k_n$ we have $\big| J^X_{n,i}\big| \leq 2^{-n}|C|$
and $\big| J^Y_{n,i}\big| \leq 2^{-n}|C|$. 
\end{proof} 

\begin{propo}\label{P:straight-hairy-matching-partitions}
Let $X$ and $Y$ be straight hairy Cantor sets based on the same Cantor set of points $C$ and with length functions 
$l^X$ and $l^Y$. 
Let $P_n=\bigcup_{i = 1}^{k_n} J_{n,i}$, for $n\geq 0$, be a nest of Cantor partitions shrinking to $C$, with 
disjoint non-empty closed intervals $J_{n,i}$, for each $n\geq 0$. 
Assume that whenever $J_{n, i} \subseteq J_{n-1, j}$ for some integers $n$, $i$, and $j$, we have 
%\[ 1 - 2^{-(n+2)} < \frac{M_{n+1,i}^2}{M_{n, j}^2} : \frac{M_{n+1, i}^1}{M_{n, j}^1} < 1 + 2^{-(n+2)}, \]
\[ 1 - 1/2^{n+1} < \frac{\op{Max}(l^Y,J_{n,i})/\op{Max}(l^Y,J_{n-1,j})}
{\op{Max}(l^X,J_{n,i})/\op{Max}(l^X,J_{n-1,j})} < 1 + 1/2^{n+1}.\]
Then, $X$ and $Y$ are ambiently homeomorphic.\footnote{Recall that two subsets of the plane are ambiently 
homeomorphic if there is a homeomorphism of the plain which maps one bijectively onto the other.}
\end{propo}

Let us denote the Euclidean ball of radius $r$ about $z\in \D{R}^2$ with $\D{D}(z,r)$. 
In the same fashion, given $K \subset \D{R}^2$, let 
\[\D{D}(K, r)= \cup_{z\in K} \D{D}(z,r).\] 

\begin{proof}
Without loss of generality we may assume that $C \subset [0,1]$, with $0$ and $1$ in $C$. 
This may be achieved by applying a translation and then a linear rescaling. 
Also, by applying two rescalings in the second coordinate in $\D{R}^2$, we may assume that 
\[\op{Max}(l^X,[0,1])= \op{Max}(l^Y,[0,1])= 1/2.\] 
These changes do not alter the ratios of the maximums in the hypotheses of the proposition.

For $n \geq 1$, we define the function $\gh_n: \D{R} \to \D{R}$ as follows. 
If $t\in P_n$, we identify the unique integers $i$ and $ j$ with $t \in J_{n,i} \subseteq J_{n-1, j}$, and set 
\[\gh_n(t) 
= \frac{\op{Max}(l^Y,J_{n,i})/\op{Max}(l^Y,J_{n-1,j})}{\op{Max}(l^X,J_{n,i})/\op{Max}(l^X,J_{n-1,j})}.\]
For $t \notin P_n$, we set $\gh_n(t)=1$. 
Clearly, $\gh_n$ is constant on each closed interval of the partition $P_n$. 

By the assumption in the proposition, for all $n\geq 1$ and all $t\in \D{R}$, we have 
\[|\gh_n(t) - 1| < 1/2^{n+1}.\]

By an inductive argument, we aim to build homeomorphisms $\gf_n : \D{R}^2 \to \D{R}^2$, for $n\geq 0$, 
such that each $\gf_{n}(X)$ is a straight hairy Cantor set based on $C$ and its length function $l^{\gf_{n}(X)}$ 
satisfies the relation 
\begin{equation}\label{E:incremental-fix}
l^{\gf_n(X)}(t) = \frac{\op{Max}(l^Y, J_{n,i})}{\op{Max}(l^X, J_{n, i})} \cdot l^X(t),
\end{equation}
whenever $t \in J_{n, i}$ for some $i$.

For $n = 0$, we let $\gf_0$ be the identity map on $\D{R}^2$. 
Assume that for some $n \geq 1$, $\gf_{n - 1}$ is defined and satisfies \refE{E:incremental-fix}.
Below we define $\gf_n$. 

Let $B_n$ be a bump function for $l^{\gf_{n-1}(X)}$ associated to the Cantor partition $P_n$. 
Since $\gh_n$ is constant on each interval of $P_n$, we have
\begin{align*}
\max\{\gh_n(t)& B_n(t)  \mid t \in [0, 1] \}  \\
& = \max \{ \gh_n(t) l^{\gf_{n-1}(X)}(t) \mid t \in C \} \\
& = \max \left\{\frac{\op{Max}(l^Y, J_{n,i})}{\op{Max}(l^X, J_{n, i})}\cdot  l^X(t) \; \Big \vert \; 
\begin{array}{ll}
1 \leq i \leq k_n \\ 
1\leq j \leq k_{n-1}
\end{array}, t \in J_{n, i} \cap C  \subset J_{n-1,j}\right\} \\
& = \max\{\op{Max}(l^Y, J_{n,i}) \mid 1 \leq i \leq k_n \} = 1/2. 
\end{align*}
In particular, for all $t\in [0,1]$, 
\[B_n(t) \leq 1/(2\gh_n(t)) \leq (1/2)(4/3)=2/3.\] 

Define $H_n: \D{R}^2 \to \D{R}^2$ as the identity map outside of $[0, 1]^2$, and for $(x, y) \in [0, 1]^2$, let 
\[H_n(x, y) =
\begin{cases}
\left( x, \gh_n(x) y \right) &\tif	y \leq B_n(x), \\
\left( x, 1 - \frac{1 - \gh_n(x)B_n(x)}{1 - B_n(x)}(1 - y) \right) & \tif	y > B_n(x).
\end{cases}\]
Evidently, $H_n$ is identity outside $P_n \times [0, 1]$, and also on the real slice $(\D{R}, 0)$.
For $x_0 \in P_n$, $H_n$ maps the line segment $\{x_0\} \times [0,1]$ onto itself in a piecewise linear fashion, with 
\[H_n \left( \{x_0\} \times [0, B_n(x_0)] \right) = \{x_0\} \times [0, \gh_n(x_0) B_n(x_0)], \] 
\[H_n \left( \{x_0\} \times [B_n(x_0), 1] \right) = \{x_0\} \times [\gh_n(x_0) B_n(x_0), 1]. \]
In particular, 
\[H_n \left (\{ (x, y) \mid x \in [0, 1], y = B_n(x) \}\right )
=\{ (x, y) \mid x \in [0, 1], y = \gh_n(x) B_n(x) \}.\]
It follows from the continuity of $B_n$ in \refL{L:bump-functions} that $H_n$ is continuous, and hence, a 
homeomorphism of the plane. We define $\gf_n= H_n \circ \gf_{n-1}$. 

The set $\gf_n(X)$ is a straight hairy Cantor set. That is because, the homeomorphism $H_n$ sends 
each vertical line $x=x_0$ into itself, and is the identity on the horizontal line $(\D{R}, 0)$. 

To prove that $\gf_n$ satisfies \refE{E:incremental-fix}, choose $t \in J_{n, i} \subseteq J_{n-1, j}$, and note that 
$\gf_n= H_n \circ \gf_{n-1}$ implies  
\begin{align*}
l^{\gf_n(X)}(t) &= \gh_n(t) l^{\gf_{n-1}(X)}(t) \\
& =\frac{\op{Max}(l^Y,J_{n,i})/\op{Max}(l^Y,J_{n-1,j})}{\op{Max}(l^X,J_{n,i})/\op{Max}(l^X,J_{n-1,j})}
\cdot \frac{\op{Max}(l^Y, J_{n-1,j})}{\op{Max}(l^X, J_{n-1, j})} \cdot l^X(t) \\
&= \frac{\op{Max}(l^Y,J_{n,i})}{\op{Max}(l^X,J_{n,i})} \cdot l^X(t).
\end{align*}
This completes the process of defining the maps $\gf_n$, for $n\geq 0$. 

Next we show that the maps $\gf_n$ converge uniformly to a homeomorphism $\gf: \D{R}^2 \to \D{R}^2$.  
To prove this, first note that 
when $0 \leq y \leq B_n(x)$, 
\begin{gather*}
|H_n(x, y) - (x, y)| = \left| (\gh_n(x) - 1) y \right| \leq 1/2^{n+1},
\end{gather*} 
and when $B_n(x) \leq y \leq 1$, 
\begin{gather*}
|H_n(x, y) - (x, y)| = \left| (1 - y)(\gh_n(x) - 1)\frac{B_n(x)}{1 - B_n(x)} \right| 
\leq \frac{1}{2^{n+1}} \frac{2/3}{1/3}= 1/2^n. 
\end{gather*} 
These imply that 
\begin{align*}
& \max \left \{ |\gf_n(x,y) - \gf_{n-1}(x,y)| \mid (x,y) \in \D{R}^2 \right\} \\
& \quad =\max \left \{ |H_n(x,y) - (x,y)|  \mid (x,y) \in \D{R}^2 \right\} \leq 1/2^n.
\end{align*}
Hence, the sequence $(\gf_n)_{n = 0}^{\infty}$ is uniformly Cauchy. In particular, the 
limiting map $\gf$ exists and is continuous. 

We claim that $\phi$ is injective. 
To see this, fix $x \in [0, 1]$. For any $n \geq 1$, and all $y$ and $y'$ in $[0, B_n(x)]$, we have 
\[|H_n(x, y) - H_n(x, y')| = \eta_n(x) |y - y'| \geq (1 - 2^{-(n+1)})|(x, y) - (x, y')|.\]
For $y$ and $y'$ in $[B_n(x), 1]$, we have 
\begin{align*}
|H_n(x, y) - H_n(x, y')| 
&= \frac{1 - \eta_n(x)B_n(x)}{1 - B_n(x)} |y - y'|  \\
&= \left( 1 - \frac{B_n(x)}{1 - B_n(x)}(1 - \eta_n(x)) \right) |y - y'| \\
&\geq \left( 1 - \frac{B_n(x)}{1 - B_n(x)} 2^{-(n+1)} \right)|y - y'| \\
&\geq ( 1 - 2^{-n} )|(x, y) - (x, y')|.
\end{align*}
Finally, since $H_n$ is a homeomorphism of ${\{x\} \times [0, 1]}$, for $0 \leq y' < B_n(x) < y \leq 1$, we have 
\begin{align*}
|H_n(x, y) - H_n(x, y')|
&= |H_n(x, y) - H_n(x, B_n(x))| + |H_n(x, B_n(x)) - H_n(x, y')| \\
&\geq (1 - 2^{-n})|(x, y) - (x, y')|.
\end{align*}
Combining the above inequalities, we conclude that for all $x$, $y$, and $y'$ in $[0, 1]$, as well as all $n \geq 0$, 
we have
\begin{align*}
|\phi_{n}(x, y) - \phi_n(x, y')|
&= | H_n \circ H_{n-1} \circ \cdots \circ H_1(x, y) - H_n \circ H_{n-1} \circ \cdots \circ H_1(x, y')| \\
&\geq (1 - 2^{-n})| H_{n-2} \circ \cdots \circ H_1(x, y) - H_{n-2} \circ \cdots \circ H_1(x, y')| \\
& \; \; \vdots \\
&\geq \left(\prod_{k = 1}^{n} (1 - 2^{-k})\right) |(x, y) - (x, y')|\\
&\geq \left(\prod_{k = 1}^{\infty} (1 - 2^{-k})\right)|y - y'|.
\end{align*}
We note that $\prod_{k = 1}^{\infty} (1 - 2^{-k}) > 0$. 
By virtue of the above inequality, $\phi$ is injective.

The map $\phi$ is surjective. To see this, we note that $\phi(C) \subseteq C$ is compact.
Therefore, if there is $z \in C \setminus \phi(C)$, then there must be $\gep > 0$ such that the disk of radius $\gep$ around $z$, 
$\D{D}(z,\gep)$, does not meet $\phi(C)$.
Choose $n \geq 0$ large enough so that $\| \phi_n - \phi \|_{\infty} < \gep$.
Then, we have $\phi(\phi_n^{-1}(z)) \in \D{D}(z, \gep)$, which is a contradiction.

Every $H_n$ (and therefore every $\phi_n$) sends vertical lines to vertical lines and is identity on the boundary of $[0, 1]^2$. 
Therefore $\phi$ also has this property. 
The convergence of $\gf_n$ to $\gf$ implies that for all $t\in C$, we have  
\[l^{\gf(X)}(t) = \lim_{n \to \infty} l^{\gf_n(X)}(t).\]

We claim that $\gf(X)=Y$.  To prove this, it is enough to show that $l^{\gf(X)}(t)=l^Y(t)$, for all $t\in C$. 

Fix $t \in C$. For each $n \geq 0$, choose $i_n$ with $t \in J_{n, i_n}$. 
By \refL{L:upper_sc_limit}, 
\[l^X(t) = \lim_{n \to \infty} \op{Max}(l^X, J_{n, i_n}), \quad l^Y(t) = \lim_{n \to \infty} \op{Max}(l^Y, J_{n, i_n}).\]
Therefore, if $l^X(t)\neq 0$, using \refE{E:incremental-fix}, we obtain 
\[ l^{\gf(X)}(t) = \lim_{n \to \infty} l^{\gf_n(X)}(t)
= \lim_{n \to \infty}  \left ( \frac{\op{Max}(l^Y, J_{n, i_n})}{\op{Max}(l^X, J_{n, i_n})} \cdot l^X(t) \right )
= \frac{l^Y(t)}{l^X(t)}\cdot l^X(t) = l^Y(t). \]

On the other hand, from the hypothesis, and using $\log (1+x) \leq x$ for $x\geq 0$, 
we note that for all $m \geq 1$, we have 
\begin{align*}
 \frac{\op{Max}(l^Y, J_{m, i_m})}{\op{Max}(l^X, J_{m, i_m})}
& = \prod_{n=1}^m  \frac{\op{Max}(l^Y,J_{n,i_n}) \op{Max}(l^X,J_{n-1,i_{n-1}})}
{\op{Max}(l^Y,J_{n-1,i_{n-1}}) \op{Max}(l^X,J_{n,i_n})} \\
& \leq \prod_{n=1}^m (1 + 1/2^{n+1}) \leq e^{1/2}. 
\end{align*}
Similarly, using $\log (1-x)\geq -2 x$ for $x\in (0, 1/4)$, we obtain 
\begin{align*}
 \frac{\op{Max}(l^Y, J_{m, i_m})}{\op{Max}(l^X, J_{m, i_m})}
 \geq \prod_{n=1}^m (1 - 1/2^{n+1}) \geq e^{-1}. 
\end{align*}
Thus, $l^X(t)=0$ if and only if $l^Y(t)=0$. However, since $\gf$ is a homeomorphism, $l^X(t)=0$ if and only if 
$l^{\gf(X)}(t)=0$. 
\end{proof}

\begin{proof}[Proof of \refT{T:straight-hairy-homeomorphic}]
First we note that it is enough to show that any two straight hairy Cantor sets based on the same Cantor set 
$C \subset \D{R}$ are ambiently homeomorphic. 
That is because, for any two Cantor sets $C$ and $C'$ in $\D{R}$, there is a homeomorphism 
$\gf : \D{R} \to \D{R}$ with $\gf(C) = C'$. 
The map $\gf$ may be extended to a homeomorphism of $\D{R}^2$ through $\gf(x, y) =  (\gf(x), y)$.
Given a straight hairy Cantor set $X$ based on $C'$, $\gf^{-1}(X)$ is a straight hairy Cantor set based on $C$.

Now consider two straight hairy Cantor sets $X$ and $Y$ based on the same Cantor set $C\subset \D{R}$ and 
with length functions $l^X$ and $l^Y$, respectively.
By \refP{P:straight-hairy-partitions-exist} there are nests of Cantor partitions 
$P_n^X= \cup_{i=1}^{k_n}[u^X_{n,i}, v^X_{n,i}]$ and $P_n^Y= \cup_{i=1}^{k_n} [u^Y_{n,i}, v^Y_{n,i}]$ 
shrinking to $C$, which enjoy the three properties in that proposition.
In particular, by properties (i) and (ii) in that proposition, there is a homeomorphism $\gy: \D{R} \to \D{R}$ 
such that for all $n\geq 0$ and all $1\leq i \leq k_n$, 
\[\gy(u^X_{n, i}) = u^Y_{n, i}, \quad \gy(v^X_{n, i}) = v^Y_{n, i}.\]
Evidently, $\gy$ maps $C$ onto $C$. We extend $\gy$ to a homeomorphism of $\D{R}^2$ through 
\[\gy(x, y) = (\gy(x), y).\]
It follows that $\gy^{-1}(Y)$ is a straight hairy Cantor set based on $C$. 
Moreover, by part (iii) of \refP{P:straight-hairy-partitions-exist}, the straight hairy Cantor sets $X$ and $\gy^{-1}(Y)$, as well as 
the nest of Cantor partitions $(P_n^X)_{n \geq 0}$ 
satisfy the hypothesis of \refP{P:straight-hairy-matching-partitions}. 
Therefore, we obtain a homeomorphism of $\D{R}^2$ which maps $X$ to $\gy^{-1}(Y)$. 
This completes the proof. 
\end{proof}

The following proposition is the main technical step towards the proof of \refT{T:straight-hairy-homeomorphic}. 
 
\begin{propo}\label{P:straight-AHCS-is-homeo-SHCS}
Let $C \subseteq \D{R}$ be a Cantor set, and let $l:C \to [0, \infty)$ be a function such that the set 
\[Z=\{(x, y) \in \D{R}^2 \mid x \in C, 0 \leq y \leq l(x) \}\]
satisfies the following properties: 
\begin{itemize}
\item[(i)] $Z$ is compact; 
\item[(ii)] $\{(x, l(x)) \mid x \in C, l(x)\neq 0\}$ is dense in $Z$.
\end{itemize}
Then, $Z$ is homeomorphic to a straight hairy Cantor set.
\end{propo}

\begin{proof}
We aim to modify $Z$, by successively applying homeomorphisms $H_n: C \times \D{R} \to C\times \D{R}$ close to the identity,
so that in the limit we obtain a straight hairy Cantor set. 
That is, we build a chain of homeomorphisms as in 
\begin{displaymath}
\xymatrix{
Z \ar[r]^{H_0} & Z_0 \ar[r]^{H_1}  & Z_1 \ar[r]^{H_2} & Z_2 \ar[r]^{H_3} & \dots \ar[r]^{H_{n-1}} 
& Z_{n-1} \ar[r]^{H_n} & Z_{n} \ar[r]^{H_{n+1}} & Z_{n+1} \ar[r]^{H_{n+2}} & \dots 
}.    
\end{displaymath}
Each $H_n$ is a piecewise translation, shuffling the arcs in $Z_{n-1}$ so that any bump function of $Z_{n}$ has jump 
discontinuities of sizes at most $1/n$. We present the details below. 

Since all Cantor sets in the real line are homeomorphic, without loss of generality, we may assume that $C$ is the 
middle-third Cantor set. 
Let $P_0 = [0, 1]$, and for $n\geq 1$, recursively 
define\footnote{Given $A \subseteq \D{R}$ and $k \in \D{R}$, we define $A + k = \{a + k \mid a \in A\}$ 
and $kA = \{ka \mid a \in A\}$.}
\[P_n = \frac{P_{n-1}}{3} \cup (\frac{2}{3} + \frac{P_{n-1}}{3}). \]
Then, $(P_n)_{n = 0}^{\infty}$ is a shrinking sequence of Cantor partitions for $C$.
For $n \geq 0$, let $P_n=\bigcup_{i = 1}^{2^n} I_{n, i}$, where each $I_{n, i}$ is a connected component of $P_n$.
We may label the subscripts so that $\sup I_{n,i} < \inf I_{n,j}$, whenever $0\leq i < j \leq 2^n$. 
For $p > q \geq 0$, and $1\leq j \leq 2^q$, consider the set 
\[K(p; q,j) = \{ 1 \leq i \leq 2^p \mid I_{p, i} \subseteq I_{q,j} \}.\]
We have 
\[K(p; q,j) = \{ i \in \D{N} \mid 2^{p - q} (j-1) + 1 \leq i \leq 2^{p-q}j\}.\]

We divide the remaining argument into three steps. 

\medskip

{\em Step 1.} There are a sequence of integers $m_0=0< m_1< m_2< m_3< \dots$, homeomorphisms 
$H_n: C\times \D{R} \to C \times \D{R}$, and functions $l_n: C \to [0, \infty)$, for $n \geq 0$, such that 
for all $n \geq 1$, and all $1 \leq j \leq 2^{m_{n-1}}$ we have  
\begin{itemize}
\item[(i)] 
\begin{equation*}
H_{n}(I_{m_{n-1}, j} \cap C) = I_{m_{n-1}, j} \cap C;
\end{equation*}
\item[(ii)] for all $x\in C$, $l_0(x)= l(x)$, and $l_n(x) = l_{n-1} \circ H_n^{-1}(x)$; 
\item[(iii)] for every $k\in \{\min K(m_n; m_{n-1},j),\max K(m_n;m_{n-1},j)\}$, 
\[\op{Max}(l_{n}, I_{m_n, k}) < 1/n;\]
\item[(iv)] for all $i$ with $\min K(m_n; m_{n-1},j) \leq i < \max K(m_n;m_{n-1},j)$, 
\[|\op{Max}(l_{n}, I_{m_n, i}) - \op{Max}(l_{n}, I_{m_n, i+1})| < 1/n.\]
\end{itemize}

\medskip 

Let $m_0=0$, $H_0$ be the identity map, and $l_0 \equiv l$. 
Fix an arbitrary $n\geq 1$. Assume that $m_i$, $H_i$, and $l_i$ are define for all $i \leq  n-1$, and we aim to define $m_n$, 
$H_n$, and $l_n$. 

Consider the set 
\[Z_{n-1} = \{(x, y) \in \D{R}^2 \mid x \in C, 0 \leq y \leq l_{n-1}(x) \}.\]
It follows that $H_{n-1} \circ H_{n-2} \circ \dots \circ H_0$ is a homeomorphism from $Z$ onto $Z_{n-1}$. 
Since $Z$ is compact, $Z_{n-1}$ must be also compact. This implies that $l_{n-1}$ is upper semi-continuous. 

By virtue of the homeomorphism $H_{n-1} \circ \dots \circ H_0$ from $Z$ onto $Z_{n-1}$, and the hypothesis of the proposition, 
the set $\{(x, l_{n-1}(x)) \mid x \in C, l_{n-1}(x) \neq 0 \}$ is dense in $Z_{n-1}$.
Therefore, for every $j$ with $1 \leq j \leq 2^{m_{n-1}}$, there is a finite set  
$A_j \subseteq I_{m_{n-1, j}} \cap C$, 
such that 
\begin{equation}\label{E:ASHCS2}
\D{D}(l_{n-1}(A_j), 1/5n) \supseteq [0, \op{Max}(l_{n-1}, I_{m_{n-1}, j})].
\end{equation}
By Lemma~\ref{L:upper_sc_limit}, for every $x\in A_j$ we have 
\[l_{n-1}(x) = \lim_{m \to \infty, x \in I_{m, i}} \op{Max}(l_{n-1}, I_{m, i}).\]
Therefore, we may find an integer $m_n > m_{n-1}$ such that for all $j$ with $1 \leq j \leq 2^{m_{n-1}}$, we have 
\begin{equation}\label{E:ASHCS1}
\begin{aligned}
\bigcup_{i\in K(m_n; m_{n-1}, j)} \D{D}\big(\op{Max}(l_{n-1}, I_{m_{n}, i}), 1/4n\big) 
\supseteq \D{D}(l_{n-1}(A_j), 1/5n).
\end{aligned}
\end{equation}

We define $H_n$ on each $I_{m_{n-1},j} \cap C$, for $1\leq j \leq 2^{m_{n-1}}$, so that it induces a 
homeomorphism of $I_{m_{n-1}, j} \cap C$. 
Fix such an integer $j$. 
By virtue of Equations \eqref{E:ASHCS2} and \eqref{E:ASHCS1}, there is a permutation 
\[\gs_j : K(m_n; m_{n-1}, j) \to K(m_n; m_{n-1},j)\] 
such that 
\begin{itemize}
\item  for $k \in \{\min K(m_n; m_{n-1}, j), \max K(m_n; m_{n-1},j)\}$, we have 
\[\op{Max}(l_{n-1}, I_{m_n, \gs_j(k)}) < 1/n,\]
\item for all integers $i$ with $\min K(m_n; m_{n-1}, j) \leq i < \max K(m_n; m_{n-1},j)$, we have
\[|\op{Max}(l_{n-1}, I_{m_n, \gs_j(i)}) - \op{Max}(l_{n-1}, I_{m_n, \gs_j(i+1)})| < 1/n.\]
\end{itemize}
For instance, to identify such $\gs_j$, one may first find a permutation $\gs_j'$ of $K(m_n; m_{n-1}, j)$ 
such that $\op{Max}(l_{n-1}, I_{m_n,\gs'(i)}) \leq \op{Max}(l_{n-1}, I_{m_n,\gs'(j)})$, whenever $i< j$ in $K(m_n; m_{n-1}, j)$. 
Then, compose $\gs_j'$ with the following permutation of $K(m_n; m_{n-1}, j)$, 
\begin{align*}
2^{m_n-m_{n-1}}(j-1)+(1, 2, 3, \dots , & 2^{m_n-m_{n-1}})  \\
& \mapsto 2^{m_n-m_{n-1}}(j-1)+(1, 3,5, \dots, 6, 4, 2).
\end{align*}

Once we have $\gs_j$, there is a unique homeomorphism $H_n$ such that for all $i$ in $K(m_n; m_{n-1},j)$
\[H_n(I_{m_n, i} \cap C) = I_{m_n, \gs_j^{-1}(i)} \cap C,\]
and each $H_n$ is a translation by a constant on each of $I_{m_n, i} \cap C$. 

Carrying out the above process for all $1 \leq j \leq 2^{m_{n-1}}$, we obtain a homeomorphism $H_n$. 
Then, we define 
\[l_n(x) = l_{n-1}(H_n^{-1}(x)), \quad \forall x\in C.\]
This completes the proof of Step 1. 

\medskip

Consider the homeomorphisms $\phi_n: C \times \D{R} \to C \times \D{R}$, defined according to 
\[\phi_0=H_0, \quad \phi_1= H_1 \circ H_0, \quad \phi_n= H_n \circ H_{n-1} \circ \dots \circ H_0, \; \forall n\geq 2.\]

\medskip

{\em Step 2.} The sequence $\phi_n$, for $n\geq 0$, uniformly converges to a homeomorphism 
\[\phi: C \times \D{R} \to C \times \D{R}.\]   

\medskip

By property (i) in Step 1, for all $n\geq 1$ and all $x \in C$, we have
\[|H_n(x) - x| \leq \max\{\diam(I_{m_{n-1}, j}) \mid 1 \leq j \leq 2^{m_{n-1}}\} = 3^{-m_{n-1}}.\]
Hence, 
\[\|\phi_{n} - \phi_{n-1} \|_{\infty} = \| H_{n} - \op{Id} \|_{\infty} \leq 3^{-m_{n-1}}.\]
Therefore, $\phi_n$ forms a Cauchy sequence, and converges to a continuous map $\phi: C \times \D{R} \to C\times \D{R}$. 

To prove that $\phi$ is surjective, note that $\phi(C) \subseteq C$ is compact.
Therefore, if there is $z$ in $C \setminus \phi(C)$, then there is $\gep > 0$ such that 
\[\D{D}(z, \gep) \cap C \subseteq C \setminus \phi(C).\]
Choose $n \geq 0$ so that $\| \phi_n - \phi \|_{\infty} < \gep$.
Then, we have $\phi(\phi_n^{-1}(z)) \in D_{\gep}(z)$, which is a contradiction.

To prove that $\phi$ is injective, fix arbitrary points $x \neq y$ in $C$.
There are $n > 0$ and $1 \leq i < i' \leq 2^{m_n}$ such that $x \in I_{m_n, i}$ and $y \in I_{m_n, i'}$.
By the above construction, there are distinct integers $j$ and $j'$ in $[1, 2^{m_n}]$ such that $\phi_n(x) \in I_{m_n, j}$ and 
$\phi_n(y) \in I_{m_n, j'}$.
On the other hand, for all $n' > n$, we have
\[H_{n'}(I_{m_n, j} \cap C) = I_{m_n, j} \cap C,\quad 
H_{n'}(I_{m_n, j'} \cap C) = I_{m_n, j'} \cap C. \]
Therefore, $\phi(x) = \cdots \circ H_{n+2} \circ H_{n+1} \circ \phi_n(x) \in I_{m_n, j}$ and similarly 
$\phi(y) \in I_{m_n, j'}$.
In particular, $\phi(x) \neq \phi(y)$.

\medskip

{\em Step 3.} The set $\phi(Z)$ is a straight hairy Cantor set. 

\medskip

Consider the function $l_* = l \circ \phi^{-1}: C \to [0, \infty)$, and the set 
\[Z_* = \{ (x, y) \in \D{R}^2 \mid x \in C, 0 \leq y \leq l_*(x) \}. \]
By Step 2, $\phi: Z \to Z_*$ is a homeomorphism.

We claim that for all $n \geq 0$ and $1 \leq j \leq 2^{m_n}$, we have
\[\op{Max}(l_*, I_{m_n, j}) = \op{Max}(l_n, I_{m_n, j}).\]
Recall that $l_n = l\circ \phi_n^{-1}$. To prove the above property, it is sufficient to show that 
$\phi^{-1}(I_{m_n,j} \cap C)= \phi_n^{-1}(I_{m_n,j} \cap C)$, or equivalently, 
$I_{m_n,j} \cap C = \phi \circ \phi_n^{-1}(I_{m_n,j} \cap C)$. 
However, by property (i) in Step 1, for all $n' > n$, $H_{n'}(I_{m_n,j} \cap C) = I_{m_n,j} \cap C$. 
Using $\phi \circ \phi_n^{-1} = \cdots \circ H_{n+2} \circ H_{n+1}$, we conclude that 
$\phi \circ \phi_n^{-1}$ preserves $I_{m_n,j} \cap C$.

Since $\phi: Z\to Z_*$ is a homeomorphism, the set $\{x \in C \mid l_*(x) \neq 0\}$ is dense in $C$.
This implies property (i) in \refD{D:strainght-hairy-Cantor}. 

Let $z$ be an end point of $C$. There are $p \geq 0$ and $1 \leq j_p \leq 2^{m_p}$ such that $z \in \partial I_{{m_p}, j_p}$. 
Then, for all $q> p$, $z \in I_{m_q, j_q}$ with $j_q \in \{\min K(m_q; m_{q-1}, j_{q-1}), \max K (m_q; m_{q-1}, j_{q-1})\}$. 
By property (iii) in Step 1, we must have 
\[\op{Max}(l_*, I_{m_q, j_q}) < 1/q.\]
Then, by \refL{L:upper_sc_limit} we have $l_*(z)= \lim_{q \to \infty} \op{Max}(l_*, I_{m_q, j_q}) = 0$.
This completes the proof of property (ii) in \refD{D:strainght-hairy-Cantor}.

Fix an arbitrary $z \in C$ with $l_*(z) > 0$, and let $\gep > 0$ be arbitrary.
Assume that $(j_n)_{n \geq 0}$ are the sequence of integers with $z \in I_{m_n, j_n}$, for all $n \geq 0$.
Using \refL{L:upper_sc_limit}, there exists $n > \max(2/\gep, 1/l_*(z))$ such that $3^{-m_{n-1}} < \gep$, and
\[|\op{Max}(l_*, I_{m_n, j_n}) - l_*(z)| < \gep/2.\]
On the other hand, since $\op{Max}(l_*, I_{m_n, j_n}) \geq l_*(z) > 1/n$, 
\[j_n \notin \{\min K(m_n; m_{n-1}, j_{n-1}), \max K(m_n; m_{n-1}, j_{n-1})\}.\] 
Therefore, 
\[I_{m_n, j_n-1} \cup I_{m_n, j_n+1} \subseteq I_{m_{n-1}, j_{n-1}},\]
and by property (iv) in Step 1, 
\[|\op{Max}(l_*, I_{m_n, j_n-1}) - \op{Max}(l_*, I_{m_n, j_n})| < 1/n,\]
\[|\op{Max}(l_*, I_{m_n, j_n}) - \op{Max}(l_*, I_{m_n, j_n+1})| < 1/n.\]
Let $\ga$ be a point in $I_{m_n, j_n-1} \cap C$ with $l_*(\ga) = \op{Max}(l_*, I_{m_n, j_n-1})$.
Then, $\ga < z$,
\[|\ga - z| < \diam(I_{m_{n-1}, j_{n-1}}) = 3^{-m_{n-1}} < \gep,\]
and 
\[|l_*(z) - l_*(\ga)| \leq |l_*(z)- \op{Max}(l_*, I_{m_n, j_n})| + |\op{Max}(l_*, I_{m_n, j_n})- l(\ga)| < \gep/2+1/n < \gep.\]
Similarly, we may find $\gb \in C$ such that
\[\gb > z,\quad |\gb - z| < \gep,\quad |l_*(\gb) - l_*(z)| < \gep.\]
This completes the proof of property (iii) in \refD{D:strainght-hairy-Cantor}.
\end{proof}

We will need the following Lemma for the proof of \refT{T:uniformization}. 

\begin{lem}\label{L:straight_HCS_is_HCS}
Let $C \subset \D{R}$ be a Cantor set, $l: C \to [0, +\infty)$, and 
\[X = \{(x, y) \in \D{R}^2 \mid x \in C, 0 \leq y \leq l(x) \}.\]
Then, $X$ is a straight hairy Cantor set if and only if the following properties hold: 
\begin{itemize}
\item[(i)] $X$ is compact;
\item[(ii)] the set $\{x \in C \mid l(x) \neq 0\}$ is dense in $C$;
\item[(iii)] any point in $\{ (x, y)\in X \mid x \in C, 0 < y < l(x) \}$ is not accessible from $\D{R}^2 \setminus X$.
\end{itemize}
\end{lem}

\begin{proof}
Assume that $X$ is a straight hairy Cantor set. 
By the uniqueness of the length function of $X$, $l^X$, we must have $l^X=l$. 
Thus, properties (i), (ii), and (iii) in the above lemma hold. 

Assume that $X$ satisfies the three properties in the lemma. We get property (i) in \refD{D:strainght-hairy-Cantor} 
for free. 
To see property (ii) in \refD{D:strainght-hairy-Cantor}, let $x \in C$ be an end point. If $l(x) \neq 0$, 
any $(x,y) \in X$ with $0 < y < l(x)$ will be accessible from $\D{R}^2 \setminus X$. 
This contradiction with property (iii) in the lemma shows that we must have $l(x)=0$. 

Now assume that $x\in C$ is not an end point. Define
\[l^+(x) = \limsup_{t \to x^+, t \in C} l(t),\quad l^-(x) = \limsup_{ t\to x^-, t \in C} l(t).\]
By the compactness of $X$, $l^+(x)$ and $l^-(x)$ are finite values. 
Moreover, we must have $l(x) \geq l^+(x)$ and $l(x) \geq l^-(x)$.  
If $l(x) > l^+(x)$, then any $(x, y)\in X$ with $l(x) > y > l^+(x)$ will be accessible from the right hand side. 
Similarly, if $l(x) > l^-(x)$, then any $(x, y)\in X$ with $l(x) > y > l^-(x)$ will be accessible from the left hand side. 
These contradict property (iii) in the lemma, so we must have $l(x) = l^+(x)= l^-(x)$.
Thus, we also have property (iii) in \refD{D:strainght-hairy-Cantor}.
\end{proof}

\begin{propo}\label{P:dense-ends}
Assume that 
\[X = \{(x, y) \in \D{R}^2 \mid x \in C, 0 \leq y \leq l(x) \}\] 
is a straight hairy Cantor set, where $C\subset \D{R}$ is a Cantor set, and $l:C \to [0,+\infty)$ is the length functions of $X$. 
Then, we have 
\begin{itemize}
\item[(i)] the set $\{l(x) \mid x\in C\}$ is dense in the interval $[0, \sup_{x\in C} l(x)]$; 
\item[(ii)] the set $\{(x, l(x)) \mid x\in C\}$ is dense in $X$. 
\end{itemize}
\end{propo}

\begin{proof}
As $X$ is compact, the value $m= \sup_{x\in C} l(x)$ is realised and is finite. 
Assume in the contrary that there is an open interval $(a,b) \subset [0,m]$ such that for all $x\in C$, $l(x) \notin (a,b)$. 
Let $C'= \{x\in C \mid l(x) \geq b\}$.
Since $X \cap \{(x, y) \in \D{R}^2 \mid y \geq b\}$ is compact, $C'$ must be compact. 

By property (ii) in \refD{D:strainght-hairy-Cantor} every $x \in C'$ is not an end point of $C$. 
Thus, by property (iii) in that definition, for every $x\in C'$ there is $x' \in C$ such that $x'< x$ and $|l(x') - l(x)| < (b-a)/2$. 
Since $l(x)\geq b$, and $l(x') \notin (a,b)$, we must have $l(x') \geq b$. 
Thus, $x'\in C'$. 
In other words, for every $x\in C'$, there is $x'\in C'$ with $x' < x$. 
This is not possible, since $C'$ is a compact set. 

To prove item (ii), let $(x,y)$ be an arbitrary point in $X$, and fix an arbitrary $\gep>0$. 
There are $\gd_1$ and $\gd_2$ in $(0, \gep)$ such that $(x-\gd_1, x+ \gd_2) \cap C$ is a Cantor set in $\D{R}$. 
Let $D= (x-\gd_1, x+ \gd_2) \cap C$. 
From \refD{D:strainght-hairy-Cantor}, one may see that the set 
\[\{(x, y) \in \D{R}^2 \mid x \in D, 0 \leq y \leq l(x) \},\] 
is a straight hairy Cantor set. 
As $(x,y)$ belongs to the above set, $y\leq \sup_{t\in D} l(t)$. 
Applying item (i) to the above straight hairy Cantor set, we conclude that there is $x'\in D$ with $|l(x')-y| < \gep$. 
Therefore, 
\[|(x,y)- (x', l(x'))| \leq |x-x'|+ |y- l(x')| \leq 2 \gep. \qedhere\] 
\end{proof}

%%%%%%%%%%%%%%%%%%%%%%%%%%%%%%%%%%%%%%%%%%%%%%%%%%%%%

\section{Height functions and base curves}\label{S:height-base}
In this section we establish some topological features of hairy Cantor sets. 
These will be used in \refS{S:uniformisation} to prove \refT{T:uniformization}.

Recall that a curve $\gamma: [0, 1) \to \D{R}^2$ \textbf{lands} at $z \in \D{R}^2$, if $\lim_{t\to 1} \gamma(t)$ exists 
and is equal to $z$. A point $z \in \D{R}^2$ is called \textbf{accessible} from $\gO \subset \D{R}^2$, 
if there is a curve $\gamma: [0, 1) \to \gO$ which lands at $z$. 

Given a set $X \subset \D{R}^2$ satisfying axioms $A_1$ to $A_4$ in the introduction, the \textbf{base} of $X$ is the unique 
Cantor set $B \subset X$ which contains all point components of $X$. 
We shall often say that $X$ is \textbf{based} on the Cantor set $B$. 

Any arc \footnote{A Jordan arc, or simply an arc, in $\D{R}^2$ is the image of a continuous and injective map 
$\gamma:[0,1] \to \D{R}^2$. The \textit{end points} of this Jordan arc are $\gamma(0)$ and $\gamma(1)$.} 
component $c$ of $X$ has two distinct end points, one of which belongs to $B$. 
The end point of $c$ which belongs to $B$ is called the \textbf{base} of $c$. 
The other end point of $c$, which does not belong to $B$, is called the \textbf{peak} of $c$. 
If $c$ is a point component of $X$, we may either refer to $c$ as the base of $c$ or as the peak of $c$.
For $x\in X$, the \textbf{base} of $x$ is defined as the base of the component of $X$ containing $x$, 
and similarly, the \textbf{peak} of $x$ is defined as the peak of that component. 
We may define the \textbf{base map} and the \textbf{peak map}
\[b: X \to B, \quad p: X \to X,\]
where $b(x)$ is the base of $x$ and $p(x)$ is the peak of $x$. 

Slightly abusing the notation, we say that $x\in X$ is a \textbf{base point}, if $b(x)=x$, and we say that $x\in X$ 
is a \textbf{peak point} if $p(x)=x$.
Since $b$ and $p$ are the identity map on $B$, any element of $B$ is both a base point, and a peak point. 
It follows from axiom $A_4$ that $b$ is continuous on $X$. However, $p$ is far from continuous. 

The base points of $X$ and the peak points of $X$ are defined solely using the topology of $X$.
It follows that any homeomorphism of a hairy Cantor set must send base points to base points, 
and peak points to peak points.

Assume that $X \subset \D{R}^2$ satisfies axioms $A_1$ to $A_4$, and $B$ is the base Cantor set of $X$. 
A function $h: X \to [0, +\infty)$ is called a \textbf{height function} on $X$, if the following properties hold:
\begin{itemize}
\item[(i)] h is continuous on $X$;
\item[(ii)] $B=h^{-1}(0)$;
\item[(iii)] $h$ is injective on any connected component of $X$. 
\end{itemize}
For instance, for a straight hairy Cantor set $X$, $h(x,y)=y$ is a height function on $X$.
However, there are many other height functions on $X$.
For the proof of \refT{T:uniformization} we require a height function on a given hairy Cantor set. 
To this end we employ a general theory of Whitney maps, which we explain below. 

Assume that $K$ is a compact metric space, and let $\op{CL}(K)$ denote the set of all non-empty compact 
subsets of $K$. The space $\op{CL}(K)$ may be equipped with the \textbf{Hausdorff topology} induced from the 
Hausdorff metric. That is, given non-empty compact sets $K_1$ and $K_2$ in $K$, the Hausdorff distance between 
$K_1$ and $K_2$ is defined as the infimum of all $\gep > 0$ such that $\gep$ neighbourhood of $K_1$ contains $K_2$ 
and $\gep$ neighbourhood of $K_2$ contains $K_1$. 
A function $\gm: \op{CL}(K)\to [0, +\infty)$ is called a \textbf{Whitney map} for $K$, if 
\begin{itemize}
\item[(i)] $\gm$ is continuous on $\op{CL}(K)$;
\item[(ii)] $\mu(\{x\}) = 0$ for every $x \in K$;
\item[(iii)] for every $K_1$ and $K_2$ in $\op{CL}(K)$ with $K_1 \subsetneq K_2$, we have $\mu(K_1) < \mu(K_2)$. 
\end{itemize}

In \cite{Wh33}, Whitney proves the existence of maps $\gm$ satisfying the above properties. 
The existence of a Whitney map for a compact subset of an Euclidean space is a classical result in topology. 
One may refer to the general reference \cite[Excercise 4.33]{Nadler92}, or the extensive monograph on the hyperspaces 
\cite{IlNa99}.

\begin{propo}\label{P:X-admits-height}
Any compact metric space $X$ which satisfies axioms $A_1$ to $A_4$ admits a height function.
\end{propo}

\begin{proof}
Fix $X$, and assume that $B$ denotes the base Cantor set of $X$. 
Let $\mu$ be a Whitney map for $X$, that is, $\gm: \op{CL}(X) \to [0, +\infty)$ enjoys the above 
properties. Consider the function $h: X \to [0, +\infty)$, define as 
\[h(x) = \mu([b(x), x]),\] 
where $[b(x), x]$ denotes the unique arc in $X$ which connects $x$ to the base point $b(x) \in B$. 

To see the continuity of $h$ on $X$, let $x_i \to x$ within $X$. 
By axiom $A_4$, $[b(x_i), x_i] \to [b(x), x]$, in the Hausdorff topology. 
Thus, by the continuity of $\gm$ on $\op{CL}(X)$, we conclude that $h(x_i) \to h(x)$. 

By property (ii) of Whitney maps, 
\[h^{-1}(0) = \{ x\in X \mid b(x)=x \} = B.\]

By property (iii) of Whitney maps, $h$ is injective on every connected component of $X$. 
\end{proof}

The existence of the height function in the above proposition is sufficient to prove \refT{T:uniformization-abstract}. 
A reader interested only in that theorem may directly go to the proof of \refT{T:uniformization-abstract} at the end of 
\refS{S:uniformisation}. 
The remaining statements in this section are required for the proof of \refT{T:uniformization}. 
So, from now on we assume that $X$ is a subset of $\D{R}^2$. 

Let $X$ be a hairy Cantor set with base Cantor set $B$. A Jordan curve $\gamma \subset \D{R}^2$ is called 
a \textbf{base curve} for $X$, if $X \cap \gamma=B$, and the bounded connected component of 
$\D{R}^2 \setminus \gamma$ does not intersect $X$. 
In particular, the closure of the unbounded component of $\D{R}^2\setminus \gamma$ contains $X$. 

We aim to prove that any hairy Cantor set admits a base curve. 
To this end we need to introduce some basic notions and prove some technical lemmas. 

By a \textbf{topological disk} in the plane we mean a connected and simply connected open subset of the plane. 
Let $X$ be a hairy Cantor set with base Cantor set $B \subset X$. 
We say that an open set $U \subseteq \D{R}^2$ is \textbf{admissible} for $X$, if the following properties hold: 
\begin{itemize}
\item[(i)] $U$ consists of a finite number of pairwise disjoint topological disks, and $B \subseteq U$;
\item[(ii)] for all $x \in X \cap U$, $[b(x), x] \subseteq U$;
\item[(iii)] for all $x\in \partial U \setminus X$, $x$ is accessible from both of $U\setminus X$ and 
$\D{R}^2 \setminus \ol{U}$; 
\item[(iv)] if $U_1$ and $U_2$ are distinct connected components of $U$, then $\ol{U_1} \cap \ol{U_2} = \emptyset$. 
\end{itemize}

Evidently, any finite union of pairwise disjoint Jordan domains which contains $X$ is an admissible set for $X$.
However, it is not immediately clear that nontrivial admissible sets for $X$ exist. 
Over the next few pages we aim to show that any hairy Cantor set admits a nest of admissible sets shrinking to its base 
Cantor set. This will be employed to build a base curve for $X$. 

\begin{lem}\label{L:disjoint-closures}
Let $X$ be a hairy Cantor set with base Cantor set $B$. Assume that $U$ is an open set which satisfies properties 
(i)-(iii) in the definition of 
admissible sets for $X$, and let $U_i$, for $1 \leq i\leq n$, be the connected components of $U$. 
There are topological disks $V_i \subseteq U_i$, for $1 \leq i \leq n$, such that $\cup_{1 \leq i \leq n} V_i$ is an 
admissible set for $X$. 
\end{lem}

\begin{proof}
If $n=1$, we let $V_1=U_1$, and there is nothing to prove. Below we assume that $n \geq 2$.
Let $J= \{j\in \D{Z} \mid 1 \leq j \leq n\}$.  
We claim that for distinct $i$ and $j$ in $J$, 
\[(\ol{X \cap U_i}) \cap (\ol{X \cap U_{j}}) = \emptyset.\] 
To prove this, it is enough to show that for every $i\in J$, if $x \in \ol{X \cap U_i}$ then $b(x) \in U_i$. 
Fix an arbitrary $x \in \ol{X \cap U_i}$. 
There is a sequence $x_l$ in $X \cap U_i$, for $l \geq 1$, which converges to $x$. 
By property (ii) of admissible sets, $b(x_l) \in U_i$, and hence $b(x)=\lim_{l \to \infty} b(x_l) \in \ol{U_i}$. 
Since $B \cap \partial U_i=\emptyset$, we must have $b(x) \in U_i$. 

%Recall that each $U_i$ is a topological disk. By a chord in $U_i$ we mean the image of a continuous map 
%$\gga: (0,1) \to U_i$ such that $\lim_{t\to 0} \gga (t)$ and $\lim _{t\to 1} \gga(t)$ exist and belong to $\partial U_i$. 
%Any chord of $U_i$ decomposes $U_i$ into two connected components. Below we refer to such components as 
%chordal domains in $U_i$. 

In two steps, we shrink each component $U_i$ of $U$ so that the closures of distinct components become disjoint.

Fix an arbitrary $i\in J$.
If $\partial U_i \cap (\ol{X \cap U_j})= \emptyset$ for all $j\neq i$, we let $W_i=U_i$. 
Otherwise, first note that the sets $\ol{X \cap U_i}$ and 
$\cup_{j \in J \setminus\{i\}}(\ol{X \cap U_j}) \cap \partial U_i$ are compact, and by the above paragraph, are disjoint. 
Let $P_i$ be a finite union of Jordan domains such that $P_i$ contains 
$\cup_{j \in J \setminus\{i\}}(\ol{X \cap U_j}) \cap \partial U_i$, but $\ol{P_i}$ does not meet $\ol{X \cap U_i}$. 
Then, define, 
\[W_i = U_i \setminus \ol{P_i}.\]
Repeating the above process, we obtain topological disks $W_i$ for all $i\in J$. 
Note that since for all $i \in J$, $\ol{P_i} \cap (\ol{X \cap U_i})= \emptyset$, we have $W_i \cap X = U_i\cap X$. 
With these modifications, for distinct values of $i$ and $j$ in $J$, we have 
\[\partial W_i \cap (\ol{X \cap W_j}) = \emptyset.\]
That is because, the intersection must lie on $\partial W_j$, but 
$\partial W_j \cap (\ol{X \cap W_j}) \subseteq P_i$ and $\partial W_i \cap P_i =\emptyset$.

Fix an arbitrary $i\in J$. If $\partial W_i$ does not meet $\partial W_j$ for all $j \neq i$, we let $V_i = W_i$. 
If there is $j\neq i$ such that $\partial W_i \cap \partial W_j \neq \emptyset$, we further modify $W_i$ as follows. 
By the above paragraph, the sets $\partial W_i \cap ( \cup_{j\in J} \partial W_j)$ and 
$\cup_{j\in J} (\ol{X \cap W_j})$ are disjoint. 
Then, let $Q_i$ be a finite union of Jordan domains such that $Q_i$ contains 
$\partial W_i \cap ( \cup_{j\in J} \partial W_j)$, but $\ol{Q_i}$ does not meet $\cup_{j\in J} (\ol{X \cap W_j})$. 
Then, define 
\[V_i = W_i \setminus Q_i.\]
It follows that $V_i$ is connected and simply connected. 
Repeating the above process for all $i$ we obtain topological disks $V_i$ for all $i\in J$. 

By construction, $\ol{V_i} \cap \ol{V_j}=\emptyset$, for distinct $i$ and $j$. 
Moreover, since $\ol{Q_i} \cap (\ol{X\cap W_i})= \emptyset$, we must have 
$V_i \cap X = W_i \cap X= U_i \cap X$. This implies that $\cup_{i \in J} V_i$ contains $B$ and satisfies 
property (ii) of admissibility. 
As $U_i$ satisfies property  (iii), and each of $P_i$ and $Q_i$ is a finite union of 
Jordan domains, $V_i$ also satisfies property (iii). 
\end{proof}

\begin{lem}\label{L:simply_connected_to_admissible}
Let $X$ be a hairy Cantor set with base Cantor set $B$. For any $\gep > 0$, there is an admissible set for 
$X$ which is contained in $\D{D}(B, \gep)$.
\end{lem}

\begin{proof}
Fix an arbitrary $\gep>0$. 
Let $U$ be a finite union of Jordan domains such that $B \subset U$, $U \subset \D{D}(B, \gep)$, 
and the closures of the connected components of $U$ are pairwise disjoint. 
Below we modify $U$ so that it becomes an admissible set for $X$. 

By \refP{P:X-admits-height} there is a height function $h:X \to [0, + \infty)$.
Since $U$ is open, $B \subset U$, and $h^{-1}(0)=B$, there must be $t > 0$ such that 
$h^{-1}([0,t)) \subset U$. Let us define 
\[X' = \{ x \in X \mid h(x) \geq t \}.\]
This is a closed set in $\D{R}^2$ with $X' \cap B = \emptyset$. 
Define 
\[V = U \setminus X'.\]
Let $Q$ denote the union of $V$ and the bounded components of $\D{R}^2 \setminus V$. 
Let $W$ denote the components of $Q$ which meet $B$. 
Then, ($V$ and hence) $W$ is contained in $U$, and every component of $W$ is simply connected. 
Also, since $W$ covers the compact set $B$, $W$ has a finite number of connected components. 

Fix an arbitrary $x \in X \cap W$. 
First assume that $x \in X \cap V$. Then $h(x) < t$, and hence 
$[b(x), x] \subset h^{-1} ([0, t)) \subset U \setminus X'= V$. 
In particular, since $[x,b(x)]$ is connected and meets $B$, $[x,b(x)]$ is contained in $W$. 
Now assume that $x \in X \cap (W \setminus V)$. 
Note that $x \notin V$, $b(x) \in V$, and for every $y\in [x,b(x)]\cap V$, $[y,b(x)] \subset V$. 
This implies that there is $z \in [x,b(x)]$ such that $[x,b(x)] \cap V= [z,b(x)] \setminus \{z\}$. 
Now, $[x,z] \cap V= \emptyset$, and since $x \in W$, the connected set $[x,z]$ must be contained in the bounded 
component of $\D{R}^2 \setminus V$. 
Thus, $[x,b(x)] = [x,z] \cup ([z, b(x)]\setminus \{z\}) \subset W$. 

Note that $U \setminus X \subseteq V \subseteq Q \subseteq U$. 
As $X$ has no interior points, we have $\ol{U \setminus X} = \ol{Q}=\ol{U}$. 
On the other hand, $U \setminus Q \subseteq X$, which implies 
$(\D{R}^2 \setminus Q) \setminus X=(\D{R}^2 \setminus U) \setminus X$. 
Combining these we obtain 
\begin{align*}
\partial Q \setminus X &= (\ol{Q} \cap (\ol{\D{R}^2 \setminus Q})) \setminus X \\
&=\ol{Q} \cap ((\D{R}^2 \setminus Q) \setminus X) =\ol{U} \cap ((\D{R}^2\setminus U) \setminus X) 
=\partial U \setminus X.
\end{align*}
The set $\partial U$ is a finite union of Jordan curves, so it is locally connected. It follows that every point in 
$\partial U \setminus X$ 
is accessible from both $U \setminus X$ and $\D{R}^2 \setminus U$. 
By the above equation, $\partial Q \setminus X = \partial U \setminus X$. 
This implies that every point in $\partial Q \setminus X$ enjoys the same property. 
Since $W$ consists of a finite number of the components of $Q$, we conclude that every point 
in $\partial W \setminus X$ enjoys that property. 

So far we have shown that $W$ satisfies properties (i)--(iii) in the definition of admissibility for $X$. 
If necessary, we employ \refL{L:disjoint-closures} to shrink each component of $W$ so that it becomes 
an admissible set for $X$. 
\end{proof}

We need the following general feature of admissible sets.   

\begin{lem}\label{L:properties_of_admissible-2}
Let $U$ be an admissible set for a hairy Cantor set $X$. Then, for any connected component $V$ of $U$, 
$V \setminus X$ is connected.
\end{lem} 
 
We shall use a criterion of Borsuk \cite[Chapter~XI, Theorem~3.7]{ES52} in the proof of the above lemma.
That is, for a closed set $Y$ in the Riemann sphere $\hat{\D{C}}$, the set $\hat{\D{C}} \setminus Y$ is connected 
if and only if every continuous map from $Y$ to the unit circle $\D{S}^1= \{(x,y) \in \D{R}^2 \mid x^2 + y^2=1\}$ 
is null-homotopic. 
Recall that a continuous map $g : Y \to \D{S}^1$ is called null-homotopic, if $g$ is homotopic to a constant map. 
For instance, this criterion may be used to show that the complement of a Cantor set in the plane is connected. 
To see this, let $C$ be a Cantor set in the plane, and $g : C \to \D{S}^1$ be a continuous map.  
As $C$ is compact, $g$ is uniformly continuous. Therefore, we may decompose $C$ into a finite number of disjoint 
Cantor sets $C_i$ such that the image of $g$ on each $C_i$ is contained in an arc of $\D{S}^1$ of length at most $\pi$. 
It follows that on each $C_i$, $g$ is homotopic to a constant map. Combining these homotopies, one obtains a homotopy 
from $g$ to a constant map on all of $C$. 

\begin{proof}
Fix an arbitrary hairy Cantor set $X$. Let $U$ be an admissible set for $X$, and $V$ be a connected component of $U$. 
As $X$ is compact, without loss of generality we may assume that $V$ is bounded in the plane, that is, $\ol{V}$ 
is compact. 
Define $X_1$ as the union of the connected components of $X$ which intersect both $V$ and $\partial V$. 
Every connected component of $X_1$ is a Jordan arc. 
Define $X_2$ as the union of the connected components of $X$ which are contained in $V$.
Note that $V \setminus X= (V \setminus X_1) \setminus X_2$. 
We shall first show that $V \setminus X_1$ is connected, and simply connected, and then show that 
$(V \setminus X_1) \setminus X_2$ is connected. 
Mostly due to perhaps complicated topological structure of $\partial V$ relative $X$ the argument is rather long. 
We shall deal with $V \setminus X_1$ in Steps 1--4 below, and then deal with $(V \setminus X_1) \setminus X_2$ in 
Steps 5--7. The main idea in both cases is to work in a quotient space where the boundary of $V$ is 
identified to a single point.  

Let $B \subset X$ be the base Cantor set of $X$. 
By \refP{P:X-admits-height}, we may choose a height function $h$ on $X$, which by rescaling, if necessary, 
we may assume that $h: X \to [0,1]$. 
Let $\gY: B \to C_0 \subset \D{R}$ be a homeomorphism from $B$ to a Cantor set $C_0$ in $\D{R}$. 
We may extend $\gY$ onto $X$, using the height function $h$, according to 
\[\gY: X \to C_0 \times[0, 1], \quad \gY(z) = (\gY(b(z)), h(z)).\]
%There is $u_0: C_0 \to [0,1]$ such that 
%\[\gY(X)= \{(x,y) \in \D{R}^2 \mid x\in C_0, 0 \leq y \leq u_0(x)\}.\]
By the continuity of the base map $b$ and properties of height functions, $\Psi$ is continuous and injective on $X$. 
As $X$ is compact, $\gY$ is a homeomorphism from $X$ to $\gY(X)$.
This provides us with a rather simple coordinate system on $X$, which we shall use to build homotopies of maps 
on subsets of $X$. 

\medskip

\textit{Step 1: The set $X_1$ is closed in $\D{R}^2$.}

\smallskip

Let $x_i$, for $i \geq 1$, be a sequence in $X_1$ that converges to some $x \in \ol{X_1}$. 
We need to show that $b^{-1}(b(x)) \cap \partial V \neq \emptyset$ and $b^{-1}(b(x)) \cap V \neq \emptyset$. 
Since each component of $X_1$ meets $\partial V$, there is $y_i \in X_1 \cap \partial V$ with $b(y_i)=b(x_i)$, for each 
$i \geq 1$. 
As $\partial V$ is compact, passing to a subsequence if necessary, we may assume that $y_i$ converges to some 
$y \in \partial V$.
By the continuity of $b$, we have 
\[b(x) = \lim_{i \to \infty} b(x_i) = \lim_{i \to \infty} b(y_i) = b(y).\] 
In particular, as $b(x_i) \in V$, $b(x) \in \ol{V}$. However, by properties of admissible sets, $b(x)\in B \subset U$, 
which implies that $b(x) \notin \partial V$. Hence, we must have $b(x) \in V$, and therefore 
$b^{-1}(b(x)) \cap V \neq \emptyset$. 
The above equation also implies that $y \in b^{-1}(b(x)) \cap \partial V$, thus, 
$b^{-1}(b(x)) \cap \partial V \neq \emptyset$. This completes Step 1. 

\medskip

Let 
\[C_1 = \gY(b(X_1)) \subseteq C_0.\]
This is a closed set in $\D{R}$. 
By property (iv) of admissible sets, for any connected component $\gga$ of $X_1$, there is 
$z \in \gamma \cap \partial V$ such that $V \cap \gamma = [b(z), z] \setminus \{z\}$.
It follows that there is a function $u_1: C_0 \to [0,1]$ such that
\[\gY(V \cap X_1) = \{(x, y) \in \D{R}^2 \mid x \in C_1, 0 \leq y < u_1(x) \}.\]
%Note that $(C_1 \times\{1\}) \cap \Psi(V \cap X_1) = \emptyset$.

Consider the equivalence relation $\sim_1$ on $\hat{\D{C}}$ defined according to $z \sim_1 z'$ if and only if 
$z = z'$, or both $z$ and $z'$ belong to $\hat{\D{C}} \setminus V$.
The quotient space $\hat{\D{C}}/\mathord\sim_1$ is the one-point compactification of the 
connected and simply connected open set $V$.
Hence, it is a topological sphere, which may be identified with $\hat{\D{C}}$.
Let $\pi_1 : \hat{\D{C}} \to \hat{\D{C}}$ be the quotient map, that is, 
\[\pi_1(z)= [z]_{\mathord \sim_1}.\]
Define
\[v_1 = \pi_1(\hat{\D{C}}\setminus V), 
\quad K_1= \pi_1((\hat{\D{C}}\setminus V) \cup X_1) = \{v_1\} \cup \pi_1(V \cap X_1).\]
Then $v_1$ is a single point in $\hat{\D{C}}$, and $K_1$ is the corresponding one-point compactification of 
$X_1 \cap V$, and is a compact set in $\hat{\D{C}}$. 

Consider the map $\gF_1 : C_1 \times [0, 1]  \to K_1$, 
\[\Phi_1(x,y) = \begin{cases}
\pi_1(\Psi^{-1}(x,y))  	& \tif    y < u_1(x),\\
v_1							        & \tif   y \geq u_1(x).
\end{cases}\]

%The map $\gF_1$ is part of the following commutative diagram
%\begin{displaymath}
%\xymatrix{ X_1 \ar[r]^{\gY} \ar[d]_{\text{inclusion}}  & C_1 \times [0,1] \ar[d]^{\gF_1} \\
%                   (\hat{\D{C}}, X_1) \ar[r]^{\pi_1} &  (\hat{\D{C}}, K_1) }
%\end{displaymath}

\medskip

{\em Step 2. The map $\gF_1$ is continuous, and $\gF_1: \gY(V \cap X_1) \to K_1 \setminus \{v_1\}$ is 
a homeomorphism.}

\smallskip

Let $A_1=\gY(X_1)$ and $A_2=C_1 \times [0,1] \setminus \gY(V \cap X_1)$. 
By Step 1, $X_1$ is closed, which implies that $A_1= \gY(X_1)$ is closed in $C_1 \times [0,1]$.  
We claim that $A_2$ is also closed in $C_1 \times [0,1]$. 
To see this, first note that $X_1 \setminus V$ is closed in $X_1$, and $\gY$ is a homeomorphism on $X_1$, 
which imply that $\gY(X_1 \setminus V)$ is closed. 
Since $A_2 \cap \gY(X_1)= \gY(X_1\setminus V)$, this set must be also closed, and hence compact. 
Using the compactness of $\gY(X_1)$ as well, we note that 
\begin{align*}
\ol{A_2} \cap \gY(V \cap X_1) &= \ol{A_2} \cap ( \gY (X_1) \cap \gY(V \cap X_1)) \\
&= (\ol{A_2} \cap \gY(X_1)) \cap \gY(V \cap X_1)\\
&= (A_2 \cap \gY(X_1)) \cap \gY(V \cap X_1) \\
&= A_2  \cap \gY(V \cap X_1)= \emptyset. 
\end{align*}
Hence, $\ol{A_2}= A_2$, and therefore, $A_2$ is closed in $C_1 \times [0,1]$. 

The map $\gF_1$ is continuous on both $A_1$ and $A_2$. That is because, $\pi_1$ is continuous on $\hat{\D{C}}$, 
and $\gY^{-1}$ is continuous on $A_1$. On th	e set $A_2$, $\gF_1$ is a constant map only taking value $v_1$. 
It follows that $\gF_1$ is continuous on $A_1 \cup A_2=C_1 \times [0,1]$. 

As $\gY: X \to \gY(X)$ is a homeomorphism, $\gY^{-1}$ is a homeomorphism on $\gY(V\cap X_1)$, which takes 
values in $V$. The map $\pi_1$ is also a homeomorphism on $V$. Therefore, $\gF_1$ is a homeomorphism from 
$\gY(V \cap X_1)$ onto $\pi_1 (V \cap X_1)=K_1 \setminus \{v_1\}$. 

\medskip

\textit{Step 3. Any continuous map $g: K_1\to \D{S}^1$ is null-homotopic.}

\smallskip

Consider  $H: (C_1 \times [0, 1]) \times [0, 1] \to C_1 \times [0, 1]$, defined as 
\[H((x, y), t) = (x, \max(y, t)).\]
This is a continuous map with $H((x,y),0)= (x,y)$ and $H((x,y), 1)= (x,1)$. 
Let $g: K_1 \to \D{S}^1$ be a continuous map. 
Define the map $G: K_1 \times [0, 1] \to \D{S}^1$ as 
\[G(w, t) = \begin{cases}
g(\Phi_1(H(\Phi_1^{-1}(w), t)))		&\tif	w \neq v_1,\\
g(v_1)				                   					&\tif	w = v_1.
\end{cases}\]
For all $w \in K_1$, $G(w, 0) = g(w)$ and $G(w, 1) = g(v_1)$. We need to show that $G$ is continuous. 

By Step 2, $G$ is continuous on $(K_1\setminus\{v_1\}) \times [0, 1]$. On the set $\{v_1\} \times [0, 1]$, $G$ is 
constant, and hence continuous. 
To prove that $G$ is continuous, it is sufficient to show that for any convergent sequence $(w_i, t_i)$ 
in $(K_1\setminus\{v_1\}) \times [0, 1]$ with $w_i \to v_1$, we have $G(w_i, t_i) \to g(v_1)$.

Let $(x_i,y_i)= \Phi_1^{-1}(w_i)$, for $i \geq 1$. 
Since $w_i \to v_1$, by Step 2, the sequence $(x_i,y_i)$ may not have a limit point in 
$\Phi_1^{-1}(K_1\setminus \{v_1\})$.
Hence, \footnote{Given compact sets $A$ and $B$ in $\D{R}^2$, we define 
$d(A, B) = \min\{d(x, y) \mid x\in A, y \in B\}.$}
\[\lim_{i \to \infty} d((x_i,y_i), \Phi_1^{-1}(v_1)) = 0.\]
As $\gF_1^{-1}(v_1)$ is compact, for every $i \geq 1$, there is $(x_i',y_i')$ in $\Phi_1^{-1}(v_1)$ with 
\[d((x_i,y_i), \Phi_1^{-1}(v_1)) = d((x_i,y_i), (x_i',y_i')).\]
Recall that 
\[\Phi_1^{-1}(v_1) = \{(x, y) \in \D{R}^2 \mid x \in C_1, u_1(x) \leq y \leq 1\}.\]
In particular, $(x_i',y_i') \in \Phi_1^{-1}(v_1)$ implies that $H((x_i',y_i'), t_i) \in \Phi_1^{-1}(v_1)$. 
Therefore, since $H(\cdot, t_i)$ is not expanding distances,
\[d(H((x_i,y_i), t_i), \Phi_1^{-1}(v_1)) \leq d(H((x_i,y_i), t_i), H((x_i',y_i'), t_i))  \leq d((x_i,y_i), (x_i',y_i')).\]
Therefore, 
\[\lim_{i \to \infty} d(H((x_i,y_i), t_i), \Phi_1^{-1}(v_1))= 0.\]
By the continuity of $\gF_1$ and $g$, the above equation implies that $G((x_i,y_i), t_i) \to g(v_1)$. 
This completes Step 3. 

\medskip

\textit{Step 4. The set $V \setminus X_1$ is connected, and simply connected.}

\smallskip

As $\pi_1$ is injective and continuous on $V$, and $V$ is open, $\pi_1: V \to \hat{\D{C}} \setminus \{v_1\}$ 
is a homeomorphism. 
It follows that $V \setminus X_1$ is connected if and only if $\pi_1(V\setminus X_1)= \hat{\D{C}} \setminus K_1$ 
is connected.
By Borsuk's criterion, discussed before the proof, and Step 3, $\hat{\D{C}} \setminus K_1$ is connected. 
Thus, $V \setminus X_1$ is connected.

Let $K = (\hat{\D{C}} \setminus V) \cup X_1$.
If $V \setminus X_1$ is not simply connected, there is a closed curve $\gamma$ in $V \setminus X_1$ which is 
not null-homotopic in $V \setminus X_1$.
Note that $K = \hat{\D{C}} \setminus (V \setminus X_1) \subseteq \hat{\D{C}} \setminus \gamma$.
Since $\gamma$ is not null-homotpic, both connected components of $\hat{\D{C}} \setminus \gamma$ must  
intersect $K$. However, this is not possible since $K$ is connected. This completes Step 4. 

\medskip

Consider the equivalence relation $\mathord \sim_2$ on $\hat{\D{C}}$ where $z \sim_2 z'$ if and only if 
$z = z'$, or both $z$ and $z'$ belong to $(\hat{\D{C}}\setminus V)\cup X_1$.
By Step 4, $V \setminus X_1$ is a topological disk. 
Thus, $\hat{\D{C}}/\mathord \sim_2$ is the one-point compactification of the topological disk $V \setminus X_1$, 
which is a topological sphere. Hence, we may identify $\hat{\D{C}}/\mathord \sim_2$ with $\hat{\D{C}}$.
Let $\pi_2 : \hat{\D{C}} \to \hat{\D{C}}$ be the quotient map,
\[\pi_2(z)= [z]_{\mathord \sim_2}.\]
Define 
\begin{gather*}
v_2 = \pi_2((\hat{\D{C}}\setminus V)\cup X_1), \\
K_2 = \pi_2((\hat{\D{C}} \setminus V) \cup X) = \pi_2( (\hat{\D{C}}\setminus V)\cup X_1 \cup X_2) 
= \{v_2\} \cup \pi_2(X_2).
\end{gather*}
Then, $v_2$ is a single point, while $K_2$ is a compact set. Note that $X_2$ may or may not be compact. 
Indeed, if $X_2$ is compactly contained in $V$, then $\pi_2(X_2)$ is compact, and $v_2$ is away from $\pi_2(X_2)$. 
However, if $X_2$ is not compactly contained in $V$, then $v_2$ is a limit point of $K_2 \setminus \{v_2\}$. 

Define
\[\Phi_2 : \Psi(X) \to K_2 ,\quad \Phi_2(x,y) = \pi_2(\Psi^{-1}(x,y)).\]
This is a continuous map on $\gY(X)$. Moreover, since $\pi_2$ is continuous and injective on $V \setminus X_1$, 
and $X_2 \subset V \setminus X_1$, $\gF_2$ induces a homeomorphism from $\gY(X_2)$ to $K_2 \setminus \{v_2\}$.

\medskip

\textit{Step 5 : Any continuous map $g: K_2 \to \D{S}^1$ is homotopic to a map $g': K_2 \to \D{S}^1$ 
which is constant on each connected component of $K_2$.}

\smallskip

Consider $H : \Psi(X) \times [0, 1] \to \Psi(X)$, defined as 
\[H((x, y), t) = (x, \min(y, 1-t)).\]
This a continuous map, with $H((x,y), 0)= (x,y)$, and $H((x,y),1)= (x,0)$.
Also, for each $t\in [0,1]$, $H(\cdot, t)$ maps $\gY(X_2)$ into $\gY(X_2)$. 

Let $g : K_2 \to \D{S}^1$ be a continuous map.
Consider $G : K_2 \times [0, 1] \to \D{S}^1$ defined as
\[G(w,t) = \begin{cases}
g(\Phi_2(H(\Phi_2^{-1}(w), t)))		&\tif	w \neq v_2,\\
g(v_2)								                       	&\tif	w = v_2.
\end{cases}\]
We have $G(w,0) \equiv g(w)$. On the other hand, since $\gF_2^{-1}$ maps each connected component of $K_2$ to a 
vertical line segment, $G(w,1)$ is constant on each connected component of $K_2$. 
There remains to show that $G$ is continuous. 

Evidently, $G$ is continuous on $(K_2 \setminus \{v_2\}) \times [0, 1]$, and is constant on $\{v_2\} \times [0,1]$. 
If $v_2$ is away from $K_2$, then we are done. Otherwise, assume that a sequence $w_i$ in $K_2 \setminus \{v_2\}$ 
converges to $v_2$, and $t_i$ converges to $t$. Below we show that $G(w_i, t_i) \to g(v_2)$. 

Note that  $\ol{X_2} \subseteq X_1 \cup X_2$. That is because, if $z_i$ is a sequence in $X_2$ converging to some $z$, 
then $b(z_i)\to b(z) \in \ol{V}$. However, $\partial V \cap B=\emptyset$, which implies that $b(z)\in V$. Thus, 
$b(z)\in X_1 \cup X_2$, and hence $z\in X_1 \cup X_2$. 
Now recall that $\pi_2^{-1}(w_i) \in X_2$, for $i \geq 1$, and $\pi_2: X_2 \to K_2\setminus \{v_2\}$ 
is a homeomorphism. 
Since $w_i \to v_2$, the sequence $\pi_2^{-1}(w_i)$ may not have a limit point in $X_2$. 
Therefore, any limit point of this sequence belongs to $X_1$, that is,   
\[\lim_{ i \to \infty} d(\pi_2^{-1}(w_i), X_1) = 0, \]
which implies 
\[\lim_{ i \to \infty} d(\Psi(\pi_2^{-1}(w_i)), \Psi(X_1)) = 0. \]
For $i \geq 1$, let $(x_i,y_i) = \Phi_2^{-1}(w_i) = \Psi(\pi_2^{-1}(w_i))$. 
As $\Psi(X_1)$ is compact, for every $i \geq 1$, there is $(x_i', y_i')$ in  
$\gY(X_1)$ such that 
\[d((x_i, y_i),\Psi(X_1)) = d((x_i, y_i), (x_i', y_i')).\]
Since $H((x_i', y_i'), t_i) \in \Psi(X_1)$, for all $i \geq 1$, and $H(\cdot, t_i)$ is non-expanding,
\[d(H((x_i,y_i), t_i), \Psi(X_1))  \leq d(H((x_i, y_i), t_i), H((x_i', y_i'), t_i)) 
\leq d((x_i,y_i), (x_i', y_i')).\]
Thus, $H((x_i,y_i), t_i)$ converges to $\gY(X_1)$. 
Recall that $\gF_2(\gY(X_1))=v_2$. Using the continuity of $\gF_2$ and $g$, we conclude that $G(w_i, t_i) \to g(v_2)$.
This completes Step 5. 

\medskip

{\em Step 6. The map $g': K_2 \to \D{S}^1$ is null-homotopic.}

\smallskip

Define 
\[K_3 = \{v_2\} \cup \pi_2(B \cap X_2) = \{v_2\} \cup \pi_2(b(X_2)). \]
We have $K_3\subset K_2$, and $K_3$ is closed and totally disconnected. 
The only point in $K_3$ which might be isolated is $v_2$.
Hence $K_3$ is either a Cantor set, or the union of an isolated point and a Cantor set.
By the discussion preceding the proof of this lemma, $g': K_3 \to \D{S}^1$ is null-homotopic.
Hence, there is a constant $\theta \in \D{S}^1$, and a continuous map $F: K_3 \times [0, 1] \to \D{S}^1$ 
such that 
\[F(w,0) \equiv g'(w), \tand F(w,1) \equiv \theta.\]
Now consider the map $G: K_2 \times [0, 1] \to \D{S}^1$ defined as
\[G(w, t) = 
\begin{cases}
F(\pi_2(b(\pi_2^{-1}(w))), t)		&\tif	w \neq v_2,\\
F(v_2, t)		               						&\tif	w = v_2.
\end{cases}\]
For $w\in K_2 \setminus \{v_2\}$, $\pi_2(b(\pi_2^{-1}(w))) \in K_3$, and therefore, $G$ is well-defined.
Since $g'$ is constant on each component of $K_2$, for $w \in K_2 \setminus\{v_2\}$, 
\[G(w,0)= g'(\pi_2(b(\pi_2^{-1}(w))))= g'(w).\]
For $w= v_2$, $G(w,0)= F(v_2,0)= g'(v_2)$. Thus, $G(w,0)\equiv g'(w)$. 
Also, $G(w,1) \equiv \theta$. 
Below we show that $G$ is continuous. 

Let $(w_i, t_i)$, for $i \geq 1$, be a convergent sequence in $K_2 \times [0, 1]$. 
To prove that $G$ is continuous, it is sufficient to consider sequences with $w_i \in (K_2 \setminus \{v_2\})$ which 
converge to $v_2$.
Let $w'_i = \pi_2(b(\pi_2^{-1}(w_i))) \in K_3$. 
We claim that $w_i' \to v_2$. If not, let $w \in K_3 \setminus \{v_2\}$ be a limit point of the sequence $w_i'$.
Then $\pi_2^{-1}(w) \in \pi_2^{-1}(K_3 \setminus \{v_2\}) \subseteq X_2$ is a limit point of the sequence 
$b(\pi_2^{-1}(w_i))$.
Therefore, $\pi_2^{-1}(w_i)$ has a limit point in $b^{-1}(\pi_2^{-1}(w)) \subseteq X_2$.
Hence, $w_i$ has a limit point in $\pi_2(X_2) = K_2 \setminus \{v_2\}$, which is a contradiction. 
Thus, $w_i' \to v_2$, which implies that $F(w_i', t_i)$ tends to $F(v_2, \lim_{i \to \infty} t_i)$, as $i \to +\infty$. 
Therefore, $G$ is continuous.

\medskip

\textit{Step 7 : The set $(V \setminus X_1) \setminus X_2$ is connected.}

\smallskip

The map $\pi_2$ is injective and continuous on $V \setminus X_1$. Hence, $\pi_2$ is a homeomorphism  from 
$V\setminus X_1$ to $\hat{\D{C}} \setminus \{v_2\}$.
This implies that $(V \setminus X_1)\setminus X_2$ is connected if and only if the set 
$\pi_2((V\setminus X_1)\setminus X_2) = \hat{\D{C}} \setminus K_2$ is connected.
By Borsuk's criterion discussed before the proof, and Steps 5--6, the latter set is connected. 
Thus, $(V \setminus X_1) \setminus X_2$ is connected. 
\end{proof}

Let $X$ be a hairy Cantor set, and $U$ be an admissible set for $X$. 
By a marking for $U$ we mean a collection $u$ of points in $\partial U \setminus X$ such that for every connected 
component $U'$ of $U$ there is a unique $u' \in u$ such that $u' \in \partial U' \setminus X$. 
We refer to the pair $(U,u)$ as a \textbf{marked admissible set} for $X$. 

Let $(U,u)$ and $(V, v)$ be marked admissible sets for $X$ with $\ol{V} \subset U$. 
Let $V'$ be a connected component of $V$ which is contained in a component $U'$ of $U$. 
We say that $V'$ is a \textbf{free component} of $V$ in $(U,u)$, if there is a curve 
\[\gamma: (0,1) \to U' \setminus (\ol{V} \cup X)\]
such that 
\[\lim_{t\to 0} \gamma(t) \in \partial U' \cap u,  \qquad  \tand \qquad \lim_{t \to 1} \gamma(t) \in \partial V' \cap v.\] 
Let us say that $(V, v)$ is a \textbf{free admissible set} for $X$ within $(U,u)$, if $V$ is an admissible set for $X$, 
$\ol{V} \subseteq U$, and every component of $V$ is a free component of $V$ in $(U,u)$. 

\begin{lem}\label{L:admissible_to_free}
Let $X$ be a hairy Cantor set, $(U,u)$ be a marked admissible set for $X$, and $V$ be an admissible set for $X$ with 
$\ol{V} \subset U$. There exists a marked admissible set $(W, w)$ for $X$ such that 
$W \subseteq V$ and $(W,w)$ is free in $(U,u)$. 
\end{lem}

\begin{proof}
Let $V'$ be a component of $V$, which is contained in a component $U'$ of $U$. 
Let us say that $V'$ is an accessible component of $V$ in $(U,u)$, if there is a curve 
\[\gl_{V'}: (0,1) \to U' \setminus (\ol{V} \cup X)\] 
such that 
\[\lim_{t \to 0} \gl_{V'}(t) \in \partial U' \cap u, \quad \tand \quad \lim_{t \to 1} \gl_{V'}(t) \in \partial V' \setminus X.\] 

If every component $V'$ of $V$ is an accessible component of $V$ in $(U,u)$, we define $W=V$, and $w$ 
be the collection of $\lim_{t \to 1} \gl_{V'}(t)$ over all components $V'$  of $V$. 
If $V$ has non-accessible components within $(U,u)$, we modify the components of $V$ so that all components of 
$V$ become accessible. We explain this process below. 

Let $V'$ be a non-accessible component of $V$. Let $U'$ be the component of $U$ which contains $V'$, 
and let $u'$ be the unique element in $u \cap \partial U'$. 
Let $Q_i$, for $1 \leq i \leq m$, denote the components of $V$ which lie in $U'$.
Rearranging the indexes if necessary, we may assume that $Q_1=V'$. 
Recall from \refL{L:properties_of_admissible-2} that $U' \setminus X$ is connected. As $U' \setminus X$ is also open, 
it must be path-connected. 
Also, recall that $u'$ is accessible from $U' \setminus X$. 
These imply that there is a curve 
\[\gl: (0,1) \to U' \setminus (\ol{V'} \cup X)\] 
such that $\lim_{t \to 0} \gamma(t)= u'$ and $\lim_{t \to 1} \gamma(t)\in \partial V' \setminus X$. 

For $2 \leq i \leq m$, let $Q_i'$ denote the union of the connected components of $Q_i \setminus \gl((0,1))$ 
which intersect $X$.
We define  
\[P = (V \setminus \cup_{2 \leq i\leq m} Q_i) \bigcup (\cup_{2 \leq i \leq m} Q_i')= V \setminus \gl((0,1)).\]
Since each $Q_i$ is a topological disk, the connected components of $Q_i'$ are topological disks. 
As $X \cap \gl([0,1]) = \emptyset$, for each $i$, $Q_i \cap X= Q_i' \cap X$, and $V \cap X = P \cap X$. 
It follows that $P$ satisfies properties (i)--(iii) in the definition of admissible sets. 

At this stage, the set $P$ might not satisfy property (iv) in the definition of admissible sets. 
However, we may employ \refL{L:disjoint-closures} to shrink the components of $P$ so that property (iv) holds as well. 
Let $L$ denote this modified admissible set obtain from employing that lemma. 
Note that the number of components of $L$ is the same as the number of components of $P$. 

If $\gl ((0,1))$ meets some $Q_i$, for $2 \leq i \leq m$, then every component of $Q_i'$ is an accessible component of 
$P$ in $(U, u)$ (via a portion of the curve $\gl((0,1))$).
Also, if some $Q_i$, for $2 \leq i \leq m$, is an accessible component of $V$ within $(U,u)$, and 
$\gl((0,1)) \cap Q_i= \emptyset$, then $Q_i$ is an accessible component of $P$ within $(U,u)$.
Moreover, as $V' \cap \gl((0,1))=\emptyset$, $V'$ is an accessible component of $P$ in $(U,u)$, 
via the curve $\gl$.

By the above paragraph, the number of non-accessible components of $P$ within $(U,u)$ is strictly 
less than the number of non-accessible components of $V$ within $(U,u)$. 
On the other hand, if $L'$ is a component of $L$ which is contained in a component $P'$ of $P$, and $P'$ 
is an accessible component of $P$ within $(U,u)$, then $L'$ is an accessible component of $L$ within $(U,u)$. 
Therefore, the number of non-accessible components of $L$ within $(U,u)$ is strictly 
less than the number of non-accessible components of $V$ within $(U,u)$. 

Recall that $V$ has a finite number of connected components. By repeating the above process, we may reduce the 
number of non-accessible components of $V$ until all its components become accessible.
\end{proof}

\begin{lem}\label{L:construction_for_base_curve}
Let $X$ be a hairy Cantor set with base Cantor set $B$. 
There are marked admissible sets $(U_n, u_n)$, for $n\geq 1$, such that the following hold: 
\begin{itemize}
\item[(i)] $\cap_{n \geq 1} U_n= B$;
\item[(ii)] for $n \geq 1$, $U_{n+1}$ is compactly contained in $U_n$;
\item[(iii)] for $n \geq 1$, $(U_{n+1}, u_{n+1})$ is a free admissible set in $(U_n, u_n)$.
\end{itemize}
\end{lem}

\begin{proof}
For $n=1$, let $U_1$ be a Jordan domain containing $X$.  
Let $u_1$ consist of an arbitrary point on $\partial U_1$. 
Evidently, $(U_1, u_1)$ is a marked admissible set for $X$.

Assume that $(U_m, u_m)$ is defined for all $1 \leq m < n$, for some $ n \geq 2$, and satisfy properties (ii) and (iii) 
listed in the lemma. Below, we build $(U_n, u_n)$. 

Let $\gep_n = d(\partial U_{n-1}, B)/2$.
By \refL{L:simply_connected_to_admissible}, there is an admissible set $U_n'$ for $X$ such that 
$U_n' \subset \D{D}(B, \eps_n)$. It follows from the choice of $\gep_n$ that $U_n'$ is compactly contained in 
$U_{n-1}$. 
By the induction hypothesis, $(U_{n-1}, u_{n-1})$ is a marked admissible set for $X$. 
Therefore, by \refL{L:admissible_to_free}, there is a marked admissible set $(U_n, u_n)$ for $X$ such that 
$U_n \subset U'_n$ and $(U_n, u_n)$ is free within $(U_{n-1}, u_{n-1})$. 

Note that for $n\geq 2$, $\gep_n \leq \diam (U_1)/ 2^{n-1}$ and $U_n \subset \D{D}(B, \gep_n)$. 
Therefore, $\cap_{n \geq 1} U_n= B$. 
\end{proof}

\begin{propo}\label{P:every_hcs_has_base}
Every hairy Cantor set admits a base curve.
\end{propo}

\begin{proof}
Let $X$ be a hairy Cantor set, and let $B\subset X$ be the base Cantor set of $X$. 
By \refL{L:construction_for_base_curve}, there is a nest of marked admissible sets $(U_n,u_n)$, for $n\geq 1$, 
such that $(U_{n+1}, u_{n+1})$ is free in $(U_n, u_n)$, and $U_n$ shrinks to $B$. 
Recall that $U_1$ is a Jordan domain containing $X$. 
To simplify the forthcoming presentation, let us choose a Jordan domain $U_0$ which contains $\ol{U_1}$, and 
let $u_0 \subset \partial U_0$ contain a single point. 
For each $n\geq 0$, let $U_{n,i}$, for $1\leq i \leq k_n$, denote the connected components of $U_n$. 
We have $k_0=k_1=1$. 
Also, let $u_{n,i}$ denote the unique point in $u_n \cap \partial U_{n,i}$. 
For $n\geq 0$ and $1 \leq i \leq k_n$, define $J_{n,i}$ as the set of integers $j$ with $1 \leq j \leq k_{n+1}$ and 
$U_{n+1,j} \subset U_{n,i}$. 

For each $n\geq 0$, $(U_{n+1}, u_{n+1})$ is free in $(U_n, u_n)$. Therefore, there are  
\[\gh_{n,i,j}: (0,1) \to U_{n,i} \setminus (\cup_{j\in J_{n,i}} \ol{U_{n+1,j}} \cup X),\]
such that
\[\lim_{t\to 0} \gh_{n,i,j}(t)=u_{n,i}, \qquad \lim_{t \to 1} \gh_{n,i,j}(t)=u_{n+1,j},\]
for all $1 \leq i \leq k_n$ and $j \in J_{n,i}$. 
We may extend $\gh_{n,i,j}$ onto $[0,1]$ by setting $\gh_{n,i,j}(0)=u_{n,i}$ and $\gh_{n,i,j}(1)=u_{n+1,j}$. 
If necessary, we may modify these curves such that $\gh_{n,i,j}((0,1)) \cap \gh_{n,i,j'}((0,1))= \emptyset$, 
for distinct values of $j$ and $j'$. 

For $m \geq 0$, let $T_m$ denote the union of the curves $\gh_{n,i,j}([0,1])$, for all $0 \leq n \leq m$, 
$1 \leq i\leq k_n$, and $j\in J_{n,i}$. 
Each $T_m$ forms a finite tree embedded in the plane. 
Let 
\[T = \cup_{m \geq 0} T_m \cup B.\] 
Since $\cap_{m\geq 1} U_m = B$, $T$ is a compact set in the plane. Moreover, $T \cap X= B$. 

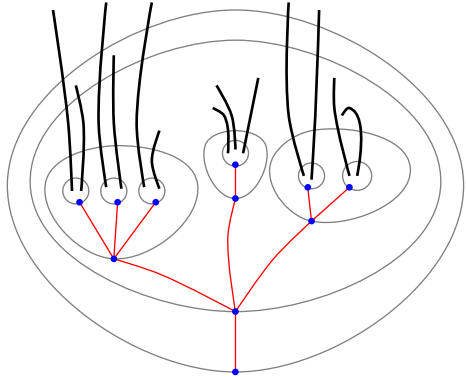
\begin{figure}[ht]
\begin{pspicture}(6,5)%\psgrid
\psccurve[linewidth=.5pt,linecolor=gray](3,0)(6,2.5)(3,4.8)(0,2.5)(3,0) % U_0
\psccurve[linewidth=.5pt,linecolor=gray](3,.8)(5.7,2.5)(3,4.4)(.3,2.5)(3,.8) % U_1
\psccurve[linewidth=.5pt,linecolor=gray](1.4,1.5)(2.5,2.5)(1.4,3)(.5,2.5)(1.4,1.5) % U_2,1
\psccurve[linewidth=.5pt,linecolor=gray](3,2.3)(3.4,3)(3,3.2)(2.6,3)(3,2.3) % U_2,2
\psccurve[linewidth=.5pt,linecolor=gray](4,2)(5.3,2.7)(4,3.2)(3.45,2.7)(4,2) % U_2,3

\pscurve[linewidth=.5pt,linecolor=red](3,0)(3,.7)(3,.8) % eta_{0,1,1}
\pscurve[linewidth=.5pt,linecolor=red](3,.8)(2,1.3)(1.4,1.5) % eta_{1,1,1}
\pscurve[linewidth=.5pt,linecolor=red](3,.8)(2.9,1.8)(3,2.3) % eta_{1,1,2}
\pscurve[linewidth=.5pt,linecolor=red](3,.8)(3.5,1.5)(4,2) % eta_{1,1,3}

\psline[linewidth=.5pt,linecolor=red](1.4,1.5)(.95,2.25) % eta_{2,1,1}
\psline[linewidth=.5pt,linecolor=red](1.4,1.5)(1.45,2.25) % eta_{2,1,2}
\psline[linewidth=.5pt,linecolor=red](1.4,1.5)(1.95,2.25) % eta_{2,1,3}

\psline[linewidth=.5pt,linecolor=red](3,2.3)(3,2.75) % eta_{2,2,1}

\psline[linewidth=.5pt,linecolor=red](4,2)(3.95,2.45) % eta_{2,3,1}
\psline[linewidth=.5pt,linecolor=red](4,2)(4.5,2.45) % eta_{2,3,2}

\pscircle[linewidth=.5pt,linecolor=gray](.9,2.4){.18} % circle in U_2,1
\pscircle[linewidth=.5pt,linecolor=gray](1.4,2.4){.18} % circle in U_2,1
\pscircle[linewidth=.5pt,linecolor=gray](1.9,2.4){.18} % circle in U_2,1

\pscircle[linewidth=.5pt,linecolor=gray](3,2.9){.18} % circle in U_2,2

\pscircle[linewidth=.5pt,linecolor=gray](4,2.6){.18} % circle in U_2,3
\pscircle[linewidth=.5pt,linecolor=gray](4.6,2.6){.2} % circle in U_2,3

\psdots[dotsize=2.5pt,linecolor=blue](3,0)(3,.8)(1.4,1.5)(3,2.3)(4,2)(.95,2.25)(1.45,2.25)(1.95,2.25)(3,2.75)(3.95,2.45)(4.5,2.45)

\pscurve[linewidth=1pt](.85,2.4)(.8,3.3)(.6,4.8) % first curve in $X$
\pscurve[linewidth=1pt](.97,2.4)(1,3.3)(.9,3.8) % first curve in $X$

\pscurve[linewidth=1pt](1.3,2.45)(1.2,3.3)(1.3,4.9) % first curve in $X$
\pscurve[linewidth=1pt](1.5,2.43)(1.4,3.3)(1.4,4.2) % first curve in $X$

\pscurve[linewidth=1pt](1.8,2.45)(1.7,3.3)(1.9,4.9) % first curve in $X$
\pscurve[linewidth=1pt](2,2.43)(1.9,2.8)(2,3.2) % first curve in $X$

\pscurve[linewidth=1pt](2.9,2.9)(2.85,3.4)(2.7,3.5) % first curve in $X$
\pscurve[linewidth=1pt](3,2.95)(2.95,3.4)(2.75,3.8) % first curve in $X$
\pscurve[linewidth=1pt](3.1,2.9)(3.2,3.4)(3.3,3.9) % first curve in $X$

\pscurve[linewidth=1pt](3.9,2.6)(3.7,3.4)(3.7,4.9) % first curve in $X$
\pscurve[linewidth=1pt](4,2.55)(4.05,3.4)(4.1,4.8) % first curve in $X$

\pscurve[linewidth=1pt](4.5,2.6)(4.3,3.5)(4.3,3.9) % first curve in $X$
\pscurve[linewidth=1pt](4.6,2.6)(4.5,3.5)(4.4,3.4) % first curve in $X$
\end{pspicture}
\caption{An admissible nest, the mark points, and the three $T_2$.}
\label{F:Nest}
\end{figure}

We aim to turn $T$ into a base curve $\gamma$ for $X$. 
The curve $\gamma$ will be the limit of a sequence of curves $\gamma_n$, where $\gamma_n$ is obtained from
slightly ``thickening'' the tree $T_n$. We present the details of this construction 
below. See Figure \ref{F:Tree-to-base-curve}. 

%To simplify the forthcoming discussion, we let $k_0=1$ and $J_{0,1}= \{1\}$. Also, we consider a curve 
%\[\gh_{0,1,1}: (0,1) \to \D{R}^2 \setminus T,\] 
%such that 
%\[\lim_{t\to 0} \gh_{0,1,1}(t) \in  \D{R}^2 \setminus T,  \quad \lim_{t\to 0} \gh_{0,1,1}(t)= u_{1,1}.\] 

Note that the curves $\gh_{n,i,j}((0,1))$ are pairwise disjoint, for distinct values of the triple $(n,i,j)$. 
Hence, there are tubular neighbourhoods of these curves, which are pairwise disjoint, and do not meet $X$. 
In other words, there are orientation preserving, continuous and injective maps 
\[\tilde{\gh}_{n,i,j}:(0,1) \times [-0.1,0.1] \to U_{n,i} \setminus (\cup_{j \in J_{n,i}} \ol{U_{n+1,j}} \cup X),\]
such that for all $(t,y) \in (0,1) \times [-0.1,0.1]$, we have 
\begin{equation*}
\begin{gathered}
\tilde{\gh}_{n,i,j}(t,0)=\gh_{n,i,j}(t), \\ 
\lim_{t \to 0} \tilde{\gh}_{n,i,j}(t,y)=\gh_{n,i,j}(0)= u_{n,i}, \\
\lim_{t \to 1} \tilde{\gh}_{n,i,j}(t,y)=\gh_{n,i,j}(1)=u_{n+1,j},
\end{gathered}
\end{equation*}
and for distinct triples $(n,i,j)$ and $(n',i',j')$ we have 
\[\tilde{\gh}_{n,i,j}((0,1) \times [-0.1,0.1]) \bigcap \tilde{\gh}_{n',i',j'}((0,1) \times [-0.1,0.1]) =\emptyset.\]

We may extend $\tilde{\gh}_{n,i,j}$ onto $[0,1] \times [-0.1,0.1]$ by setting $\tilde{\gh}_{n,i,j}(0,y)=u_{n,i}$ and 
$\tilde{\gh}_{n,i,j}(1,y)=u_{n+1,j}$, for all $y\in [-0.1,0.1]$.
Consider  
\[M= \bigcup_{n\geq 0} \bigcup_{1 \leq i \leq k_n} \bigcup_{j\in J_{n,i}} \tilde{\gh}_{n,i,j}([0,1] \times [-0.1,0.1]) 
\bigcup B.\] 
This is a compact set in $\D{R}^2$, with $M \cap X = B$. 

Fix $n\geq 1$ and $1\leq i \leq k_n$. There are $2(|J_{n,i}|+1)$ curves on $\partial M$ which land at $u_{n,i}$. 
These are the curves $\tilde{\gh}_{n,i,j} ([0,1]\times \{-0.1\})$ and $\tilde{\gh}_{n,i,j} ([0,1]\times \{+0.1\})$, 
for $j \in J_{n,i}$, as well as $\tilde{\gh}_{n-1,i',i} ([0,1]\times \{-0.1\})$ and 
$\tilde{\gh}_{n-1,i',i} ([0,1]\times \{+0.1\})$, for some $1\leq i' \leq k_{n-1}$ with $i \in J_{n-1,i'}$. 
Let $j_1$, $j_2$, \dots, $j_m$ denote the elements of $J_{n,i}$, labelled in such a way that the curves 
$\tilde{\gh}_{n,i,j_l}([0,1]\times \{-0.1\})$, for $1\leq l\leq m$, followed by 
$\tilde{\gh}_{n-1,i',i} ([0,1]\times \{-0.1\})$ and then  
$\tilde{\gh}_{n-1,i',i} ([0,1]\times \{+0.1\})$ land at $u_{n,i}$ in a clockwise fashion. 

\begin{figure}[ht]
\psset{xunit=2,yunit=2}
\begin{pspicture}(.9,0)(4.6,3)

￼
% eta_{0,1,1}
￼\pscustom[linewidth=1pt,fillstyle=solid,fillcolor=lightgray,linecolor=lightgray]{% 
\pscurve[liftpen=1](3,0)(3.15,.4)(3,.8) 
\pscurve[liftpen=2](3,0)(2.85,.4)(3,.8) 
}
\pscurve[linewidth=1pt,linecolor=red](3,0)(3,.7)(3,.8) 

\psecurve[linewidth=1pt,linecolor=black](3,0)(3,0)(3.15,.4)(3.035,.74)(3,.8) 
\psecurve[linewidth=1pt,linecolor=black](3,0)(3,0)(2.85,.4)(2.965,.74)(3,.8)

% eta_{1,1,1}
￼\pscustom[linewidth=1pt,fillstyle=solid,fillcolor=lightgray,linecolor=lightgray]{% 
\pscurve[liftpen=1](3,.8)(2.28,1.35)(1.4,1.5)
\pscurve[liftpen=2](3,.8)(2.12,1.05)(1.4,1.5)
}
\psline[linecolor=lightgray](3,.8)(1.4,1.5)
\pscurve[linewidth=1pt,linecolor=red](3,.8)(2,1.3)(1.4,1.5) 

\psecurve[linewidth=1pt,linecolor=black](3,.8)(2.95,.85)(2.28,1.35)(1.46,1.5)(1.4,1.5)
\psecurve[linewidth=1pt,linecolor=black](3,.8)(2.93,.81)(2.12,1.05)(1.45,1.46)(1.4,1.5)

% eta_{1,1,2}
￼\pscustom[linewidth=1pt,fillstyle=solid,fillcolor=lightgray,linecolor=lightgray]{% 
\pscurve[liftpen=1](3,.8)(2.8,1.7)(3,2.3)
\pscurve[liftpen=1](3,.8)(3.1,1.7)(3,2.3)
}
\psline[linecolor=lightgray](3,.8)(3,2.3)
\pscurve[linewidth=1pt,linecolor=red](3,.8)(2.95,1.8)(3,2.3) 

\psecurve[linewidth=1pt,linecolor=black](3,.8)(2.975,.865)(2.8,1.7)(2.975,2.25)(3,2.3)
\psecurve[linewidth=1pt,linecolor=black](3,.8)(3.01,.86)(3.1,1.7)(3.01,2.25)(3,2.3)

% eta_{1,1,3}
￼\pscustom[linewidth=1pt,fillstyle=solid,fillcolor=lightgray,linecolor=lightgray]{% 
\pscurve[liftpen=1](3,.8)(3.4,1.55)(4,2)
\pscurve[liftpen=2](3,.8)(3.6,1.4)(4,2)
}
\psline[linecolor=lightgray](3,.8)(4,2)
\pscurve[linewidth=1pt,linecolor=red](3,.8)(3.5,1.5)(4,2) % eta_{1,1,3}

\psecurve[linewidth=1pt,linecolor=black](3,.8)(3.02,.86)(3.4,1.55)(3.95,1.975)(4,2)
\psecurve[linewidth=1pt,linecolor=black](3,.8)(3.055,.845)(3.6,1.4)(3.97,1.95)(4,2)

% eta_{2,1,1}
￼\pscustom[linewidth=1pt,fillstyle=solid,fillcolor=lightgray,linecolor=lightgray]{% 
\pscurve[liftpen=1](1.4,1.5)(1.23,1.9)(.95,2.25)
\pscurve[liftpen=2](1.4,1.5)(1.15,1.8)(.95,2.25)
}
\psline[linecolor=lightgray](1.4,1.5)(.95,2.25)
\psline[linewidth=1pt,linecolor=red](1.4,1.5)(.95,2.25) 

\psecurve[linewidth=1pt,linecolor=black](1.4,1.5)(1.36,1.54)(1.15,1.8)(.97,2.19)(.95,2.25)
\psecurve[linewidth=1pt,linecolor=black](1.4,1.5)(1.379,1.555)(1.23,1.9)(.995,2.205)(.95,2.25)

% eta_{2,1,2}
￼\pscustom[linewidth=1pt,fillstyle=solid,fillcolor=lightgray,linecolor=lightgray]{% 
\pscurve[liftpen=1](1.4,1.5)(1.37,1.9)(1.45,2.25)
\pscurve[liftpen=2](1.4,1.5)(1.48,1.9)(1.45,2.25)
}
\psline[linecolor=lightgray](1.4,1.5)(1.45,2.25)
\psline[linewidth=1pt,linecolor=red](1.4,1.5)(1.45,2.25) % eta_{2,1,2}

\psecurve[linewidth=1pt,linecolor=black](1.4,1.5)(1.39,1.55)(1.37,1.9)(1.43,2.19)(1.45,2.25)
\psecurve[linewidth=1pt,linecolor=black](1.4,1.5)(1.42,1.56)(1.48,1.9)(1.46,2.19)(1.45,2.25)

% eta_{2,1,3}
￼\pscustom[linewidth=1pt,fillstyle=solid,fillcolor=lightgray,linecolor=lightgray]{% 
\pscurve[liftpen=1](1.4,1.5)(1.76,1.9)(1.95,2.25)
\pscurve[liftpen=2](1.4,1.5)(1.65,1.95)(1.95,2.25)
}
\psline[linecolor=lightgray](1.4,1.5)(1.95,2.25)
\psline[linewidth=1pt,linecolor=red](1.4,1.5)(1.95,2.25) 

\psecurve[linewidth=1pt,linecolor=black](1.4,1.5)(1.42,1.55)(1.65,1.95)(1.903,2.21)(1.95,2.25)
\psecurve[linewidth=1pt,linecolor=black](1.4,1.5)(1.445,1.54)(1.76,1.9)(1.925,2.195)(1.95,2.25)

% eta_{2,2,1}
￼\pscustom[linewidth=1pt,fillstyle=solid,fillcolor=lightgray,linecolor=lightgray]{% 
\pscurve[liftpen=1](3,2.3)(3.06,2.55)(3,2.75)
\pscurve[liftpen=2](3,2.3)(2.94,2.55)(3,2.75)
}
\psline[linecolor=lightgray](3,2.3)(3,2.75)
\psline[linewidth=1pt,linecolor=red](3,2.3)(3,2.75) 

\psecurve[linewidth=1pt,linecolor=black](3,2.3)(2.98,2.35)(2.937,2.55)(2.975,2.69)(3,2.75)
\psecurve[linewidth=1pt,linecolor=black](3,2.3)(3.02,2.35)(3.057,2.55)(3.025,2.69)(3,2.75)

% eta_{2,3,1}
￼\pscustom[linewidth=1pt,fillstyle=solid,fillcolor=lightgray,linecolor=lightgray]{% 
\pscurve[liftpen=1](4,2)(3.92,2.23)(3.95,2.45)
\pscurve[liftpen=2](4,2)(4.03,2.23)(3.95,2.45)
}
\psline[linecolor=lightgray](4,2)(3.95,2.45)
\psline[linewidth=1pt,linecolor=red](4,2)(3.95,2.45) 

\psecurve[linewidth=1pt,linecolor=black](4,2)(3.975,2.05)(3.92,2.23)(3.935,2.39)(3.95,2.45)
\psecurve[linewidth=1pt,linecolor=black](4,2)(4.015,2.06)(4.03,2.23)(3.977,2.39)(3.95,2.45)

% eta_{2,3,2}
￼\pscustom[linewidth=1pt,fillstyle=solid,fillcolor=lightgray,linecolor=lightgray]{% 
\pscurve[liftpen=1](4,2)(4.2,2.25)(4.5,2.45)
\pscurve[liftpen=2](4,2)(4.3,2.19)(4.5,2.45)
}
\psline[linecolor=lightgray](4,2)(4.5,2.45)
\psline[linewidth=1pt,linecolor=red](4,2)(4.5,2.45) % eta_{2,3,2}

\psecurve[linewidth=1pt,linecolor=black](4,2)(4.035,2.05)(4.22,2.27)(4.44,2.422)(4.5,2.45)
\psecurve[linewidth=1pt,linecolor=black](4,2)(4.06,2.028)(4.31,2.2)(4.465,2.39)(4.5,2.45)

% the arcs at end of eta_{0,1,1}
\psarc[linewidth=1pt,linecolor=black](3,.8){.16}{300}{40}
\psarc[linewidth=1pt,linecolor=black](3,.8){.16}{65}{80}
\psarc[linewidth=1pt,linecolor=black](3,.8){.16}{110}{135}
\psarc[linewidth=1pt,linecolor=black](3,.8){.16}{170}{240}

% the arcs at end of eta_{1,1,3}
\psarc[linewidth=1pt,linecolor=black](4,2){.13}{240}{30}
\psarc[linewidth=1pt,linecolor=black](4,2){.13}{55}{75}
\psarc[linewidth=1pt,linecolor=black](4,2){.13}{115}{210}

% the arcs at end of eta_{1,1,1}
\psarc[linewidth=1pt,linecolor=black](1.4,1.5){.14}{360}{45}
\psarc[linewidth=1pt,linecolor=black](1.4,1.5){.14}{65}{75}
\psarc[linewidth=1pt,linecolor=black](1.4,1.5){.14}{100}{110}
\psarc[linewidth=1pt,linecolor=black](1.4,1.5){.14}{135}{325}

% the arcs at end of eta_{1,1,2}
\psarc[linewidth=1pt,linecolor=black](3,2.3){.13}{285}{68}
\psarc[linewidth=1pt,linecolor=black](3,2.3){.13}{110}{245}

% the arcs at end of eta_{2,1,1}
\psarc[linewidth=1pt,linecolor=black](.95,2.25){.14}{287}{315}

% the arcs at end of eta_{2,1,2}
\psarc[linewidth=1pt,linecolor=black](1.45,2.25){.14}{252}{280}

% the arcs at end of eta_{2,1,3}
\psarc[linewidth=1pt,linecolor=black](1.95,2.25){.14}{220}{245}

% the arcs at end of eta_{2,2,1}
\psarc[linewidth=1pt,linecolor=black](3,2.75){.15}{245}{294}

% the arcs at end of eta_{2,3,1}
\psarc[linewidth=1pt,linecolor=black](3.95,2.45){.14}{255}{295}

% the arcs at end of eta_{2,3,2}
\psarc[linewidth=1pt,linecolor=black](4.5,2.45){.15}{210}{240}
\end{pspicture}
\caption{Turning the tree $T_2$ (red curves) to the Jordan curve $\gamma_2$ (black loop) through a thickening process.}
\label{F:Tree-to-base-curve}
\end{figure}
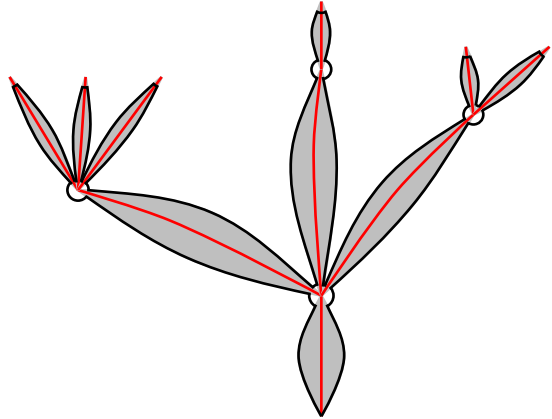

Let $\gep_n \in (0,1/4)$ be small enough, and consider pairwise disjoint curves 
\[\tau_{n,i,j_l}: (0,1) \to \D{R}^2 \setminus M,\] 
for $1\leq l \leq m-1$, such that 
\[\lim_{t\to 0} \tau_{n,i,j_l}(t)= \tilde{\gh}_{n,i,j_l}(\gep_{n}, -0.1), \quad 
 \lim_{t\to 1} \tau_{n,i,j_l}(t)= \tilde{\gh}_{n,i,j_l+1}(\gep_{n}, 0.1).\]
Similarly, let $\rho_{n,i,1}:(0,1) \to \D{R}^2 \setminus M$ and $\rho_{n,i,2}:(0,1) \to \D{R}^2 \setminus M$ 
be such that 
\begin{gather*}
\lim_{t \to 0} \rho_{n,i,1}(t)= \tilde{\gh}_{n-1,i',i}(1-\gep_{n}, 0.1), \quad 
 \lim_{t\to 1} \rho_{n,i,l}(t)= \tilde{\gh}_{n,i,j_1}(\gep_{n}, 0.1), \\
\lim_{t \to 0} \rho_{n,i,2}(t)= \tilde{\gh}_{n,i,j_m}(\gep_{n}, -0.1), \quad 
\lim_{t\to 1} \rho_{n,i,2}(t)= \tilde{\gh}_{n-1,i',i}(1-\gep_{n}, -0.1).  
 \end{gather*} 

Let $\gep_n'$ be a decreasing sequence of positive numbers tending to $0$ such that 
$\D{D}(u_{n,i}, \gep_n') \cap X = \emptyset$, for all $n$ and $i$.  
We may choose the above maps so that the union of the curves $\tilde{\gh}_{n,i,j}(\{\gep_n\} \times [-0.1, +0.1])$, 
for $j \in J_{n,i}$, $\tilde{\gh}_{n-1, i',i}(\{1-\gep_{n-1}\} \times [-0.1,+0.1])$, $\tau_{n,i,j_l}((0,1))$, 
for $1\leq l \leq m-1$, $\rho_{n,i,1}(0,1)$, and $\rho_{n,i,2}(0,1)$ forms a closed loop which encloses the mark point 
$u_{n,i}$ and is contained in $\D{D}(u_{n,i}, \gep_n')$. 

Recall that $\D{S}^1 \subset \D{R}^2$ denotes the circle of radius $1$ about $(0,0)$. 
We may choose closed and connected sets $C_{n, i} \subsetneq \D{S}^1$, for $n \geq 1$ and $1 \leq i \leq k_n$, 
satisfying the following properties: 
\begin{itemize}
\item $C_{n, i}$ is contained in the interior of $C_{m,j}$ if and only if $U_{n, i} \subsetneq U_{m,j}$;
\item if $i \neq j$, then $C_{n,i} \cap C_{n, j}=\emptyset$;  
\item the Euclidean length of $C_{n, i}$ is bounded by  $2^{-n}$.
\end{itemize}
For $n\geq 1$, let $C_n = \cup_{i=1}^{k_n} C_{n,i}$ and define 
\[C= \cap _{n \geq 1} C_n.\]
It follows that $C$ is a Cantor set of points on $\D{S}^1$. 
We aim to build a Jordan curve $\gamma: \D{S}^1 \to \D{R}^2$ such that $\gamma(\D{S}^1) \cap X=B$, and 
$\gamma(t) \in B$ if and only if $t\in C$. 
The map $\gamma$ will be defined as the limit of a sequence of maps $\gamma_n: \D{S}^1 \to \D{R}^2$, for $n\geq 1$.
We define $\gamma_n$ inductively as follows. 

For $n=1$, we define $\gamma_1: \D{S}^1 \to \D{R}^2$ such that 
\begin{gather*}
\gamma_1(C_{1,1})= \tilde{\gh}_{0,1,1}(\{1-\gep_1\} \times [-0.1, +0.1]), \\
\gamma_1(\D{S}^1 \setminus C_{1,1})= \tilde{\gh}_{0,1,1}([0,1-\gep_1) \times \{-0.1\}) 
\cup \tilde{\gh}_{0,1,1}([0,1-\gep_1) \times \{+0.1\}).
\end{gather*}
The choice of the parameterisation of this curve is not important; any choice will work. 

Now assume that the curve $\gamma_n$ is defined for some $n \geq 1$. We define $\gamma_{n+1}$ as follows. 
On $\D{S}^1 \setminus C_n$ we let $\gamma_{n+1} \equiv \gamma_n$. 
On each $C_{n,i}$, we define $\gamma_{n+1}$ in such a way that for all $j \in J_{n,i}$
\[\gamma_{n+1}(C_{n+1,j})= \tilde{\gh}_{n,i,j}(\{1-\gep_n\} \times [-0.1,+0.1]),\]
and $\gamma_{n+1}(C_{n,i} \setminus \cup_{j\in J_{n,i}} C_{n+1,j})$ becomes 
\begin{align*} 
&\bigcup_{j\in J_{n,i}} \tilde{\gh}_{n,i,j}((\gep_n,1-\gep_n) \times \{-0.1\})
\bigcup_{j\in J_{n,i}} \tilde{\gh}_{n,i,j}((\gep_n,1-\gep_n) \times \{+0.1\}]) \\ 
& \quad  \bigcup_{1 \leq l \leq m-1} \tau_{n,i,j_l}([0,1])  \bigcup \rho_{n,i,1}([0,1]) \bigcup \rho_{n,i,2}([0,1]).
\end{align*}
Again, the specific parametrisation of the curve $\gamma_{n+1}$ is not important.
This completes the definition of $\gamma_n$ for all $n\geq 1$. 
Note that for all $n\geq 1$ and $1 \leq i \leq k_n$ we have 
\begin{equation}\label{E:P:image-gamma_n}
\gamma_n (C_{n,i}) \subset \D{D}(U_{n,i}, \gep_n').
\end{equation}

Now we show that $\gamma_n: \D{S}^1 \to \D{R}^2$, for $n\geq 1$, forms a Cauchy sequence. 
Fix $l > 0$, and assume that $m$ and $n$ are bigger than $l$. 
If $x \notin C_l$, we have $\gamma_n(x) = \gamma_m(x)$. 
If $x \in C_{l, i}$, for some $1 \leq i \leq k_l$, by \refE{E:P:image-gamma_n}, $\gamma_l(x) \in \D{D}(U_{l,i}, \gep_l')$. 
It follows from the construction that $\gamma_n(x)$ and $\gamma_m(x)$ also belong to $\D{D}(U_{l,i}, \gep_l')$. 
Therefore, for all $x\in \D{S}^1$, we have 
\[|\gamma_n(x) - \gamma_m(x)| \leq \diam(U_{l, i})+ 2 \gep_l'.\]
However, as $\cap_{n \geq 1} \cup_{i=1}^{k_n} U_{n, i} = B$, $\diam(U_{l, i})$ tends to zero as $l$ tends to infinity.
Hence, the sequence of maps $\gamma_n$ converges uniformly on $\D{S}^1$ to a map 
$\gamma : \D{S}^1 \to \D{R}^2$.
In particular, $\gamma$ is continuous. 

If $x \notin C$, there are $n$ and $i$ such that $x \notin C_{n,i}$.
This implies that for all $m\geq n$, $\gamma_m(x)= \gamma_n(x) \notin B$. 
Therefore, $\gamma(x) \notin X$.
If $x\in C$, there are an arbitrarily large $n$ and $1\leq i \leq k_n$ with $x \in C_{n,i}$. 
By \refE{E:P:image-gamma_n}, $\gamma_n (x) \in \D{D}(U_{n,i},\gep_n')$.
Since $\cap _{n\geq 1} \cup_{1 \leq i \leq k_n} \D{D}(U_{n,i}, \gep_n')= B$, we conclude that 
$\gamma(x) = \lim_{n \to \infty} \gamma_n(x) \in B$. 
On the other hand, for every $z\in B$ there is $x \in C$ with $\gamma(x)=z$. 
To see this, choose $1 \leq i_n \leq k_n$ such that $\cap_{n \geq 1} U_{n,i_n}= z$. There is $x_{n} \in C_{n,i}$ 
with $\gamma_n(x_n) \in \D{D}(U_{n,i_n}, \gep_n')$. If $x_{n_j}$ is a convergent subsequence of $x_n$, converging to
some $x\in \D{S}^1$, we obtain $\gamma(x)= \lim_{j \to \infty} \gamma_{n_j}(x_{n_j})=z$. 
It follows from these statements that $\gamma(\D{S}^1)\cap X= B$. 

To see that $\gamma$ is injective, let $x$ and $y$ be distinct elements of $\D{S}^1$. 
If $x$ and $y$ belong to $C$, then there must be $n\geq 1$ and $i\neq j$ such that $x\in C_{n,i}$ and $y\in C_{n,j}$. 
By \refE{E:P:image-gamma_n} and the decreasing property of $\gep_n'$, for all 
$m\geq n$, $\gamma_m (x) \in \D{D}( U_{n,i}, \gep_m')$ and $\gamma_m(y) \in \D{D}(U_{n,j}, \gep_m')$. 
Since $\ol{U_{n,i}} \cap \ol{U_{n,j}}=\emptyset$, we must have $\gamma(x) \neq \gamma(y)$. 
If $x$ and $y$ do not belong to $C$, then there is $n\geq 1$ such that $x \notin C_n$ and $y \notin C_n$. 
Then, $\gamma(x) = \gamma_n (x) \neq \gamma_n(y) = \gamma(y)$, by the definition of $\gamma$, and the
injectivity of $\gamma_n$. 
If exactly one of $x$ and $y$ belongs to $C$, say $x\in C$ and $y \notin C$, by the above paragraph, 
$\gamma(x)\in B$ and $\gamma (y) \notin B$. Thus, $\gamma(x)\neq \gamma(y)$. 

There remains to show that $X$ does not meet the bounded connected component of 
$\D{R}^2 \setminus \gamma(\D{S}^1)$. 
To see this, note that $\gamma_n(\D{S}^1)$ may be continuously deformed into the tree $T_{n-1}$ in the complement of 
$X$. 
This implies that $X$ does not meet the bounded component of 
$\D{R}^2 \setminus \gamma_n(\D{S}^1)$. Since $\gamma_n$ converges to $\gamma$ uniformly on $\D{S}^1$, 
one infers that $X$ does not meet the bounded component of $\D{R}^2 \setminus \gamma(\D{S}^1)$. 
\end{proof}

\begin{propo}\label{P:plane_minus_based_hcs_has_two_components}
Let $X$ be a hairy Cantor set, and $\go$ be a base curve for $X$. 
Then, the set $\D{R}^2 \setminus (X \cup \go)$ has two connected components, where the bounded component is simply 
connected, and the unbounded component is doubly connected.
\end{propo}

\begin{proof}
Let $U'$ denote the bounded component of $\D{R}^2 \setminus \go$, and $X_\go= X \cup \go$. 
Since $\go$ is a base curve for $X$, $X$ does not meet $U'$. Therefore, $U'$ is a bounded component of 
$\D{R}^2 \setminus X_\go$. Evidently, $U'$ is simply connected. 
Let $U= \D{R}^2 \setminus (X_\go \cup U')$. 
We aim to show that $U$ is connected and doubly connected. 
To this end, we use Borsuk's criterion (see the paragraph after \refL{L:properties_of_admissible-2}), and show that 
$U\cup \{\infty\}$ is connected and simply connected in the Riemann sphere $\hat{\D{C}}$. 

By Borsuk's criterion, in order to prove that $U\cup \{\infty\}$ is connected, it is sufficient to show that any continuous map 
$g : (X_\go \cup U') \to \D{S}^1$ is null-homotopic. 
Below we construct a homotopy between a given such map $g$ and a constant map.

By \refP{P:X-admits-height}, there is a height function $h$ on $X$, which after rescaling we may assume $h: X \to [0, 1]$.
Let $\gy : B \to \psi(B)$ be a homeomorphism from $B$ onto a Cantor set in the line $\{(x,y) \in \D{R}^2 \mid y=0\}$. 
We extend the map $\gy$ to $\gy: X \to \gy(B) \times [0,1] \subset \D{R}^2$, according to 
\[\gy(z) = (\gy(b(z)), h(z)).\]
As $\gy$ is continuous and injective on a compact set, it must be a homeomorphism onto its image. 

Consider the family of map $H_1 :(\gy(B) \times [0, 1]) \times [0, 1] \to \gy(B) \times [0, 1]$, defined as 
\[H_1((x,y), t) = (x, \min(y, 1 - t)).\]
Define the family of maps $G_1: (X_\go \cup U') \times [0, 1] \to \D{S}^1$, defined as 
\[G_1(z, t) = \begin{cases}
g(\gy^{-1}(H_1(\gy(z), t)))		&   \tif	z \in X,\\
g(z)							                    	&   \tif	z \notin X.
\end{cases}\]
One may see that $G_1$ is continuous. Moreover, $G_1(z, 0) = g(z)$, for all $z \in X_\go \cup U'$.
Hence, $G_1$ is a homotopy between $g$ and the map 
\[g': (X_\go \cup U') \to \D{S}^1, \quad g'(z) = G_1(z, 1).\]
Note that $g'(z) = g(z)$, for all $z \in \ol{U'}$, and $g'(z) = g(b(z))$, for all $z \in X$.

Now we construct a homotopy between $g'$ and a constant map.
The set $\ol{U'} = U' \cup \go$ is homeomorphic to the closed unit disk $\ol{\D{D}(0,1)}$. 
Hence, the map $g'|_{\ol{U'}} = g|_{\ol{U'}}$ is null-homotopic.
Therefore, there is a continuous map $H_2:\ol{U'} \times [0, 1] \to \D{S}^1$ and a constant $\theta \in \D{S}^1$ such 
that $H_2(z, 0) \equiv g(z)$ and $H_2(z, 1) \equiv \theta$.
Now define $G_2: (X_\go \cup U') \times [0, 1] \to \D{S}^1$ as $G_2(z, t) = H_2(b(z)), t)$.
Evidently, $G_2$ is continuous, and for all $z \in (X_\go \cup U')$ we have $G_2(z, 0) = g(b(z)) = g'(z)$ and 
$G_2(z, 1) = \theta$.
Hence, $g$ is null-homotopic and therefore, $U$ is connected. 

The complement of $U\cup \{\infty\}$ in $\hat{\D{C}}$, which is $X_\go \cup U'$, is connected. 
This implies that $U\cup \{\infty\}$ must be simply connected. 
\end{proof}

%%%%%%%%%%%%%%%%%%%%%%%%%%%%%%%%%%%%%%%%%%%%%%

\section{Uniformisation of hairy Cantor sets}\label{S:uniformisation}
In this section we prove \refT{T:uniformization}. 
Let us fix a hairy Cantor set $X$, with base Cantor set $B \subset X$, and a base curve 
$\go$. 
By \refP{P:plane_minus_based_hcs_has_two_components}, $X\cup \go$ divides the plane $\D{C}$ into two 
connected components. Let $U'$ denote the bounded connected component of $\D{C} \setminus (X \cup \go)$. 
We have $\partial U'=\go$.
We fix the notations
\begin{equation}\label{E:S:fixed-notations}
X_\go= X \cup \go , \quad X_\go^* = X_\go \cup U',  \quad U= \D{C} \setminus X_\go^*, \quad
U^*= U \cup \{\infty\} \subset \hat{\D{C}}.
\end{equation}
Then, $U^*$ is an open, connected, and simply connected subset of $\hat{\D{C}}$. 

We aim to complete $X$ into a one dimensional (topological) foliation of the plane so that each component of $X$ is 
contained in a leaf of that foliation. 
To that end, we identify a single chart (``a box neighbourhood'') containing $X$ so that within that chart $X$ may be 
completed into a foliation which is homeomorphic to the foliation of the unit square by vertical lines. 
The required portions of leaves will come from hyperbolic geodesics in $U^*$ and $U'$.
This analysis will be carried out in the framework of Carath\'eodory's theory of prime ends, 
which relates the geometry of a plane domain to the topology of its boundary. 
For the general theory of prime ends one may see the classic references \cite{Caratheodory1913} and \cite{Koebe1915}, 
or the modern treatment \cite{Pom75}.
However, below we briefly recap the basic notions and statements of the theory which are employed here. 

We use the notations $\ds$ and $\diams$ for the distance and the diameter, respectively, with respect to the 
spherical metric on $\hat{\D{C}}$.

Let $V \subset \hat{\D{C}}$ be an open, connected, and simply connected set whose boundary contains at least two points. 
A \textbf{fundamental chain} in $V$ is a nest of open sets $(\gO_i)_{i\geq 1}$ in $V$ satisfying the following 
properties: 
\begin{itemize}
\item for $i \geq 1$, $\partial \gO_i \cap V$ is a simple curve in $V$, which is homeomorphic to $(0,1)$ and its 
closure contains two distinct points in $\partial V$; 
\item $\ol{(\partial \gO_i \cap V)} \cap \ol{(\partial \gO_j \cap V)}=\emptyset$, for $1\leq i < j$; 
\item for $i \geq 1$, $\gO_{i+1} \subset \gO_i$;
\item $\diams(\partial \gO_i \cap V) \to 0$, as $i\to +\infty$. 
\end{itemize}
Note that by the Jordan curve theorem, $\partial \gO_i \cap V$ divides $V$ into two connected components. 
Equivalently, some authors define the notion of fundamental chains using sequences of Jordan arcs in $V$ with end 
points in $\partial V$. 

Two fundamental chains $(\gO_i)_{i\geq 1}$ and $(\gO'_i)_{i \geq1}$ in $V$ are called \textbf{equivalent} if
every $\gO_i$ contains some $\gO'_j$ and every $\gO'_j$ contains some $\gO_i$. 
Any two fundamental chains in $V$ are either equivalent or eventually disjoint, that is, $\gO_i \cap \gO'_j=\emptyset$ 
for large enough $i$ and $j$. 
An equivalence class of fundamental chains in $V$ is called a \textbf{prime end} of $V$.

The \textbf{impression} of a prime end of $V$, represented by a fundamental chain $(\gO_i)_{i\geq 1}$, is defined as 
$\cap_{i\geq 1} \ol{\gO_i}$. This is a non-empty compact and connected subset of $\partial V$. 
Evidently, the impression of a prime end is independent of the choice of the fundamental chain $(\gO_i)_{i\geq 1}$. 
Any point on $\partial V$ is contained in the impression of a prime end of $V$. 

We say that $x \in \partial V$ is the \textbf{principal point} of a fundamental chain $(\gO_i)_{i\geq 1}$ in $V$, 
if $\diams (\{x\} \cup (\partial \gO_i \cap V))\to 0$. 
In general, equivalent fundamental chains might have different principal points. 
The set of \textbf{principal points of a prime end} is defined as the set consisting of all principal points of the 
fundamental chains in the class of that prime end. 
Every prime end of $V$ has at least one principal point.
By the theory of prime ends, $p \in \partial V$ is accessible from $V$, if and only if, there is a prime end of $V$ 
whose impression contains $p$, and $p$ is the only point in the set of principal points of that prime end. 
Abusing the terminology, we say a prime end is \textbf{accessible} if it has a unique principal point.

\begin{lem}\label{L:boundary-U}
We have $X_{\go} = \partial U^*$.
\end{lem}

\begin{proof}
Let $z\in \partial U^*$. Since $U^*$ and $U'$ are open in $\hat{\D{C}}$, and $U^* \cap U'=\emptyset$, $z$ may not be in 
$U^* \cup U'$. As $\hat{\D{C}} = U^* \cup U' \cup X_\go$, $z\in X_\go$. 
Therefore, $\partial U^* \subseteq X_\go$. 

Let $z \in X_\go$ and $r>0$. Since $U'$ is a Jordan domain, $\D{D}(z, r) \setminus \ol{U'}$ is a non-empty open set. 
As $X_\go$ has empty interior, $\D{D}(z, r) \setminus \ol{U'}$ may not be contained in $X_\go$, and hence, 
must meet $U$.
This implies that $z\in \partial U$. Thus, $X_\go \subseteq \partial U$. 
\end{proof}

Recall the base map $b: X \to B$ and the peak map $p: X \to X$ defined at the beginning of \refS{S:height-base}.
We may extend these maps to 
\[b: X_\go \to \go, \quad p: X_\go \to X_\go,\]
by setting $b(x)=p(x) =x$, for all $x \in X_\go \setminus X$. Also, for $x\in X_\go \setminus X$, we set 
$[b(x), p(x)] =\{x\}$.
As $b$ is continuous on both $X$ and $\go$, and $X$ and $\go$ are closed sets, $b$ is continuous on $X_\go$. 

We present the properties of the prime ends of $U^*$ in the following lemma. 

\begin{lem}\label{L:prime-ends-impressions}
The following properties hold:  
\begin{itemize}
\item[(i)] For any prime end $P$ of $U^*$ there is $z\in X_\go$ such that the impression of $P$ is equal to $[b(z),p(z)]$. 
On the other hand, for any $z \in X_\go$, there is a prime end of $U^*$ whose impression is equal to $[b(z),p(z)]$.
\item[(ii)] Every prime end of $U^*$ is accessible. Moreover, $w \in X_\go$ is accessible from $U^*$ if and only if 
$w= p(z)$ for some $z \in X_\go$.
\item[(iii)] Distinct prime ends of $U^*$ have disjoint impressions.
\end{itemize}
\end{lem}

\begin{proof}
{\em (i):} 
By \refL{L:boundary-U}, $\partial U^*= X_\go$. Therefore, any $z\in X_\go$ belongs to the impression of a prime 
end of $U^*$. 
That is, there is a fundamental chain $(\gO_i)_{i \geq 1}$ in $U^*$ whose impression contains $z$.
We aim to show that $\cap_{i\geq 1} \ol{\gO_i}= [b(z), p(z)]$. 

For $i \geq 1$, $\overline{\partial \gO_i \cap U^*}$ is a Jordan arc connecting two points on $X_\go$. 
Let $\ga_i$ and $\gb_i$ denote those points. 
%Since for all $i \geq 1$, $\ol{\gO_{i+1}} \subset \ol{\gO_i}$, we obtain 
%$b(\ol{\gO_{i+1}} \cap X_\go) \subseteq b(\ol{\gO_{i}} \cap X_\go)$. 
Since $\ol{\gO_{i}} \cap X_\go$ is connected, for all  $i\geq 1$, $b(\ol{\gO_{i}} \cap X_\go)$ is a closed arc on 
$\go$, which is bounded by $b(\ga_i)$ and $b(\gb_i)$. 
As $z$ belongs to $\partial \gO_i \cap X_\go$, we must have $b(\ga_i) \leq b(z) \leq b(\gb_i)$, with respect to a fixed 
cyclic order on $\go$. 
On the other hand, as $i\to \infty$, $\diams(\partial \gO_i \cap U^*) \to 0$, which implies that $\ds (\ga_i, \gb_i)\to 0$. 
By the uniform continuity of $b: X_\go \to \go$, we conclude that 
\[\lim_{i\to \infty} b(\ga_i)=b(z) = \lim_{i\to \infty} b(\gb_i).\]

By the definition of hairy Cantor sets, the only accessible points of $X$ from $\D{C} \setminus X$ belong to 
$b(X) \cup p(X)$. It follows that $\ga_i$ and $\gb_i$ belong to $b(X_\go) \cup p(X_\go)$. 
Using $\ol{(\partial \gO_i \cap U^*)} \cap \ol{(\partial \gO_j \cap U^*)}=\emptyset$, for $j > i \geq 1$, 
one concludes that for all $i\geq 1$, 
\begin{equation}\label{E:L:prime-ends-impressions-2}
b(\ga_i) < b(z) < b(\gb_i).
\end{equation}
In particular, $[b(z), p(z)] \subset \ol{\gO_i}$, for all $ i\geq 1$, and hence, 
$[b(z), p(z)] \subseteq \cap_{i\geq 1}\ol{\gO_i}$.
On the other hand, for all $w \in \cap_{i\geq 1} \ol{\gO_i}$, $b(\ga_i) \leq b(w) \leq b(\gb_i)$, for all $i\geq 1$. 
By the above paragraph, $b(\ga_i) \to b(z)$ and $b(\gb_i) \to b(z)$, which leads to $b(w)=b(z)$. 
Therefore, $\cap_{i\geq 1} \ol{\gO_i} \subseteq [b(z), p(z)]$. Combining the two relations, we obtain 
$\cap_{i\geq 1} \ol{\gO_i}= [b(z), p(z)]$.  

\medskip

{\em (ii):} 
Let $P$ be a prime end in $U^*$, and let $(\gO_i)_{i \geq 1}$ be a fundamental chain in $U^*$ in the class of $P$ 
which has  a principal point, say $y \in X_\go$. We aim to show that $p(y)=y$. 

Assume in the contrary that $p(y) \neq  y$. In particular, $b(y)\neq p(y)$, and $[b(y), p(y)]$ is a Jordan arc. 
We may extend $[b(y), p(y)]$ into $U'$ so that we obtain a Jordan arc $\gamma$ in $[b(y), p(y)]\cup U'$ 
with of the end points of $\gamma$ in $U'$. 
There is a Jordan domain $D$ such that $[b(y), y] \subset D$, $p(y) \notin \ol{D}$, $\partial D \cap \go$ consists of two 
points, and $D \setminus \gamma$ has two connected components. 

Let $\ga_i$ and $\gb_i$ denote the landing points of $\partial \gO_i \cap U^*$ on $X_\go$. Since $y$ is the principal 
point of $(\gO_i)_{i \geq 1}$, we must have $\ga_i \to y$ and $\gb_i \to y$, as $i \to \infty$. 
Moreover, since $y$ belongs to the impression of $(\gO_i)_{i\geq1}$, as in \refE{E:L:prime-ends-impressions-2},
we have $b(\ga_i) < b(y) < b(\gb_i)$, for all $i\geq 1$. 
By the choice of $D$, this implies that for large enough $i$, $\ga_i$ and $\gb_i$ belong to distinct components of 
$D \setminus \gamma$.
On the other hand, $\ga_i$ and $\gb_i$ belong to $\ol{\partial \gO_i \cap U^*}$, 
with $\ol{\partial \gO_i \cap U^*}  \cap \gamma=\emptyset$ and $\diams (\ol{\partial \gO_i \cap U^*}) \to 0$. 
This implies that for large enough $i$, both $\ga_i$ and $\gb_i$ belong to the same component of $D \setminus \gamma$. 
This contradiction shows that we must have $p(y)=y$. 

Let $P$ be an arbitrary prime end of $U^*$. 
By part (i), the impression of $P$ is of the form $[b(z), p(z)]$, for some $z\in X_\go$. 
By the above argument, if $y$ is a principal point of $P$, we must have $y=p(y)$. 
Since $y\in [b(z), p(z)]$, and there is a single point $y$ in $[b(z), p(z)]$ with $p(y)=y$, we conclude that $P$ 
has a unique principle point. Thus, $P$ is accessible. 
On the other hand, by Part (i), for any $z\in X_\go$, there is a prime end of $U^*$ whose impression is equal 
to $[b(z), p(z)]$. 
By the above argument, $p(z)$ is the unique principal point of that prime end. 
Therefore, the set of principal points of the prime ends of $U^*$ is equal to $\{p(z) \mid z\in X_\go\}$. 

\medskip 

{\em (iii):}
Assume in the contrary that there are distinct prime ends $P^1$ and $P^2$ in $U$ whose impressions intersect.
By part (i), the impressions of $P_1$ and $P_2$ must be equal to $[b(y), p(y)]$, for some $y \in X_\go$.
Let $(\Omega_i^1)_{i \geq 1}$ and $(\Omega_i^2)_{i \geq 1}$ be fundamental chains in the classes of $P^1$ and $P^2$, 
respectively.
Since $P_1$ and $P_2$ are not equivalent, there is $j \geq 1$ such that $\Omega_j^1 \cap \Omega_j^2 = \emptyset$.
For $i \geq 1$, choose $y_i \in \Omega_i^1$.
We have $d(y_i, [b(y), p(y)]) \to 0$.

Let $\ga_i$ and $\gb_i$ denote the landing points of $\partial \gO_i \cap U^*$ on $X_\go$. We may relabel these points so 
that $\ga_i < b(y) < \gb_i$, with respect to a fixed cyclic order on $\go$. 
The open set
\[W_j =  \interior(\ol{\Omega^2_{j}} \cup \ol{U'}) \]
is a Jordan domain, which is bounded by
\[\ol{\partial \Omega^2_{j} \cap U} \cup [\ga_{j}, b(\ga_{j})] \cup [\gb_{j}, b(\gb_{j})] \cup 
\{ w \in \omega \mid b(\gb_{j}) \leq b(w) \leq b(\ga_{j}) \}). \]
In particular, $[b(y), p(y)] \subseteq W_j$.
Since $d(y_i, [b(y), p(y)]) \to 0$, as $i \to \infty$, there must be $i > j$ such that $y_i \in W_j$, and hence 
$y_i \in W_j \cap U=\Omega^2_j$.
Therefore, 
\[y_i \in \Omega_i^1 \cap \Omega_j^2 \subseteq \Omega_j^1 \cap \Omega_j^2 = \emptyset,\]
which is a contradiction.
\end{proof} 

Let $V$ be a connected and simply connected domain in $\hat{\D{C}}$ whose boundary contains at least two points.
Let $\gy : \D{D}(0,1) \to V$ be a Riemann map, that is, a one-to-one and onto holomorphic map. 
By the general theory of prime ends, for any fundamental chain $(\Omega_i)_{i \geq 1}$ in $V$, 
$(\gy^{-1}(\Omega_i))_{i \geq 1}$ is a fundamental chain in $\D{D}(0,1)$. 
This correspondence induces a bijection between the set of prime ends of $V$ and the set of prime ends of $\D{D}(0,1)$.

The prime ends of $\D{D}(0,1)$ are easy to understand. Any prime end of $\D{D}(0,1)$ is accessible, 
and distinct prime ends of $\D{D}(0,1)$ have distinct principal points. 
The map which sends a prime end of $\D{D}(0,1)$ to its unique principal point 
induces a bijection between the set of prime ends of $\D{D}(0,1)$ and $\partial \D{D}(0,1)$. 
Combining with the above paragraph, $\gy$ induces a one-to-one correspondence between $\partial \D{D}(0,1)$ and 
the set of the prime ends of $V$. 
Let $P$ be a prime end of $V$ which is the image of $e^{i \theta} \in \partial \D{D}(0,1)$ under this bijection and 
$P$ has a unique principal point in $\partial V$. 
Then, the radial limit $\lim_{r \to 1^-} \gy(r e^{i\theta})$ exists and is equal to the principal point of $P$. 
See Corollary 2.17 in \cite{Po92}.

By \refP{P:plane_minus_based_hcs_has_two_components}, the set $U^*$ is connected, and simply connected. 
We may consider a Riemann map
\[\gF: \hat{\D{C}} \setminus \ol{\D{D}(0,1)} \to U^*,\]
with $\gF(\infty)= \infty$. 
By the above paragraphs, $\gF$ induces a bijection between the set 
$\partial \D{D}(0,1) = \partial (\hat{\D{C}}\setminus \ol{\D{D}(0,1))}$ and the set of the prime ends of $U^*$.
By \refL{L:prime-ends-impressions}, every prime end in $U^*$ is accessible, and has a unique principal point 
which belongs to $\{p(z) \mid z\in X_\go\}$. Thus, there is a bijection between $\partial \D{D}(0,1)$ and $p(X_\go)$. 
For every $\theta \in \D{R}$, the radial limit 
\[\lim_{r \to 1^+} \gF(r e^{i\theta})\] 
exists and belongs to $p(X_\go)$. 
Moreover, for $\theta \neq \theta'$ in $\D{R}/2\pi \D{Z}$, 
$\lim_{r \to 1^+} \gF(r e^{i\theta}) \neq \lim_{r \to 1^+} \gF(r e^{i \theta'})$.

We extend the base map $b: X_\go \to \go$ to 
\begin{equation}\label{E:base-map-extended}
b: U \cup X_\go \to \go
\end{equation}
as follows. 
For an arbitrary $x \in U$ we let $\theta(x)= \arg (\gF^{-1}(x)) \in \D{R}/2\pi \D{Z}$, and define 
\[b(x) =b\Big ( \lim_{r \to 1^+} \gF\big(r e^{i\theta(x)}\big)\Big).\] 

For each $w\in \go$, there is a unique $\theta \in \D{R}/2\pi \D{Z}$ such that $\lim_{r \to 1^+} \gF(r e^{i\theta})=p(w)$. 
Then, for all $w'$ on the curve $\gF(\{re^{i\theta} \mid r > 1\})$, $b(w')= b(p(w))=w$. 
This implies that for each $w\in \go$, the set $b^{-1}(w)$ in $U \cup X_\go$ is a Jordan arc consisting of $[w, p(w)]$ and 
$\gF(\{re^{i\theta} \mid r > 1\})$. This curve meets $\go$ only at $w$.

\begin{propo}\label{P:b-is-continuous}
The map $b: U \cup X_\go \to \go$ is continuous.
\end{propo}

\begin{proof}
We aim to prove that $b^{-1}(I)$ is closed in $U \cup X_\go$, for every closed set $I \subseteq \go$.
Since $\go$ is a Jordan curve, it is sufficient to prove that $b^{-1}(I)$ is closed whenever $I \subseteq \go$ is 
either a point or a Jordan arc.
If $I$ is a single point, then $b^{-1}(I)$ is a Jordan arc and therefore is closed.
Let $I \subseteq \go$ be a Jordan arc and let $\ga$ and $\gb$ be its end-points.
We know that $b^{-1}(\ga)$ and $b^{-1}(\gb)$ are disjoint Jordan arcs which land at $\ga$ and $\gb$ on $\go$. 
Therefore, $b^{-1}(\ga) \cup b^{-1}(\gb))$ divides $U\cup X_\go$ into two disjoint Jordan domains, denoted by 
$V_1$ and $V_2$.
The sets $\ol{V_1} \cap \go$ and $\ol{V_2} \cap \go$ are Jordan arcs in $\go$ with end points at $\ga$ and $\gb$.
Hence, one of these sets, say $\ol{V_1} \cap \go$, is equal to $I$.
On the other hand, $U \cup X_\go$ is a disjoint union of $b^{-1}(w)$, for $w \in \go$.
Therefore, $\ol{V_1}$ is a disjoint union of such arcs. This implies that 
\[\ol{V_1} = b^{-1}(b(\ol{V_1}))= b^{-1}(\ol{V_1} \cap \go) = b^{-1}(I). \qedhere\]
\end{proof}

For $w \in \go$, $b^{-1}(w)$ is a Jordan arc in $U \cup X_\go$ which lands at a single point on $\go$. 
For $x$ and $y$ in $b^{-1}(w)$, we define $[x, y]$ to be the unique Jordan arc in $b^{-1}(w)$ which ends at $x$ and $y$. 
When $x=y$, we have $[x,y]=\{x\}$. 
Note that if $x$ and $y$ belong to $X_\go \cap b^{-1}(w)$, this notation is consistent with our earlier notation $[x, y]$.

\begin{propo}\label{P:bracket-continuous}
For any convergent sequence $(x_i)_{i \geq 1}$ in $U \cup X_\go$ with $x_i \to x$, we have 
\[\lim_{i \to \infty} [x_i, b(x_i)] = [x, b(x)],\]
in the Hausdorff topology.
\end{propo}

When $x_i \in X$, for all $i\geq 1$, we already have the convergence of $[x_i, b(x_i)]$ to $[x, b(x)]$ from the definition of 
hairy Cantor sets. 
When $x_i \notin X$, $[x_i, b(x_i)]$ is the union of $[b(x_i), p(x_i)]$ and $[p(x_i), x_i] \setminus \{p(x_i)\}$, 
where the latter set is a hyperbolic geodesic in $U^*$. 
For the above proposition, we need to show that such a sequence of hyperbolic geodesics may not have pathological
behaviour. 
We control the location of those hyperbolic geodesic by employing a theorem of Gehring and Hayman, see \cite{GeHa62} 
or \cite[Page 88]{Po92}. 
That is, let $V \subseteq \hat{\D{C}}$ be a simply connected domain whose boundary contains at least two points, and let 
$x$ and $y$ be distinct points in $V$. Assume that $\lambda$ is an arbitrary curve in $V$ connecting $x$ to $y$, 
and $\gamma$ is a geodesic in $V$, with respect to the hyperbolic metric on $V$, connecting $x$ to $y$. 
By Gehring-Hayman theorem, 
\[\diams(\gamma) \leq C_{GH} \diams(\lambda),\]
where $C_{GH}$ is a universal constant independent of $V$, $x$, $y$, and $\lambda$. 

For $r \geq 1$, let 
\begin{equation}\label{E:A_R}
A_r = \gF\big( \{w\in \D{C} \mid 1< |w| \leq r \}\big) \cup X_\go.
\end{equation}
This is a compact annulus in $\D{C}$. 

\begin{proof}[Proof of \refP{P:bracket-continuous}]
For any $w\in \go$, $b^{-1}(w) \subset U\cup X_\go$ is a Jordan arc which is homeomorphic to $[0, 1)$. 
So there is a linear order on each $b^{-1}(w)$, where $w=b(w)$ is the smallest point.

Let us fix $R\geq 1$ such that for all $i\geq 1$, $x_i \in A_R$, where $A_R$ is the compact annulus defined 
by \refE{E:A_R}.
Then, for all $i\geq 1$, $[x_i, b(x_i)] \subset A_R$. 
The set of non-empty compact subsets of $A_R$ with respect to the Hausdorff topology is compact. 
Therefore, to prove the proposition, it is sufficient to prove that if $[x_i, b(x_i)] )_{i \geq 1}$ is a 
convergent sequence 
in the Hausdorff topology, then $[x_i, b(x_i)] \to [x, b(x)]$.

Assume that $([x_i, b(x_i)])_{i\geq 1}$ converges in the Hausdorff topology, and let $K$ denote the limit 
of this sequence. 
Then, 
\[K = \{ y \in A_R \mid y = \lim_{i \to \infty} y_i, y_i \in [x_i, b(x_i)], \forall i\geq 1\}.\]
By \refP{P:b-is-continuous}, $b:U \cup X_\go \to \go$ is continuous. 
This implies that $K$ is contained in $b^{-1}(b(x))$.
On the other hand, since each $[x_i, b(x_i)]$ is connected, $K$ must be also connected.
Therefore, $K = [y, b(x)]$, for some $y \in b^{-1}(b(x))$. 
Moreover, since $x_i \to x$, we have $x\in K$, and hence $x \leq y$, with respect to the linear order 
on $b^{-1}(b(x))$. 
We need to show that $x=y$. 
Assume in the contrary that $y > x$. We shall derive a contraction by considering two cases. 

\medskip

{\em Case (i):} Assume that $y > x$ and $y \in U$. Define 
\[r = \begin{cases}
(|\gF^{-1}(y)| + |\gF^{-1}(x)|)/2 &\tif	x \in U,\\
(|\gF^{-1}(y)| + 1)/2 &\tif	x \in X_\go.
\end{cases}\]
Consider the compact annulus $A_r$ defined by \refE{E:A_R}.
Since $x \in \op{int} (A_r) \cup \go$, and $x_i \to x$, there is an integer $i_0$ such that for all $i \geq i_0$, 
$x_i \in A_r$. 
In particular, for $i\geq i_0$, $[x_i, b(x_i)]$ is contained in $A_r$. 
As $A_r$ is compact, the Hausdorff limit of $[x_i, b(x_i)]$, that is $K$, must be contained in $A_r$. 
But, $y\in K$ and $y \notin A_r$, since $|\gF^{-1}(y)| > r$. This is a contradiction. 

\medskip

{\em Case (ii):} Assume that $y > x$ and $y \in X_\go$. 
By passing to a subsequence, we may assume that $(b(x_i))_{i \geq 1}$ is monotone on $\go$, that is, 
for all $i\geq 1$, $b(x_i) < b(x_{i+1}) < b(x)$, with respect to a fixed cyclic order on $\go$. 
Define 
\[\gamma = b^{-1}(b(x)) \cap A_R , \quad \gamma_i = b^{-1}(b(x_i)) \cap A_R,\; \forall i\geq 1.\]
By virtue of the continuity of $b: A_R \to \go$ in \refP{P:b-is-continuous}, $\gamma_i \to \gamma$, as $i\to \infty$, 
with respect to the Hausdorff topology. 

The Jordan arcs $\gamma$ and $\gamma_1$ divide $A_R$ into two Jordan domains. Let  
\[E=\{ w \in \op{int} (A_R) \mid b(x_1) < b(w) < b(x) \}.\footnote{$\op{int}(A)$ denotes the topological interior of 
a given set $A \subseteq \D{C}$.}\]

Since $x< y$ in this case, we may choose two points $z$ and $w$ in $\gamma$ such that $x < z < w < y$.
There are Jordan arcs $\lambda^z : [0, 1] \to \ol{E}$ and $\lambda^w: [0,1] \to \ol{E}$ such that
\begin{equation*}
\begin{gathered}
\lambda^z(0) \in \gamma_1, \quad \lambda^z(1) = z \in \gamma, \quad \lambda^z((0, 1)) \subseteq E, \\
\lambda^w(0) \in \gamma_1, \quad \lambda^w(1) = w \in \gamma, \quad \lambda^w((0, 1)) \subseteq E,
\end{gathered}
\end{equation*}
and 
\[\lambda^z([0, 1]) \cap \lambda^w([0, 1])= \emptyset.\]
Each of the curves $\lambda^z((0,1))$ and $\lambda^w((0,1))$ divides $E$ into two connected components. 

\begin{figure}[ht]
\begin{pspicture}(-1.2,0)(10,5)
\pscurve[linewidth=.5pt](.5,.8)(1,1)(2,1.2)(4,1.5)(6,1.5)(6.5,1.5)(8.7,1.2)(9.2,1) % omega
\pscurve[linewidth=.5pt,linecolor=blue](.5,3.8)(.75,3.9)(3,4.4)(5.3,4.5)(8,4.2)(9.2,3.9)(9.5,3.8) % outer boundary of A_R
\pscurve[linewidth=.5pt,linecolor=blue](1,1)(1.2,2)(.9,3)(.75,3.9) % gamma_1
\pscurve[linewidth=.5pt,linecolor=blue](8.7,1.2)(9,2)(9,3)(9.2,3.9) % gamma 
\pscurve[linewidth=.5pt,linecolor=red](1.2,2)(3,2.3)(4,2.4)(6,2.4)(9,2) % lambda^z
\pscurve[linewidth=.5pt,linecolor=red](.9,3)(2,3.15)(5,3.5)(6.6,3.35)(9,3) % lambda^w

\pscurve[linewidth=.5pt,linecolor=blue](2,1.2)(2,3.15)(2.2,4)(2.4,3)(2.5,3.5)(2.6,3)(2.8,1.7)(3,2.3)(3,4.4) % gamma_2

\pscurve[linewidth=.5pt,linecolor=blue](6,1.5)(6,2.4)(5.9,4)(5.7,3)(5.5,1.9)(4.8,3.5)(4.8,3.8)(5.1,3.3)(5.3,4.5) % gamma_3

\pscurve[linewidth=.5pt,linecolor=blue](6.5,1.5)(6.6,3.35)(6.8,4)(6.8,3.3)(6.8,2.3)(6.9,2.1)(6.95,2.3)(7.1,3.3)(7.2,3.6)(7.3,3.3)(7.4,2.2)(7.5,1.8)(7.65,2.2)(8,4.2) % gamma_4

\pscurve[linewidth=.5pt](6.2,1.5)(6.3,3.35)(6.8,4.3)(7.3,3.8)(7.4,3.3)(7.5,2.2)(7.5,2) % component of X near gamma_4

\psecurve[linewidth=.5pt](2,.8)(2,1.2)(2,3.15)(2.2,4)(2.4,3)(2.5,3.5)(2.6,3)(2.8,1.7)(3,2.3) % coloring part of gamma_2
\psecurve[linewidth=.5pt](6,1.2)(6,1.5)(6,2.4)(5.9,4)(5.7,3)(5.5,1.9) % coloring part of gamma_3
\psecurve[linewidth=.5pt](6.5,1.2)(6.5,1.5)(6.5,2)(6.6,3.35)% coloring part of gamma_4

\psecurve[linewidth=1pt,linecolor=green](1.2,2)(2.58,2.23)(3,2.3)(4,2.4)% lambda^z_2
\psecurve[linewidth=1pt,linecolor=green](4,2.4)(5,2.41)(5.7,2.4)(6,2.4)% lambda^z_3
\psecurve[linewidth=1pt,linecolor=green](6,2.4)(7.42,2.23)(7.65,2.2)(9,2) % lambda^z_4

\psecurve[linewidth=1pt,linecolor=green](2,3.15)(2.6,3.23)(3.05,3.3)(5,3.5) % lambda^w_2
\psecurve[linewidth=1pt,linecolor=green](5,3.5)(5.3,3.5)(5.7,3.45)(6.6,3.35) % lambda^w_3
\psecurve[linewidth=1pt,linecolor=green](6.6,3.35)(7.33,3.25)(7.87,3.17)(9,3) % lambda^w_4

\psecurve[linewidth=.5pt](.9,.8)(1,1)(1.13,1.5)(1.2,2) % coloring part of gamma_1
\psecurve[linewidth=.5pt](8.65,1)(8.7,1.2)(9,2)(9,3)(9.15,3.7)(9.2,3.9) % coloring part of gamma 

\pscurve[linewidth=.5pt,linestyle=dotted](-.5,.4)(.5,.8)(1,1)(2,1.2)(4,1.5)(6,1.5)(6.5,1.5)(8.7,1.2)(9.2,1)(9.7,.75) % dotted extension of omega

\pscurve[linewidth=.5pt,linecolor=blue,linestyle=dotted](-.5,3.4)(.5,3.8)(.75,3.9)(3,4.4)(5.3,4.5)(8,4.2)(9.2,3.9)(9.5,3.8)(9.9,3.65) % dotted extension of outer boundary of A_R

\psdot[dotsize=2pt](8.87,1.6) \rput(9.2,1.5){\small $x$}
\psdot[dotsize=2pt](9,2) \rput(9.2,2){\small $z$}
\psdot[dotsize=2pt](9,3) \rput(9.2,3){\small $w$}
\psdot[dotsize=2pt](9.1,3.5) \rput(9.3,3.5){\small $y$}

\psdot[dotsize=2pt](1,1) \rput(1,.5){\small $b(x_1)$}
\psdot[dotsize=2pt](2,1.2) \rput(2,.7){\small $b(x_i)$}
\psdot[dotsize=2pt](6,1.5) \rput(5.8,1){\small $b(x_j)$}
\psdot[dotsize=2pt](6.5,1.5) \rput(6.6,1){\small $b(x_l)$}
\psdot[dotsize=2pt](8.7,1.2) \rput(8.7,.7){\small $b(x)$}

\psdot[dotsize=2pt](.75,3.9) \rput(.75,4.3){\small $t_1$}
\psdot[dotsize=2pt](3,4.4) \rput(3,4.8){\small $t_i$}
\psdot[dotsize=2pt](5.3,4.5) \rput(5.3,4.9){\small $t_j$}
\psdot[dotsize=2pt](8,4.2) \rput(8,4.6){\small $t_l$}

\psdot[dotsize=2pt](2.2,4) \rput(1.9,4){\tiny $y_i$}
\psdot[dotsize=2pt](5.9,4) \rput(5.9,4.2){\tiny $y_j$}
\psdot[dotsize=2pt](6.8,4) \rput(6.8,4.2){\tiny $y_l$}

\psdot[dotsize=2pt](2.8,1.7) \rput(2.8,1.5){\tiny $x_i$}
\psdot[dotsize=2pt](5.5,1.9) \rput(5.5,1.7){\tiny $x_j$}
\psdot[dotsize=2pt](7.5,1.8) \rput(7.5,1.6){\tiny $x_l$}

\rput(1.1,3.3){\tiny $\lambda^w$}
\rput(1.4,2.3){\tiny $\lambda^z$}

\rput(2.8,3.5){\tiny $\lambda_i^w$}
\rput(2.8,2.5){\tiny $\lambda_i^z$}

\rput(5.53,3.7){\tiny $\lambda_j^w$}
\rput(5.3,2.7){\tiny $\lambda_j^z$}

\rput(7.59,3.5){\tiny $\lambda_l^w$}
\rput(7.55,2.5){\tiny $\lambda_l^z$}

\rput(-.3,3.9){\small $\gF(\partial \D{D}(0,R))$}
\rput(.2,.9){\small $\go$}

\end{pspicture}

\caption{Illustration of the curves and points in the proof of \refP{P:bracket-continuous}. 
The black curves lie in $X_\go$ while the blue curves lie in $U$. The green Jordan arcs denote the curve 
$\lambda_i^z$, $\lambda_i^w$, \dots.}
\label{F:foliation}
\end{figure}
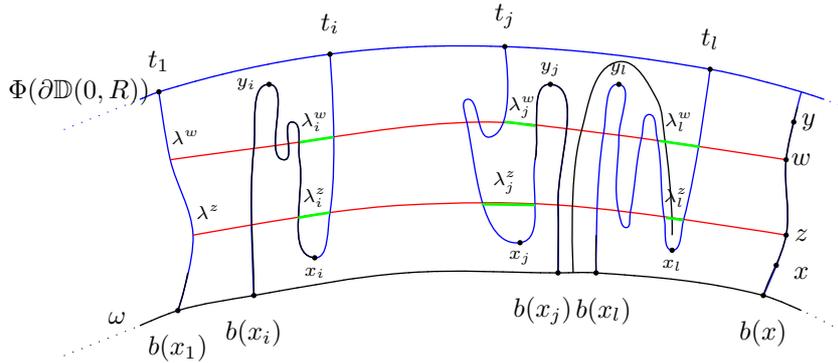

Let us fix $y_i \in [x_i, b(x_i)]$, for $i\geq 1$, such that $y_i \to y$ as $i \to \infty$. 
Let $t_i$ denote the unique point in $\gamma_i \cap \gF(\partial \D{D}(0,R))$.
Then, $\gamma_i= [t_i, b(t_i)]=[t_i, b(x_i)]$, for all $i\geq 1$. 

The curve $\gl^w((0,1))$ divides $E$ into two connected components. 
As $y_i \to y$, $x_i \to x$, and $x < w < y$, there must be $i_1\geq 1$ such that for all $i\geq i_1$, 
$x_i$ and $y_i$ lie in distinct components of $E \setminus \lambda^w((0,1))$.
In particular, for every $i\geq i_1$, both of the curves $[b(x_i), x_i]$ and $[x_i, t_i]$ meet the curve $\lambda^w((0,1))$. 
It follows that for $i \geq i_1$, the sets $C_{i,1}=[b(x_i), x_i] \cap \lambda^w((0,1))$ and 
$C_{i,2}= [x_i, t_i] \cap \lambda^w((0,1))$ are non-empty compact sets in $\lambda^w((0,1))$.
Since $C_{i,1} \cap C_{i,2} = \emptyset$, there must be a Jordan arc 
\[\lambda_i^w: [0,1] \to \lambda^w((0,1))\] 
such that 
\[\lambda_i^w(0) \in C_{i,1}, \quad \lambda_i^w(1) \in C_{i,2}, 
\quad \lambda_i^w((0,1)) \cap (C_{i,1} \cup C_{i,2})= \emptyset.\]
Note that the last property in the above equation implies that $\lambda_i^w((0,1)) \cap \gamma_i=\emptyset$. 
Then, the union of the Jordan arcs $\lambda_i^w([0,1])$ and $[\lambda_i^w(0), \lambda_i^w(1)]$ forms a 
Jordan curve in the plane. Let $W_i$ denote the bounded component of the complement of that Jordan curve. 
Thus, we have 
\begin{equation}\label{E:P:bracket-continuous-7}
\partial W_i=\lambda_i^w((0,1)) \cup [\lambda_i^w(0), \lambda_i^w(1)].
\end{equation}
and 
\begin{equation}\label{E:P:bracket-continuous-3}
W_i \cap \gamma_i = \emptyset.
\end{equation}

The closure of the Jordan domain $W_i$ is homeomorphic to $\ol{\D{D}(0,1)}$. 
Moreover, as $x_i$ belongs to $[\lambda_i^w(0), \lambda_i^w(1)]$, $\ol{W_i}$ contains $x_i$ and $\lambda_i^w([0,1])$. 
Since $x_i$ and $\lambda_i^w([0,1])$ lie on distinct components of $E \setminus \lambda^z((0,1))$, 
the curve $\lambda^z((0,1))$ must divide $\ol{W_i}$ into at least two connected components. 
It follows that there is a Jordan arc 
\[\lambda_i^z: [0,1] \to \ol{W_i} \cap \lambda^z((0,1))\]
such that 
\begin{equation}\label{E:P:bracket-continuous-4}
\lambda_i^z(0) \in [\lambda_i^w(0), x_i], \quad \lambda_i^z(1) \in [x_i, \lambda_i^w(1)], \quad 
\lambda_i^z((0,1)) \subset W_i,
\end{equation}
and $\lambda_i^z([0,1])$ divides $\ol{W_i}$ into two components, one of which contains $x_i$ and the other one contains 
$\lambda^w_i([0,1])$. 

We claim that for every $i\geq i_1$, the Jordan domain $W_i$ enjoys the following property: 
\begin{equation}\label{E:P:bracket-continuous-8}
\text{if } w \in \lambda^z_i((0,1)) \cap X, \quad \text{ then } \quad [b(w), w] \cap \lambda^w_i((0,1)) \neq \emptyset. 
\end{equation} 
To see this, fix an arbitrary $w \in \lambda^z_i((0,1)) \cap X$. By \refE{E:P:bracket-continuous-4}, $w\in W_i$,  
but $b(w) \notin \ol{W_i}$. Therefore, the Jordan arc $[w, b(w)]$ must cross 
$\partial W_i= \lambda_i^w((0,1)) \cup [\lambda_i^w(0), \lambda_i^w(1)]$.
However, $[w, b(w)] \cap [\lambda_i^w(0), \lambda_i^w(1)]=\emptyset$. 
That is because, $[b(w), w] \subset b^{-1}(b(w))$ and $[\lambda_i^w(0), \lambda_i^w(1)] \subset \gamma_i$, while 
$b^{-1}(b(w))$ and $\gamma_i= b^{-1}(b(x_i)$ are either disjoint or identical. 
But, $w$ belongs to $b^{-1}(b(w))$, and by \refE{E:P:bracket-continuous-3}), $w \notin \gamma_i$ (since $w\in W_i$). 
This completes the proof of the claim. 

Since all the points $\lambda_i^w(0)$, $\lambda_i^w(1)$, $\lambda_i^z(0)$, and $\lambda_i^z(1)$ 
lie on $\gamma_i$, we may compare them using the linear order on $\gamma_i$. We have 
\begin{equation}\label{E:P:bracket-continuous-0}
b(x_i) <\lambda_i^w(0) < \lambda_i^z(0) < x_i <  \lambda_i^z(1) < \lambda_i^w(1) < t_i.
\end{equation}
Since $\lambda_i^w(0)$ and $\lambda_i^w(1)$ belong to $\lambda^w((0,1)) \cap \gamma_i$, we must have 
\begin{equation}\label{E:P:bracket-continuous-5}
\lim_{i \to \infty} \lambda_i^w(0)= \lim_{i \to \infty} \lambda_i^w(1)= w.
\end{equation}
Similarly, as $\lambda_i^z(0)$ and $\lambda_i^z(1)$ belong to $\lambda^z((0,1)) \cap \gamma_i$, we must have 
\begin{equation}\label{E:P:bracket-continuous-6}
\lim_{i \to \infty} \lambda_i^z(0)= \lim_{i \to \infty} \lambda_i^z(1)= z
\end{equation}
In particular, we have 
\begin{equation}\label{E:P:bracket-continuous-1}
\lim_{i \to \infty} \diams(\lambda_i^z([0,1])) = 0.
\end{equation}
Moreover, since $x_i \in [\lambda_i^z(0), \lambda_i^z(1)]$, 
\begin{equation}\label{E:P:bracket-continuous-2}
\liminf_{i \to \infty} \diams([\lambda_i^z(0), \lambda_i^z(1)]) \geq \ds(z, x) > 0.
\end{equation}

Now note that at least one of the following three possibilities occurs:
\begin{itemize}
\item[(a)] there are infinitely many $i$ with $\lambda^z_i ([0,1]) \cap X = \emptyset$, 
\item[(b)] there are infinitely many $i$ with $\{\lambda^z_i (0), \lambda^z_i(1)\} \cap X \neq \emptyset$, 
\item[(c)] there are infinitely many $i$ with $\lambda^z_i((0,1)) \cap X \neq \emptyset$.  
\end{itemize}
Below we show that each of the above scenarios leads to a contradiction. 

If (a) occurs, let $(i_k)_{k\geq 1}$ be an increasing sequence such that $\lambda^z_{i_k} ([0,1]) \cap X = \emptyset$. 
In particular, the points $\lambda_{i_k}^z(0)$ and $\lambda_{i_k}^z(1)$ do not belong to $X$. 
Then, the Jordan arc $[\lambda^z_{i_k}(0), \lambda^z_{i_k}(1)]$ does not meet $X$, and hence, it is a hyperbolic 
geodesic in $U$. 
Therefore, according to the Gehring-Hayman theorem discussed before the proof, for $k\geq 1$, we must have 
\[\diams([\lambda^z_{i_k}(0), \lambda^z_{i_k}(1)]) \leq C_{GH} \diams(\lambda_{i_k}^z([0,1])).\]
This contradicts the limiting behaviours in Equations \eqref{E:P:bracket-continuous-1} 
and \eqref{E:P:bracket-continuous-2}. 

If (b) occurs, let $(i_k)_{k\geq 1}$ be an increasing sequence with 
$\{\lambda^z_{i_k} (0),\lambda^z_{i_k}(1)\} \cap X \neq \emptyset$. 
By \refE{E:P:bracket-continuous-0}, we must have $\lambda_{i_k}^z(0) \in X$. 
Therefore, by property (v) in the definition of hairy Cantor sets, and \refE{E:P:bracket-continuous-6},
 $[\lambda_{i_k}^z(0), b(x_{i_k})]$ converges to $[z, b(x)]$. 
On the other hand, by \refE{E:P:bracket-continuous-0}, we have $\lambda_{i_k}^w(0) \in [\lambda_{i_k}^z(0), b(x_i)]$, 
and by \refE{E:P:bracket-continuous-5}, $\lambda_{i_k}^w(0) \to w$, But $w\notin [z, b(x)]$, which is a contradiction. 

If (c) occurs, let $(i_k)_{k\geq 1}$ be an increasing sequence with $\lambda^z_{i_k}((0,1)) \cap X \neq \emptyset$.
Let $z_{i_k} \in \lambda^z_{i_k}((0,1)) \cap X$. 
By \refE{E:P:bracket-continuous-8}, $[z_{i_k}, b(z_{i_k})] \cap \lambda^w_{i_k}((0,1))\neq \emptyset$.  
Let $w_{i_k} \in [z_{i_k}, b(z_{i_k})] \cap \lambda^w_{i_k}((0,1))$. 
As $w_{i_k} \in \lambda^w_{i_k}((0,1))$, from \refE{E:P:bracket-continuous-5}, we conclude that $w_{i_k} \to w$ 
as $k\to \infty$. 
On the other hand, since $z_{i_k} \in X$, by property (v) in the definition of hairy Cantor sets, we have 
$[z_{i_k}, b(z_{i_k})] \to [z, b(z)]$.
But, $w_{i_k} \in [z_{i_k}, b(z_{i_k})]$ and $w_{i_k} \to w \notin [z, b(z)]$. This is a contradiction.
\end{proof}

\begin{proof}[Proof of \refT{T:uniformization}]
By virtue of \refT{T:straight-hairy-homeomorphic}, it is sufficient to prove that every hairy Cantor set in 
the plane is ambiently homeomorphic to a straight hairy Cantor set in the plane.
Let $X$ be an arbitrary hairy Cantor set, with base Cantor set $B$, that is, $B$ is the closure of the set of point 
components of $X$. 
By \refP{P:every_hcs_has_base}, $X$ admits a base curve, that is, a Jordan curve $\go$ such that $X \cap \go= B$ and the 
bounded component of $\D{R}^2 \setminus \go$ does not meet $X$. 
Then we define the sets in \refE{E:S:fixed-notations}. 

Let $A_2$ denote the closed annulus defined by \refE{E:A_R}.
The base map $b: X \to B$ extends to the map $b:A_2 \to \go$; see \refE{E:base-map-extended}.
By \refP{P:b-is-continuous}, $b:A_2 \to \go$ is continuous. 
On the other hand, following the discussions in \refS{S:height-base}, we may consider a Whitney map $\mu$ 
for the set $A_2$. Then, we define 
\[h:A_2 \to [0, \infty), \quad h(z) = \mu([z, b(z)]). \]
By virtue of \refP{P:bracket-continuous} and properties of Whitney maps, $h:A_2 \to [0, \infty)$ is continuous. 
Moreover, if $b(z) = b(w)$ and $h(z) = h(w)$ for some $z$ and $w$ in $A_2$, by the properties of height functions, 
we have $[z, b(z)] = [w, b(z)]$, which implies $z = w$.

Let $I \subseteq \go$ be a Jordan arc such that $b(X)=B \subset I$. 
Consider a homeomorphism 
\[u: I \to \big\{ (x,y) \in \D{R}^2 \mid y=0, x \in [0,1]\big\}.\]
Then, we may define the map 
\[\Psi : \{z \in A_2 \mid b(z) \in I \} \to \{ (x,y) \in \D{R}^2 \mid y \geq 0, x \in [0,1]\}\]
according to  
\[\Psi(z) = (u(b(z)), h(z)).\]
By the above paragraph, $\Psi$ is continuous and injective. 
As $\{ z \in A_2 \mid b(z) \in I \}$ is a compact set, $\Psi$ must be a homeomorphism from its domain onto its image. 
Since the domain of $\Psi$ is a Jordan curve, its image must be also a Jordan curve.
Therefore, one may extend $\Psi$ to a homeomorphism of the plane.

There remains to show that $\Psi(X)$ is a straight hairy Cantor set. 
To prove that, we employ \refL{L:straight_HCS_is_HCS}. 
Let $C= u(B)$. Since $u$ is a homeomorphism, and $B$ is a Cantor set of points, $C$ must be a Cantor set of points. 
Define $l: C \to [0, \infty)$ according to 
\[l(x)= \max \{y \in \D{R} \mid (x, y) \in \Psi(X)\}.\]
As $\Psi(X)$ is compact, $l(x)$ is well-defined and non-negative for all $x \in C$. 
We have 
\[\Psi(X)= \{(x,y) \in \D{R}^2 \mid x \in C, 0 \leq y \leq l(x)\}.\]

We need to verify the three conditions in \refL{L:straight_HCS_is_HCS}.
Evidently, $\Psi(X)$ is compact, so we have item (i). 
By property (iii) in the definition of hairy Cantor sets, $X\setminus B$ is dense in $X$. In particular, the closure of 
$X\setminus B$ contains $B$. This implies that the closure of $\Psi(X\setminus B)$ contains $\Psi(B)=C$.
In particular, we have item (ii) in \refL{L:straight_HCS_is_HCS}. 
Property (iv) in the definition of hairy Cantor sets directly implies item (iii) in \refL{L:straight_HCS_is_HCS}.
Therefore, by \refL{L:straight_HCS_is_HCS}, $Y$ is a straight hairy Cantor set. 
\end{proof}

\begin{proof}[Proof of \refT{T:uniformization-abstract}]
Fix a compact metric space $X$ which satisfies axioms $A_1$ to $A_4$, and $A_6'$. Let $B$ denote the base Cantor set of $X$, 
and $b : X \to B$ its base map.

By \refP{P:X-admits-height}, there is a height function $h:X \to [0, \infty)$, which is continuous on $X$, injective on 
any connected component of $X$, and $h^{-1}(0)=B$. 

Let $C \subset \D{R}$ be a Cantor set, and let $g :B \to C$ be an arbitrary homeomorphism.
Define the map $\Phi : X \to C \times [0, \infty)$ as 
\[\Phi(x) = (g(b(x)), h(x)).\]
The map $\Phi$ is injective. That is because, 
if $g(b(x_1)) = g(b(x_2))$ and $h(x_1) = h(x_2)$, for some $x_1$ and $x_2$ in $X$, then $x_1$ and $x_2$ belong to 
the same connected component of $X$, and we have $[x_1, b(x_1)] = [x_2, b(x_2)]$. This implies that $x_1 = x_2$.

The map $\Phi$ is continuous and injective on a compact set.
Therefore, it is a homeomorphism onto its image. We claim that $\Phi(X)$ is homeomorphic to a straight hairy Cantor set. 
To show that, we apply \refP{P:straight-AHCS-is-homeo-SHCS} to the set $Z=\Phi(X)$. 

Consider the function $l : C \to [0, \infty)$ such that 
\[\Phi(X) = \{ (x, y) \mid x \in C, 0 \leq y \leq l(x) \}.\]
As $\Phi$ is a homeomorphism, and $X$ is compact, $\Phi(X)$ must be compact. 

Define the sets 
\[E_b = \{(x, 0) \mid x \in C, l(x) \neq 0\} ,\quad E_p = \{(x, l(x)) \mid x \in C, l(x) \neq 0\}.\]
By axiom $A_6'$, $E_b \cup E_p$ is dense in $\Phi(X)$.
For any $\gep>0$, 
\[\ol{E_b} \cap \{(x, y) \mid x \in C, \gep \leq y \leq l(x)\} = \emptyset,\]
which implies that 
\[ \ol{E_p} \supseteq \{(x, y) \mid x \in C, \gep \leq y \leq l(x)\}. \]
Therefore, 
\[\ol{E_p} \supseteq \bigcup_{\gep > 0}\{(x, y) \mid x \in C, \gep \leq y \leq l(x)\} = \{(x, y) \mid x \in C, 0 < y \leq l(x)\}, \]
which implies that  
\[ \ol{E_p} \supseteq \{(x, y) \mid x \in C, 0 \leq y \leq l(x), l(x) \neq 0\} \supseteq E_b. \]
Therefore, $E_p$ is dense in $X$.
This completes the proof of the theorem. 
\end{proof}

Axiom $A_6'$ is stronger than axiom $A_6$. But it turns out that any compact set $X \subset \D{R}^2$ which 
satisfies axioms $A_1$ to $A_6$, also satisfies axiom $A_6'$. 
We state the following corollary for future application. 

\begin{cor}
For any hairy Cantor set $X \subset \D{R}^2$ the set of peak points of $X$ is dense in $X$.  
\end{cor}

\begin{proof}
By \refT{T:uniformization}, $X$ is ambiently homeomorphic to a straight hairy Cantor set. 
By \refP{P:dense-ends}, for every straight hairy Cantor set, the set of peak points is dense in that straight hairy Cantor set. 
Therefore, the set of peak points of $X$ must be dense in $X$.
\end{proof}

In \cite{DeRo14}, the authors propose a topological model for the attractor (post-critical set) of a especial class of infinitely satellite 
renormalisable maps. The main object in that paper is built as the intersection of a nest of plain domains. However, they do not study 
the topological features of the model presented in the paper. One may see that the topological model presented in that paper is a Hairy 
Cantor set. 

\subsection*{Acknowledgements.} The first author acknowledges funding from EPSRC(UK) -
grant No. EP/M01746X/1 - rigidity and small divisors in holomorphic dynamics.
The second author acknowledges funding from ERC advanced grant No 339523 -- Rigidity and global deformations in dynamics.

\bibliographystyle{amsart}
%\bibliography{/Users/dcheragh/Work/Inscriptions/Data}
\bibliography{Data}
\end{document}